\documentclass[a4paper,11pt]{article}
\usepackage[latin1]{inputenc}
\usepackage[hmargin={25mm,25mm},vmargin={25mm,30mm}]{geometry}
\usepackage{amsmath,amsfonts,amsthm,amssymb,newtxtext,newtxmath}
\usepackage{authblk}
\usepackage{cancel}
\usepackage{comment}
\usepackage{graphicx}
\usepackage[ocgcolorlinks,allcolors={blue}]{hyperref}
\usepackage{xcolor}
\usepackage{pgfplots}
\DeclareMathAlphabet{\mathpgoth}{OT1}{pgoth}{m}{n}
\usepackage{multirow}
\usepackage{booktabs}
\usepackage[compact]{titlesec}
\usepackage{enumitem}
\usepackage{pgfplotstable}
\usepackage{tikz}
\usetikzlibrary{spy}

\usepackage[font={footnotesize}]{subcaption}
\usepackage[absolute,showboxes,overlay]{textpos} 
\textblockorigin{0pt}{0pt}   
\definecolor{darkbrown}{HTML}{996633}
\newcommand{\logLogSlopeTriangle}[5]
{
    \pgfplotsextra
    {
        \pgfkeysgetvalue{/pgfplots/xmin}{\xmin}
        \pgfkeysgetvalue{/pgfplots/xmax}{\xmax}
        \pgfkeysgetvalue{/pgfplots/ymin}{\ymin}
        \pgfkeysgetvalue{/pgfplots/ymax}{\ymax}

        \pgfmathsetmacro{\xArel}{#1}
        \pgfmathsetmacro{\yArel}{#3}
        \pgfmathsetmacro{\xBrel}{#1-#2}
        \pgfmathsetmacro{\yBrel}{\yArel}
        \pgfmathsetmacro{\xCrel}{\xArel}
        \pgfmathsetmacro{\lnxB}{\xmin*(1-(#1-#2))+\xmax*(#1-#2)} % in [xmin,xmax].
        \pgfmathsetmacro{\lnxA}{\xmin*(1-#1)+\xmax*#1} % in [xmin,xmax].
        \pgfmathsetmacro{\lnyA}{\ymin*(1-#3)+\ymax*#3} % in [ymin,ymax].
        \pgfmathsetmacro{\lnyC}{\lnyA+#4*(\lnxA-\lnxB)}
        \pgfmathsetmacro{\yCrel}{\lnyC-\ymin)/(\ymax-\ymin)} % THE IMPROVED EXPRESSION WITHOUT 'DIMENSION TOO LARGE' ERROR.

        \coordinate (A) at (rel axis cs:\xArel,\yArel);
        \coordinate (B) at (rel axis cs:\xBrel,\yBrel);
        \coordinate (C) at (rel axis cs:\xCrel,\yCrel);

        \draw[black]   (A)-- node[pos=0.5,anchor=north] {\scriptsize{1}}
                    (B)-- 
                    (C)-- node[pos=0.,anchor=west] {\scriptsize{\color{#5}#4}} %% node[pos=0.5,anchor=west] {#4}
                    (A);
    }
}

\usepackage{stackengine}

\usepackage[style=numeric-comp,maxbibnames=99]{biblatex}

\bibliography{pnshho}
\allowdisplaybreaks

\newtheorem{theorem}{Theorem}

\newtheorem{lemma}[theorem]{Lemma}

\theoremstyle{remark}
\newtheorem{remark}[theorem]{Remark}
\theoremstyle{definition}
\newtheorem{assumption}{Assumption}

\newtheorem{example}[theorem]{Example}

\newcommand{\email}[1]{\href{mailto:#1}{#1}}

\def\b{\boldsymbol}
\newcommand*\cst[1]{\mathrm{#1}}

\newcommand{\R}{\mathbb{R}} %reals
\newcommand{\N}{\mathbb{N}} %naturals
\newcommand{\Poly}{\mathbb{P}} %polynomials
\newcommand{\M}[1]{\R^{#1 \times #1}} %square matrices
\newcommand{\Ms}[1]{\R^{#1 \times #1}_\cst{s}} %symmetric matrices
\newcommand\bdU[2]{{\bu U}_{#1}^{#2}} %discrete velocity space
\newcommand\dP[2]{P_{#1}^{#2}} %discrete pressure space

\def\und{\underline} %discrete scalar function
\newcommand\bu[1]{\und{\b{#1}}} %discrete vector or matrix function
\newcommand{\T}{\mathcal{T}} %elements
\newcommand{\F}{\mathcal{F}} %faces
\newcommand{\Fb}{\F_h^\cst{b}} %boundary faces
\newcommand{\Fi}{\F_h^\cst{i}} %interfaces

\newcommand{\res}[1]{\ \!\!_{|_{#1}}} %restriction
\newcommand{\dgrad}[2]{\b{\cst{G}}^{#1}_{#2}} %discrete gradient
\newcommand{\dgrads}[2]{\b{\cst{G}}^{#1}_{\cst{s},#2}} %discrete symmetric gradient
\newcommand{\bdrec}[2]{\b{\cst{r}}^{#1}_{#2}} %discrete reconstruction operator
\newcommand{\bdfbres}[2]{\b{\Delta}^{#1}_{#2}} %discrete face-based residual operator
\newcommand{\ddiv}[2]{\cst{D}^{#1}_{#2}} %discrete divergence
\newcommand{\GRAD}{\b\nabla} %gradient
\newcommand{\brkGRAD}{\b\nabla_h} %broken gradient
\newcommand{\GRADs}{\b{\nabla}_\cst{s}} %symmetric gradient
\newcommand{\brkGRADs}{\b{\nabla}_{\cst{s},h}} %broken symmetric gradient
\renewcommand{\div}{\b\nabla{\cdot}} %divergence for vector
\newcommand{\DIV}{\b\nabla{\cdot}} %divergence for matrices
\newcommand{\brkDIV}{\b\nabla_h \cdot} %broken divergence for matrices
\newcommand\bI[2]{\bu{I}_{#1}^{#2}} %interpolator
\newcommand\proj[2]{\pi_{#1}^{#2}} %scalar projection
\newcommand\PROJ[2]{\b{\pi}_{#1}^{#2}} %projection for vectors or matrices
\newcommand\stress{\b\sigma} %stress-strain function
\newcommand\convec{\b\chi}

\newcommand*\CV[2]{\underset{#1\hspace{0.05cm} \to\hspace{0.05cm} #2}{\,\smash{\mathop{
\xrightarrow{\hspace*{0.8cm}}}}\ }}

\newcommand\sob{r}
\newcommand\sobs{\tilde r}
\newcommand\conv{s}
\newcommand\convs{\tilde s}

\newcommand{\strain}{\boldsymbol{\varepsilon}}

\newcommand*\NAV[2]{(#1 \cdot \GRAD)#2}
\newcommand*\tri[4]{(#1 \cdot #2)#3 \cdot #4}
\newcommand*\btri[4]{\big(#1 \cdot #2\big)#3 \cdot #4}

\title{A Hybrid High-Order method for incompressible flows of non-Newtonian fluids with power-like convective behaviour}	
\author[1]{Daniel Castanon Quiroz \footnote{\email{danielcq.mathematics@gmail.com}}}
\author[2]{Daniele A. Di Pietro \footnote{\email{daniele.di-pietro@umontpellier.fr}}}
\author[2]{Andr\'{e} Harnist \footnote{\email{andre.harnist@umontpellier.fr}, corresponding author}}
\affil[1]{LJAD, Univ C\^ote d'Azur, Inria, CNRS, Nice, France}
\affil[2]{IMAG, Univ Montpellier, CNRS, Montpellier, France}

\begin{document}

\maketitle

\begin{abstract}
  In this work, we design and analyze a Hybrid High-Order (HHO) discretization method for incompressible flows of non-Newtonian fluids with power-like convective behaviour.
  We work under general assumptions on the viscosity and convection laws, that are associated with possibly different Sobolev exponents $\sob\in(1,\infty)$ and $\conv\in(1,\infty)$.
  After providing a novel weak formulation of the continuous problem, we study its well-posedness highlighting how a subtle interplay between the exponents $\sob$ and $\conv$ determines the existence and uniqueness of a solution.
  We next design an HHO scheme based on this weak formulation and perform a comprehensive stability and convergence analysis, including convergence for general data and error estimates for shear-thinning fluids and small data.
  The HHO scheme is validated on a complete panel of model problems.
  \medskip\\
  \textbf{Keywords:} Hybrid High-Order methods, non-Newtonian fluids, Navier--Stokes, power-law, Carreau--Yasuda law, power-like convective behaviour, lid-driven cavity problem
  \smallskip\\
  \textbf{MSC2010 classification:} 75A05, 76D05, 65N30, 65N08, 65N12
\end{abstract}

%% \tableofcontents

%------------------------------------------------------------------------------%
%------------------------------------------------------------------------------%

\section{Introduction}\label{sec:introduction}

In this paper, building on the results of \cite{Botti.Castanon-Quiroz.ea:20}, we design and analyze a Hybrid High-Order (HHO) method for incompressible flows of non-Newtonian fluids governed by the generalized Navier--Stokes equations.
The formulation considered here encompasses general viscosity and convection laws, possibly corresponding to different Sobolev exponents.
  The proposed numerical method aims at overcoming certain limitations of traditional (e.g., finite element or finite volume) schemes, particularly concerning the supported meshes and approximation orders.
  Notice, however, that the techniques introduced in this work can, in principle, be applied also to more classical discretizations.
 
Nonlinear rheologies are encountered in several fields, including ice sheet dynamics and glacier modelling \cite{Isaac.Stadler.ea:15,ahlkrona2021}, mantle convection \cite{Schubert.Turcotte.ea:01}, chemical engineering \cite{Ko.Pustejovska.ea:18}, and biological fluids \cite{Lai.Kuei.ea:78,Galdi.Rannacher.ea:18}.
Their mathematical study was pioneered in the work of Ladyzhenskaya \cite{Ladyzhenskaya:69}; detailed well-posedness and regularity analyses of the corresponding system of equations have been carried out in \cite{Malek.Rajagopal:05, Ruzicka.Diening:07, Diening.Ettwein:08, Beirao-da-Veiga:09, Berselli.Ruzicka:20}.
Recently, a variation of the classical trilinear convective term has been considered in \cite{li2017peuler} in relation to the computation of the Wasserstein distance for optimal transport applications (to the best of our knowledge, such generalizations haven't yet been considered in the context of fluid-mechanics).
In this work, we propose a further generalization of this model encompassing power-like convective behaviours that satisfy non-dissipativity relations; see Assumption \ref{ass:conv} below.
  The power-law behaviours of the viscous and convective terms are characterized by (possibly different) Sobolev exponents $r\in(1,\infty)$ and $s\in(1,\infty)$.
  In Theorem \ref{thm:well-posedness} below, we carry out a well-posedness analysis for the continuous problem showing that a subtle interplay of these exponents determines the existence and uniqueness of a solution  (in the classical case $s=2$, relevant in fluid mechanics, this translates into constraints on the Sobolev index $r$).
  Such an interplay, which reverberates at the discrete level, is required to leverage the H\"older continuity of the convective function that appears in the weak formulation of the problem.
    Notice that considering a general exponent $s\in(1,\infty)$ involves only (relatively) minor changes in the analysis with respect to the case $s=2$, and has the advantage of making the proposed method suitable for applications in promising fields such as optimal transport.
Also, to the best of our knowledge, both the analysis of this generalized Navier--Stokes problem and its numerical approximation are entirely new.

While a large body of literature deals with the numerical approximation of the Navier--Stokes equations, only a relatively small fraction of these works addresses nonlinear rheologies.
Finite element methods for creeping (Stokes) flows of non-Newtonian fluids have been considered in \cite{Barrett.Liu:94,Hirn:13,Belenki.Berselli.ea:12}.
Non-Newtonian fluid flows with standard convective behaviour have been considered in \cite{crochet1983numerical,crochet2012numerical}.
This model is encountered in several other works: see, e.g., \cite{Diening.Kreuzer.ea:13,Kreuzer.Suli:16} for finite element methods with implicit power-law-like rheologies;
\cite{Ko.Pustejovska.ea:18,Ko.Suli:18} for generalized Newtonian fluids with space variable and concentration-dependent power-law index; 
\cite{kroner2014} for a local discontinuous Galerkin method;
\cite{doi:10.1080/01932691.2014.896221} for a study of non-Newtonian polymer flows through porous media;
\cite{janecka2019} for simulations of transient flows of non-Newtonian fluids;
\cite{azoug2020} for a finite element approximation of non-Newtonian polymer aqueous solutions with fractional time-derivative.

  Recent works have emphasized the importance of handling polytopal meshes in the context of numerical fluid mechanics; see, e.g., the introduction of \cite{Di-Pietro.Droniou:20} for a broad discussion on this subject.
  General meshes, possibly combined with high order, can be used in this context to:
  adapt the shape of the elements to the local features of the flow, thus improving the resolution of boundary or internal layers;
  perform non-conforming local mesh refinement, which naturally preserves the mesh quality;
  reduce the computational cost while preserving the approximation of the domain geometry by mesh coarsening \cite{Bassi.Botti.ea:12,Bassi.Botti.ea:14}.
  Among recent methods for incompressible flows that natively support general polytopal meshes and arbitrary-order, we can cite
  discontinuous Galerkin methods \cite{Di-Pietro.Ern:10} (see also \cite[Chapter 6]{Di-Pietro.Ern:12}), 
  Virtual Element methods  \cite{Beirao-da-Veiga.Lovadina.ea:18,Gatica.Munar.ea:18,Liu.Chen:19,Irisarri.Hauke:19} (see also \cite{Veiga.Brezzi:12}),
  and HHO methods \cite{Di-Pietro.Krell:18,Botti.Di-Pietro.ea:19*1,Castanon-Quiroz.Di-Pietro:20,Botti.Castanon-Quiroz.ea:20} (see also \cite[Chapter 9]{Di-Pietro.Droniou:20}).
    While discontinuous Galerkin methods are commonly considered the golden standard in fluid mechanics, recent works  have pointed out the potential of HHO methods in terms of overall efficiency and precision for real-life problems \cite{Botti.Di-Pietro:21}.
    The developments of the present work show their ability to tackle complex, highly nonlinear physics.%
  
Specifically, in this paper we extend the HHO method of \cite{Botti.Castanon-Quiroz.ea:20} to the full generalized Navier--Stokes equations with power-law viscous and convective behaviours.
The discretization of the convective term relies on the novel formulation devised at the continuous level, which guarantees non-dissipativity at the discrete level and is obtained replacing the continuous gradient by the classical HHO gradient reconstruction in full polynomial spaces; see \cite[Eq. (4.3)]{Di-Pietro.Droniou:17} for the scalar case.
For this novel discrete convective function we prove the key properties that intervene in the stability and consistency of the method.
The former consist, in addition to non-dissipativity, in a H\"older continuity property expressed in terms of a discrete $W^{1,\sob}$-(semi)norm.
The latter include consistency for smooth functions and sequential consistency.
We perform complete stability and convergence analyses, highlighting the interplay between the exponents $\sob$ and $\conv$.
Specifically, existence and uniqueness of a discrete solution, established in Theorem \ref{thm:discrete.well-posedness}, hold under the same conditions on $\sob$ and $\conv$ as for the continuous problem and a data smallness assumption where the constants of relevant continuous inequalities are replaced by their discrete counterparts.
We then establish, in Theorem \ref{thm:convergence}, various convergence results under minimal regularity assumptions on the solutions using compactness arguments.
Finally, an error estimate for shear-thinning fluids displaying different orders of convergence according to the degeneracy of the problem in the spirit of \cite{Di-Pietro.Droniou.ea:21} is proved in Theorem \ref{thm:error.estimate}.
When polynomials of degree $k\ge 1$ are used, denoting by $h$ the meshsize, the error estimates give orders of convergence ranging from $h^{(k+1)(\sob-1)}$ to $h^{k+1}$ for the velocity and from $h^{(k+1)(\sob-1)^2}$ to $h^{(k+1)(\sob-1)}$ for the pressure, depending on the degeneracy of the problem.

The rest of the paper is organized as follows.
In Section \ref{sec:continuous.setting} we introduce the strong and weak formulations of the generalized Navier--Stokes problem, discuss the assumptions on the viscosity and convection laws, and study existence and uniqueness of a weak solution.
The construction of the HHO discretization is carried out in Section \ref{sec:discrete.setting} by defining the discrete counterparts of the viscous, convective, and coupling terms.
In Section \ref{sec:discrete.main.results}, we formulate the discrete problem and state the main stability and convergence results for the method.
In Section \ref{sec:num.res}, we investigate the performance of the method on a complete set of model problems.
Finally, Section \ref{sec:main.results.properties} collects the proofs of the consistency properties of the discrete convective function and of the main results.

%------------------------------------------------------------------------------%
%------------------------------------------------------------------------------%

\section{Continuous setting}\label{sec:continuous.setting}

Let $\Omega \subset \R^d$, $d\in\{2,3\}$, denote a bounded, connected, polyhedral open set with Lipschitz boundary $\partial\Omega$. We consider the incompressible flow of a fluid occupying $\Omega$ and subjected to a volumetric force field $\b f : \Omega \to \R^d$, governed by the following generalized Navier--Stokes problem: Find the velocity field $\b u : \Omega \to \R^d$ and the pressure field $p : \Omega \to \R$ such that
\begin{subequations}\label{eq:ns.continuous}
  \begin{alignat}{2} 
    -\DIV\stress(\cdot,\GRADs \b u) + \NAV{\b u}{\convec(\cdot,\b u)} + \GRAD p &= \b f &\qquad& \mbox{  in  } \Omega, \label{eq:ns.continuous:momentum} \\
    \div \b u &= 0 &\qquad& \mbox{ in }  \Omega, \label{eq:ns.continuous:mass} \\
    \b u &= \b 0 &\qquad& \mbox{ on } \partial \Omega, \label{eq:ns.continuous:bc} \\
    \int_\Omega p &= 0, \label{eq:ns.continuous:closure}
  \end{alignat}
\end{subequations} 
where $\DIV$ denotes the divergence operator applied to tensor-valued or vector-valued fields, $\GRADs$ is the symmetric part of the gradient operator $\GRAD$ applied to vector fields, and, denoting by $\Ms{d}$ the set of square, symmetric, real-valued $d\times d$ matrices, $\stress : \Omega \times \Ms{d} \to \Ms{d}$ is the viscosity law and $\convec : \Omega \times \R^d \to \R^d$ is the convection law.
In what follows, we formulate assumptions on $\stress$ and $\convec$ that encompass common models for non-Newtonian fluids and state a weak formulation for problem \eqref{eq:ns.continuous} that will be used as a starting point for its discretization.

\subsection{Viscosity law}\label{sec:strain.stress.law}

For all $m \in [1,\infty]$, we define the \emph{conjugate}, \emph{singular}, and \emph{Sobolev} exponents of $m$ by
\begin{equation}\label{eq:exps}
\!\!   m' \coloneqq \begin{cases} \frac{m}{m-1} &\! \text{if } m \in (1,\infty) \\ \infty &\! \text{if } m = 1 \\ 1 &\! \text{if } m = \infty \end{cases} \in [1,\infty],\quad
    \tilde m \coloneq \min(m,2) \in [1,2],\quad
    m^* \coloneqq \begin{cases} \frac{dm}{d-m} &\! \text{if } m < d \\ \infty &\! \text{if } m \geq d \end{cases} \in [m,\infty].
\end{equation}
For all $\b\tau= (\tau_{ij})_{1 \le i,j \le d}$ and $\b\eta= (\eta_{ij})_{1 \le i,j \le d}$ in $\M{d}$, we also define the Frobenius inner product $\b\tau : \b\eta \coloneqq \sum_{i,j=1}^d \tau_{ij}\eta_{ij}$ and the corresponding norm $|\b\tau|_{d \times d}\coloneqq \sqrt{\b\tau : \b\tau}$.

\begin{assumption}[Viscosity law]\label{ass:stress}
  Let a real number $\sob \in (1,\infty)$ be fixed. The viscosity law satisfies
  \begin{gather}
    \stress : \Omega \times \Ms{d} \to \Ms{d} \text{ is measurable},\label{eq:ass:stress:meas} \\
    \stress(\b x,\b 0) \in L^{\sob'}(\Omega,\Ms{d}) \text{ for almost every } \b x \in \Omega.\label{eq:ass:stress:0}
  \end{gather}
  Moreover, there exist real numbers $\delta \in [0,\infty)$ and $\sigma_\cst{hc},\sigma_\cst{hm} \in (0,\infty)$ such that, for all $\b\tau,\b\eta \in \Ms{d}$ and almost every $\b x \in \Omega$, the following H\"older continuity and H\"older monotonicity properties hold
    \begin{align} 
      \left|
      \stress(\b x,\b\tau)-\stress(\b x,\b\eta)
      \right|_{d\times d} &\le \sigma_\cst{hc} \left(\delta^\sob+|\b\tau|_{d\times d}^\sob+|\b\eta|_{d\times d}^\sob\right)^\frac{\sob-2}{\sob}| \b\tau-\b\eta |_{d\times d},\label{eq:ass:stress:holder.cont} \\
      \left(\stress(\b x,\b\tau)-\stress(\b x,\b\eta)\right):(\b\tau-\b\eta)  &\ge \sigma_\cst{hm}\left(\delta^\sob+|\b\tau|_{d\times d}^\sob+|\b\eta|_{d\times d}^\sob\right)^\frac{\sob-2}{\sob}|\b\tau-\b\eta|_{d\times d}^{2}.\label{eq:ass:stress:strong.mono}
    \end{align}  
\end{assumption}

\begin{remark}[Degeneracy function]\label{rem:stress:residual}
  The parameter $\delta$ in Assumption \ref{ass:stress} is related to the degeneracy of the flux function when $r<2$.
    Another possible choice, considered in \cite{Di-Pietro.Droniou.ea:21}, consists in taking $\delta \in  L^r(\Omega,[0,\infty))$.
    This variation requires only minor changes in the analysis, not detailed here for the sake of conciseness.
\end{remark}

\begin{example}[Carreau--Yasuda stress]\label{ex:Carreau--Yasuda}
  An example of viscosity law  satisfying Assumption \ref{ass:stress} is the $(\mu,\delta,a,\sob)$-Carreau--Yasuda law such that, for almost every $\b x\in\Omega$ and all $\b\tau \in \Ms{d}$,
  \begin{equation}\label{eq:Carreau--Yasuda}
    \stress(\b x,\b\tau) = \mu(\b x)\left(\delta^{a(\b x)}+|\b\tau|_{d \times d}^{a(\b x)}\right)^\frac{\sob-2}{a(\b x)}\b\tau,
  \end{equation}
  where $\mu : \Omega \to [\mu_-,\mu_+]$ and $a : \Omega \to [a_-,a_+]$ are measurable functions with $\mu_\pm,a_\pm \in (0,\infty)$, $\delta \in [0,\infty)$, and $\sob \in (1,\infty)$. Notice that the case $\delta = 0$ corresponds to classical power-law fluids. See \cite[Example 4]{Botti.Castanon-Quiroz.ea:20} for a proof of the fact that this law matches Assumption \ref{ass:stress} and also \cite[Example 6]{Di-Pietro.Droniou.ea:21} for the generalization of $\delta$ to a function.
\end{example}

   \begin{remark}[Traceless-stable assumption]
     Although we consider incompressible flows (which are characterized by traceless strain rate fields), our analysis does not require that $\text{tr}\,\stress(\b x,\b\tau) = \b 0$ for almost every $\b x\in\Omega$ and all $\b\tau \in \Ms{d}$ such that $\text{tr}\,\b\tau = 0$.
     In practice, however, this property holds for generalized Newtonian fluids, for which there exists a scalar function $\nu : \Omega \times \Ms{d} \to \R$ such that $\stress(\b x,\b\tau) = \nu(\b x,\b\tau)\b\tau$  for almost every $\b x\in\Omega$ and all $\b\tau \in \Ms{d}$; cf. Example \ref{ex:Carreau--Yasuda}.
\end{remark}

\subsection{Convection law}\label{sec:convective.law}

In what follows, $|{\cdot}|$ will denote both the absolute value of scalars and Euclidian norm of vectors, while $\otimes$ denotes the tensor product of two vectors such that, for all $\b x=(x_i)_{1\le i\le d}\in\R^d$ and $\b y=(y_j)_{1\le j\le d}\in\R^d$, $\b x\otimes\b y \coloneq (x_iy_j)_{1\le i,j\le d}\in\R^{d\times d}$.
\begin{assumption}[Convection law]\label{ass:conv}
  Let a real number $\conv \in (1,\infty)$ be fixed. The convection law satisfies
  \begin{subequations}\label{eq:ass:conv}
    \begin{gather}
      \convec : \Omega \times \R^d \to \R^d \text{ is measurable},\label{eq:ass:conv:meas} \\
      \convec(\b x,\b 0) = \b 0 \text{ for almost every } \b x \in \Omega.\label{eq:ass:conv:0}
    \end{gather}
    We also assume that, for all $\b w \in \R^d$, the non-dissipativity relations hold:
    \begin{align}
      \NAV{\b w}{\convec(\cdot,\b w)} &= \NAV{\convec(\cdot,\b w)}{\b w}+(s-2)\frac{\tri{\convec(\cdot,\b w)}{\GRAD}{\b w}{\b w}}{|\b w|^2}\b w,\label{eq:ass:conv:non.dissip:1}\\
      \b w \otimes \convec(\cdot,\b w) &= \convec(\cdot,\b w) \otimes \b w.\label{eq:ass:conv:non.dissip:2}
    \end{align}
    Moreover, there exists a real number  $\chi_\cst{hc} \in (0,\infty)$ such that, for all $\b v,\b w \in \R^d$ and almost every $\b x \in \Omega$, the following H\"older continuity property holds:
    \begin{equation}\label{eq:ass:conv:holder.cont}
      |\convec(\b x,\b w)-\convec(\b x,\b v)| \le \chi_\cst{hc}\left(|\b w|^\conv+|\b v|^\conv\right)^\frac{\conv-\convs}{\conv}|\b w - \b v|^{\convs-1}.
    \end{equation}
  \end{subequations}
\end{assumption}

\begin{example}[Standard convection law]
  The standard convection law is obtained taking $s=2$ and $\convec(\cdot,\b w) \equiv \b w$.
    According to Theorem \ref{thm:convergence} below, with this choice we can prove convergence provided that $r\in\big(\frac32,\infty\big)$ if $d=2$ and $r\in\big(\frac95,\infty\big)$ if $d=3$.
    Error estimates stated by Theorem \ref{thm:error.estimate}, on the other hand, also require $r\le 2$, reducing the above intervals to $r\in\big[\frac32,2\big]$ if $d=2$ and $r\in\big[\frac95,2\big]$ if $d=3$; note that the additional regularity assumption allows the left limit points.
\end{example}

\begin{example}[$(\nu,s)$-Laplace convection law] \label{ex:Laplace}
  Another example of convection law satisfying Assumption \ref{ass:conv} is the $(\nu,s)$-Laplace law considered in \cite{li2017peuler} in the context of applications to optimal transport and such that, for almost every $\b x\in\Omega$ and all $\b w \in \R^d$, 
  \begin{equation}\label{eq:Laplace}
    \convec(\b x,\b w) = \nu(\b x)|\b w|^{s-2}\b w,
  \end{equation}
  where $\nu : \Omega \to [0,\nu_+]$ is a measurable function with $\nu_+ \in [0,\infty)$ corresponding to the local flow convection index, while $\conv \in (1,\infty)$ is the convection behaviour index.  
    It can be proved as in \cite[Example 4]{Botti.Castanon-Quiroz.ea:20} that $\convec$ is an $s$-power-framed function, which implies \eqref{eq:ass:conv:holder.cont}.
\end{example}

\subsection{Weak formulation}\label{sec:weak.formulation}

The starting point for the HHO discretization of problem \eqref{eq:ns.continuous} is the weak formulation studied in this section.
We define the following velocity and pressure spaces embedding, respectively, the homogeneous boundary condition for the velocity and the zero-average constraint for the pressure: 
\[
\b U \coloneqq \big\{\b v \in W^{1,\sob}(\Omega)^d\ : \ \b v\res{\partial\Omega} = \b 0 \big\},
\qquad
P \coloneqq  \big\{q \in L^{\sob'}(\Omega)\ : \ \textstyle\int_\Omega q = 0 \big\}. 
\]
Assuming $\b f \in L^{\sob'}(\Omega)^d$, the weak formulation of problem \eqref{eq:ns.continuous} reads:
Find $(\b u,p) \in \b U \times P$ such that
\begin{subequations}\label{eq:ns.weak}
  \begin{alignat}{2}
     a(\b u,\b v)+c(\b u,\b v)+b(\b v,p) &= \displaystyle\int_\Omega \b f \cdot \b v &\qquad \forall \b v \in \b U,\label{eq:ns.weak:momentum} \\
     -b(\b u,q) &= 0 &\qquad \forall q \in P,  \label{eq:ns.weak:mass} 
  \end{alignat}
\end{subequations}
where the function $a : \b U \times \b U \to \R$,
the bilinear form $b : \b U \times L^{\sob'}(\Omega) \to \R$,
and the function $c : \b U \times \b U \to \R$ are defined such that, for all $\b v,\b w \in \b U$ and all $q \in L^{\sob'}(\Omega)$,
\begin{gather}\label{eq:weak:ab}
  a(\b w,\b v) \coloneq
  \int_\Omega \stress(\cdot,\GRADs \b w) : \GRADs \b v, \qquad
  b(\b v,q)  \coloneq -\int_\Omega (\GRAD \cdot \b v) q,
  \\ \label{eq:weak:c}
  c(\b w,\b v) \coloneq
  \frac{1}{\conv}\int_\Omega \tri{\convec(\cdot,\b w)}{\GRAD}{\b w}{\b v}
  + \frac{\conv-2}{\conv}\int_\Omega \frac{\b v \cdot \b w}{|\b w|^{2}}\tri{\convec(\cdot,\b w) }{\GRAD}{\b w}{\b w}- \frac{1}{\conv'}\int_\Omega \tri{\convec(\cdot,\b w) }{\GRAD}{\b v}{\b w}.
\end{gather}
In order to obtain the function $c$ in \eqref{eq:weak:c}, we have used \eqref{eq:ass:conv:non.dissip:1}--\eqref{eq:ass:conv:non.dissip:2} as follows:
Denoting by $\b n_{\partial\Omega}$ the unit normal vector pointing out of $\partial\Omega$ and observing that $1 = \frac{1}{s}+\frac{1}{s'}$, for smooth enough functions $\b v,\b w:\Omega\to\R^d$ such that $\DIV \b w = 0$ and $\b w\res{\partial\Omega}= \b 0$ we can write
\begin{equation}\label{eq:c:weak.formulation}
\begin{aligned}
  &\int_\Omega \tri{\b w}{\GRAD}{\convec(\cdot,\b w)}{\b v}
  \\
  &\qquad
  = \frac{1}{s}\int_\Omega \tri{\b w}{\GRAD}{\convec(\cdot,\b w)}{\b v}+\frac{1}{s'}\int_\Omega \tri{\b w}{\GRAD}{\convec(\cdot,\b w)}{\b v}
  \\
  &\qquad
  = \frac{1}{\conv}\left(\int_\Omega \tri{\convec(\cdot,\b w)}{\GRAD}{\b w}{\b v}+(\conv-2)\int_\Omega \frac{\b v \cdot \b w}{|\b w|^{2}}\tri{\convec(\cdot,\b w) }{\GRAD}{\b w}{\b w}\right)\\
  &\qquad
  \quad  -\frac{1}{s'}\left(\int_\Omega \tri{\b w}{\GRAD}{\b v}{\convec(\cdot,\b w)}+\cancel{\int_\Omega (\convec(\cdot,\b w) \cdot \b v)(\DIV \b w)}-\cancel{\int_{\partial\Omega} (\convec(\cdot,\b w) \cdot \b v)(\b w \cdot \b n_{\partial\Omega})}\right),
\end{aligned}
\end{equation}
where the second equality follows from an integration by parts along with \eqref{eq:ass:conv:non.dissip:1},
while the cancellations are a consequence of the assumptions on $\b w$.
Using \eqref{eq:ass:conv:non.dissip:2} on the last non-zero term of \eqref{eq:c:weak.formulation}, we finally get the expression in the right-hand side of \eqref{eq:weak:c}. 
This version of $c$ satisfies, by construction, the following non-dissipativity property:
For all $\b w \in \b U$,
\begin{equation}\label{eq:c:zero}
c(\b w,\b w) = 0.
\end{equation}

We now recall the following Korn inequality (see \cite[Theorem 1]{Geymonat.Suquet:86}): For all $m \in (1,\infty)$, there is $C_{\cst{K},m} \in (0,\infty)$ only depending on $m$, $d$, and $\Omega$ such that, for all $\b v \in\b U$,
\begin{equation}\label{eq:Korn}
\|\b v\|_{W^{1,m}(\Omega)^d} \le C_{\cst{K},m} \|\GRADs \b v\|_{L^m(\Omega)^{d\times d}}.
\end{equation}

\begin{theorem}[Existence and uniqueness for problem \eqref{eq:ns.weak}]\label{thm:well-posedness}
  Under Assumptions \ref{ass:stress} and \ref{ass:conv}, there exists a solution $(\b u,p) \in \b U \times P$ to the weak formulation \eqref{eq:ns.weak}, and any solution satisfies
  \begin{equation}
    \begin{aligned}
      |\b u|_{W^{1,r}(\Omega)^d} &\le C_\cst{v}\left[
        \left(\sigma_\cst{hm}^{-1}\| \b f \|_{L^{\sob'}(\Omega)^d}\right)^{r'}\!\!+\min\left(\delta^r;\left(\delta^{2-\sobs}\sigma_\cst{hm}^{-1}\| \b f \|_{L^{\sob'}(\Omega)^d}\right)^\frac{r}{\sob+1-\sobs}\right)
        \right]^\frac{1}{r}
    \end{aligned} \label{eq:well-posedness:bound:u}
  \end{equation}
  with $C_\cst{v} > 0$ depending only on $\Omega$, $d$, and $r$.
  Moreover, assuming $2 \le s \le \frac{\sobs^*}{\sobs'}$, i.e.,
  \begin{equation}\label{eq:intervals.s:cont.uniqueness}
    s\in\begin{cases}
      \big[2,\frac{d(r-1)}{d-r}\big] & \text{if $d=2$ and $r\in\big[\frac32,2\big)$ or $d=3$ and $r\in\big[\frac95,2\big]$,}
        \\
        [2,3] & \text{if $d=3$ and $r\in(2,3)$},
        \\
        [2,\infty) & \text{if $d=2$ and $r\in[2,\infty)$ or $d=3$ and $r\in[3,\infty)$,}
    \end{cases}
  \end{equation}
  and that the following data smallness condition holds:
  \begin{equation}\label{eq:well-posedness:f}
    \delta^r+\left(\sigma_\cst{hm}^{-1}\| \b f \|_{L^{\sob'}(\Omega)^d}\right)^{r'} < \left(1+2C_\cst{v}^r\right)^{-1}\left(C_{\cst{c},\sobs}^{-1}\chi_\cst{hc}^{-1}C_\cst{a}\sigma_\cst{hm}\delta^{r-\sobs}\right)^{\frac{r}{s+1-\sobs}},
  \end{equation}
  the solution of \eqref{eq:ns.weak} is unique.
\end{theorem}

\begin{remark}[Uniqueness]\label{rem:uniqueness}
  Uniqueness of the solution of \eqref{eq:ns.weak} is not guaranteed for any value of $r$ and $s$, and in particular when $r > 2$ in the degenerate case $\delta = 0$.
  The assumptions on $r$ and $s$ ensuring the uniqueness of the continuous weak solution will carry out to the discrete level, in both the well-posedness result of Theorem \ref{thm:discrete.well-posedness} and the error estimate of Theorem \ref{thm:error.estimate}, where $r\le 2$ is additionally required.
\end{remark}

\begin{proof}[Proof of Theorem \ref{thm:well-posedness}]\label{proof:thm:well-posedness}
  1. \emph{Existence.}
  Replacing the function $a$ by the sum $a+c$ in \cite[Remark 6]{Botti.Castanon-Quiroz.ea:20} and using the non-dissipativity \eqref{eq:c:zero} of the convective function $c$ yields the existence of a solution to the weak formulation \eqref{eq:ns.weak} and the a priori estimate \eqref{eq:well-posedness:bound:u}, see also \cite[Proposition 6]{Di-Pietro.Droniou.ea:21} for the $\min$-term.
  \medskip\\
  2. \emph{Uniqueness.}
  Let $(\b u,p),(\b u',p') \in \b U \times P$ be two solutions of \eqref{eq:ns.weak}. Taking the difference of \eqref{eq:ns.weak:momentum} written first for $(\b u,p)$ and then for $(\b u',p')$, we infer, for all $\b v \in \b U$,
  \begin{equation} \label{eq:well-posedness:0}
    a(\b u,\b v)-a(\b u',\b v)  + c(\b u,\b  v)- c(\b u',\b v)+b(\b v,p-p') = 0.
  \end{equation}
  If $\b u = \b u'$, \eqref{eq:well-posedness:0} yields $b(\b v,p-p') = 0$ for all $\b v \in \b U$. This relation combined with the inf-sup stability of $b$ (cf. \cite[Theorem 1]{Bogovski:79}) yields,
  \[
  \| p-p' \|_{L^{\sob'}(\Omega)} \le C_\cst{b} \sup\limits_{\b v \in {\b U}, |\b v|_{W^{1,r}(\Omega)^d} = 1} b(\b v,p-p') = 0.
  \]
  Hence, uniqueness of the solution is equivalent to uniqueness of the velocity.

  Assume now $2 \le s \le \frac{\sobs^*}{\sobs'}$ and $(\b u,p) \neq (\b u',p')$. Setting $\b e \coloneq \b u-\b u'$, the previous reasoning yields $\b u\neq \b u'$, hence $| \b e|_{W^{1,\sobs}(\Omega)^d} > 0$. 
  Using the H\"older continuity \eqref{eq:c:holder.cont} of $c$ proved in Lemma \ref{lem:c:holder.cont} below with $(\b u,\b w,\b v) = (\b u, \b u',\b e)$ and $m=\tilde r$, we infer
  \[
  \begin{aligned}
    |\b e|_{W^{1,\sobs}(\Omega)^d}^{2} 
    &\ge C_{\cst{c},\sobs}^{-1}\chi_\cst{hc}^{-1}\left(| \b u|_{W^{1,r}(\Omega)^d}^\sob+|\b u'|_{W^{1,r}(\Omega)^d}^\sob\right)^\frac{1-\conv}{\sob}\left(c(\b u',\b e)-c(\b u,\b e)\right) \\
    &= C_{\cst{c},\sobs}^{-1}\chi_\cst{hc}^{-1}\left(| \b u|_{W^{1,r}(\Omega)^d}^\sob+|\b u'|_{W^{1,r}(\Omega)^d}^\sob\right)^\frac{1-\conv}{\sob}\left(a(\b u,\b e)-a(\b u',\b e)\right) \\ 
    &\ge  C_{\cst{c},\sobs}^{-1}\chi_\cst{hc}^{-1}C_\cst{a}\sigma_\cst{hm}\left(\delta^\sob+| \b u|_{W^{1,r}(\Omega)^d}^\sob+| \b u'|_{W^{1,r}(\Omega)^d}^\sob\right)^\frac{\sobs-1-s}{\sob}\delta^{r-\sobs}| \b e|_{W^{1,\sobs}(\Omega)^d}^{2}\\
    &\ge C_{\cst{c},\sobs}^{-1}\chi_\cst{hc}^{-1}C_\cst{a}\sigma_\cst{hm}(1+2C_\cst{v}^r)^\frac{\sobs-1-s}{r}\left[
      \delta^r+\left(\sigma_\cst{hm}^{-1}\| \b f \|_{L^{\sob'}(\Omega)^d}\right)^{r'}
      \right]^\frac{\sobs-1-s}{r}\delta^{r-\sobs}| \b e|_{W^{1,\sobs}(\Omega)^d}^{2},\\
  \end{aligned}
  \]
  where the equality in the second line is obtained using \eqref{eq:well-posedness:0} with $\b v = \b e$ together with the fact that $b(\b e,p-p')=0$ thanks to \eqref{eq:ns.weak:mass}, the third line is obtained invoking the H\"older monotonicity \eqref{eq:a:strong.mono} of $a$ with $m=\sobs$,
  while the conclusion follows using the a priori bound \eqref{eq:well-posedness:bound:u}.
  Since $| \b e|_{W^{1,\sobs}(\Omega)^d} > 0$, simplifying and raising to the power $\frac{r}{\sobs-1-\conv} < 0$, we infer the contrapositive of \eqref{eq:well-posedness:f}.
\end{proof}

We next prove the H\"older monotonicity of $a$ used in the proof of Theorem \ref{proof:thm:well-posedness} above after recalling the following result.
Let $X \subset \R^d$ be measurable, $n \in \N^*$, and let $t,p_1,\ldots,p_n \in (0,\infty\rbrack$ be such that $\sum_{i=1}^n\frac{1}{p_i} = \frac{1}{t}$.
The continuous $(t;p_1,\ldots,p_n)$-H\"older inequality reads:
For any $(f_1,\ldots,f_n) \in \bigtimes_{i=1}^n L^{p_i}(X)$, 
\begin{equation}\label{eq:holder}
  \left\| \prod_{i=1}^n f_i \right\|_{L^t(X)} \le\ \prod_{i=1}^n\| f_i \|_{L^{p_i}(X)}.
\end{equation}

\begin{lemma}[H\"older monotonicity of $a$]\label{lem:a:strong.mono}
   For $m \in \{\sobs,\sob\}$ and all $\b u,\b w \in \b U$, it holds
  \begin{equation}\label{eq:a:strong.mono}
a(\b u,\b u - \b w)-a(\b w,\b u - \b w)  \ge C_\cst{a}\sigma_\cst{hm}\left(\delta^\sob+| \b u|_{W^{1,r}(\Omega)^d}^\sob+| \b w|_{W^{1,r}(\Omega)^d}^\sob\right)^\frac{\sobs-2}{\sob}\delta^{\sob-m}| \b u - \b w|_{W^{1,m}(\Omega)^d}^{m+2-\sobs}.
  \end{equation}
\end{lemma}

\begin{proof}
  The case $m=\sob$ is obtained reasoning as in \cite[Eq. (46)]{Botti.Castanon-Quiroz.ea:20} and using the Korn inequality \eqref{eq:Korn}.
  It remains to prove the case $\sob > 2$ with $m=\sobs=2$.
  Let $\b e \coloneq \b u - \b w$. Using the H\"older monotonicity \eqref{eq:ass:stress:strong.mono} of $\stress$ with $(\b\tau,\b\eta) = (\GRADs\b u,\GRADs\b w)$ yields
  \begin{equation}\label{eq:a:sm:0}
    \begin{aligned}
      \sigma_\cst{hm}\| \GRADs \b e\|_{L^{2}(\Omega)^{d\times d}}^{2}
      &\le \int_\Omega \left(\delta^\sob+|\GRADs\b u|_{d\times d}^\sob+|\GRADs\b w|_{d\times d}^\sob\right)^{\frac{2-\sob}{r}}
      \big(\stress(\cdot,\GRADs\b u)-\stress(\cdot,\GRADs\b w)\big):\GRADs \b e
      \\
      &\le \delta^{2-\sob}  \big(a(\b u,\b e)-a(\b w,\b e)\big),
    \end{aligned}
  \end{equation}
  where we used the monotonicity of $\R \ni x \mapsto x^{2-\sob} \in \R$ since $r > 2$, together with the definition \eqref{eq:weak:ab} of $a$.
  Applying the Korn inequality \eqref{eq:Korn} yields \eqref{eq:a:strong.mono}.
\end{proof}

We move to the H\"older continuity of $c$ used in the proof of Theorem \ref{proof:thm:well-posedness} above after recalling some preliminary results.
For any $n \in \N^*$, the $\conv$-power framed function $\R^n \ni x \mapsto |x|^{\conv-2}x \in \R^n$ enjoys the following properties (see \cite[Appendix A]{Botti.Castanon-Quiroz.ea:20}):
There exist $C_\cst{hc},C_\cst{hm} \in (0,\infty)$ depending only on $s$, such that the following H\"older continuity and the H\"older monotonicity properties hold for all $x,y \in \R^n$,
\begin{subequations}\label{eq:s.power}
\begin{align}
||x|^{\conv-2}x-|y|^{\conv-2}y| &\le C_\cst{hc} (|x|^\conv+|y|^\conv)^\frac{\conv-\convs}{\conv}|x - y|^{\convs-1}, \label{eq:s.power:holder.cont}\\
(|x|^{\conv-2}x-|y|^{\conv-2}y) \cdot (x-y) &\ge C_\cst{hm}(|x|^\conv+|y|^\conv)^\frac{\convs-2}{\conv} |x - y|^{\conv+2-\convs}.\label{eq:s.power:strong.mono}
\end{align}
\end{subequations}
Finally, we recall the following Sobolev embeddings:
For all $m \in [1,\infty)$ such that $m \le r^*$ and all $\b v \in \b U$,
\begin{equation}\label{eq:sob:emb}
\|\b v\|_{L^m(\Omega)^d} \le C_{\cst{S},m} | \b v|_{W^{1,r}(\Omega)^d},
\end{equation}
where $C_{\cst{S},m} > 0$ is a real number depending only on $m$, $r$, $d$, and $\Omega$.

\begin{lemma}[H\"older continuity of $c$]\label{lem:c:holder.cont}
  Under Assumption \ref{ass:conv}, for all $m \in [1,r]$ such that
  \begin{equation}\label{eq:cont.c:s}
    \conv \le \frac{m^*}{m'},
  \end{equation}
  and all $\b u, \b v, \b w \in \b U$, 
  \begin{equation}\label{eq:c:holder.cont}
      \left|
      c(\b u,\b v)-c(\b w,\b v)
      \right| \le C_{\cst{c},m}\chi_\cst{hc}\left(| \b u|_{W^{1,r}(\Omega)^d}^\sob+| \b w|_{W^{1,r}(\Omega)^d}^\sob\right)^\frac{\conv+1-\convs}{\sob}| \b u-\b w|_{W^{1,m}(\Omega)^d}^{\convs-1}| \b v|_{W^{1,m}(\Omega)^d},
  \end{equation}
where $C_{\cst{c},m} > 0$ depends only on $m$, $r$, $s$, $d$, and $\Omega$.
\end{lemma}

\begin{proof}
  Throughout the proof, $a\lesssim b$ (resp. $a\gtrsim b$) means $a\le Cb$ (resp. $a\ge Cb$) with $C>0$ having the same dependencies as $C_{\cst{c},m}$.
  
  Using the definition \eqref{eq:weak:c} of $c$ and inserting
  $\pm\big(
  {\frac{1}{\conv}\int_\Omega \tri{\convec(\cdot,\b u) }{\GRAD}{\b w}{\b v}}
  + {\frac{\conv-2}{\conv}\int_\Omega \frac{\b v \cdot \b u}{|\b u|^{2}}\tri{\convec(\cdot,\b u) }{\GRAD}{\b w}{\b u}}
  + {\frac{1}{\conv'}\int_\Omega \tri{\convec(\cdot,\b u) }{\GRAD}{\b v}{\b w}}
  \big)
  $,
  we get
  \begin{equation}\label{eq:c:holder.cont:decomp}
    \begin{aligned}
      &c(\b u,\b v)-c(\b w,\b v) = \frac{1}{\conv}\biggl(\!~\underbrace{\int_\Omega \tri{\convec(\cdot,\b u) }{\GRAD}{(\b u-\b w)}{\b v}}_{\mathcal T_1}\
      + \underbrace{\int_\Omega \btri{(\convec(\cdot,\b u) -\convec(\cdot,\b w) )}{\GRAD}{\b w}{\b v}}_{\mathcal T_2}\!~\biggr) \\
      &\ +\frac{\conv-2}{\conv}\biggl(\!~\underbrace{\int_\Omega \frac{\b v \cdot \b u}{|\b u|^{2}}\tri{\convec(\cdot,\b u) }{\GRAD}{(\b u-\b w)}{\b u}}_{\mathcal T_3}\ +\underbrace{\int_\Omega\left( \frac{\b v \cdot \b u}{|\b u|^{2}}\tri{\convec(\cdot,\b u) }{\GRAD}{\b w}{\b u}-\frac{\b v \cdot \b w}{|\b w|^{2}}\tri{\convec(\cdot,\b w) }{\GRAD}{\b w}{\b w}\right)}_{\mathcal T_4}\!~\biggr) \\
      &\ -\frac{1}{\conv'}\biggl(\!~\underbrace{\int_\Omega \tri{\convec(\cdot,\b u) }{\GRAD}{\b v}{(\b u-\b w)} }_{\mathcal T_5}\ +\underbrace{\int_\Omega \btri{(\convec(\cdot,\b u) -\convec(\cdot,\b w) )}{\GRAD}{\b v}{\b w}}_{\mathcal T_6}\!~\biggr).
    \end{aligned}
  \end{equation}
  We start by writing
  \begin{equation}\label{eq:c:holder.cont:T1+T3}
    \begin{aligned}
      |\mathcal T_1|
      + |\mathcal T_3|
      %% + |\mathcal T_5|
      &\lesssim \int_\Omega |\convec(\cdot,\b u)||\GRAD(\b u-\b w)|_{d\times d}|\b v| \\
      &\le \chi_\cst{hc}\int_\Omega |\b u|^{s-1}|\GRAD(\b u-\b w)|_{d\times d}|\b v| \\
      &\le \chi_\cst{hc}\|\b u\|_{L^{\conv m'}(\Omega)^d}^{s-1}|\b u-\b w|_{W^{1,m}(\Omega)^d}\|\b v \|_{L^{\conv m'}(\Omega)^d}\\
      &\lesssim\chi_\cst{hc}| \b u|_{W^{1,r}(\Omega)^d}^{s-1}| \b u-\b w|_{W^{1,r}(\Omega)^d}^{2-\convs}| \b u-\b w|_{W^{1,m}(\Omega)^d}^{\convs-1}| \b v|_{W^{1,m}(\Omega)^d}\\
      &\lesssim\chi_\cst{hc}\!\left(| \b u|_{W^{1,r}(\Omega)^d}^\sob+| \b w|_{W^{1,r}(\Omega)^d}^\sob\right)^\frac{\conv+1-\convs}{\sob}| \b u-\b w|_{W^{1,m}(\Omega)^d}^{\convs-1}| \b v|_{W^{1,m}(\Omega)^d},
    \end{aligned}
  \end{equation}
  where we have used the H\"older continuity \eqref{eq:ass:conv:holder.cont} of $\convec$ with $(\b w, \b v) = (\b u, \b 0)$ to pass to the second line,
  the $(1;\frac{\conv m'}{s-1},m,\conv m')$-H\"older inequality \eqref{eq:holder} in the third line,
  the Sobolev embedding \eqref{eq:sob:emb} (valid since $\conv m' \le m^*$ by \eqref{eq:cont.c:s} and $m\le r\implies m^*\le r^*$, so that $\conv m'\le r^*$) together with the $(1;\frac{r}{m},\frac{r}{r-m})$-H\"older inequality \eqref{eq:holder} (valid since $m \le r$) in the fourth line, while the conclusion follows writing first $| \b u|_{W^{1,r}(\Omega)^d}^{\conv-1} \le (| \b u|_{W^{1,r}(\Omega)^d}+| \b w|_{W^{1,r}(\Omega)^d})^{\conv-1}$ by monotonicity of $\R \ni x \mapsto x^{s-1} \in \R$,
  then $| \b u-\b w|_{W^{1,r}(\Omega)^d}^{2-\convs} \le (| \b u|_{W^{1,r}(\Omega)^d}+| \b w|_{W^{1,r}(\Omega)^d})^{2-\convs}$ by a triangle inequality,
  and, finally, noticing that $(x+y)^r \lesssim x^r+y^r$ for all $x,y \in [0,\infty)$ (see \cite[Eq. (36)]{Botti.Castanon-Quiroz.ea:20}).
  Similar arguments give for the fifth term
  \begin{equation}\label{eq:c:holder.cont:T5}
    \begin{aligned}
      |\mathcal T_5|
      %&\lesssim \int_\Omega |\convec(\cdot,\b u)||\b u-\b w||\GRAD\b v|_{d\times d} \\
      &\le \chi_\cst{hc}\int_\Omega |\b u|^{s-1}|\b u-\b w||\GRAD\b v|_{d\times d} \\
      &\le \chi_\cst{hc}\|\b u\|_{L^{\conv m'}(\Omega)^d}^{s-1}\|\b u-\b w \|_{L^{\conv m'}(\Omega)^d}|\b v|_{W^{1,m}(\Omega)^d}\\
      %&\lesssim\chi_\cst{hc}| \b u|_{W^{1,r}(\Omega)^d}^{s-1}| \b u-\b w|_{W^{1,r}(\Omega)^d}^{2-\convs}| \b u-\b w|_{W^{1,m}(\Omega)^d}^{\convs-1}| \b v|_{W^{1,m}(\Omega)^d}\\
      &\lesssim\chi_\cst{hc}\!\left(| \b u|_{W^{1,r}(\Omega)^d}^\sob+| \b w|_{W^{1,r}(\Omega)^d}^\sob\right)^\frac{\conv+1-\convs}{\sob}| \b u-\b w|_{W^{1,m}(\Omega)^d}^{\convs-1}| \b v|_{W^{1,m}(\Omega)^d}.
    \end{aligned}
  \end{equation}
  
  We next estimate $\mathcal T_4$, which provides a paradigm for the remaining terms.
  Inserting $\pm\big(|\b u|^{s-2}|\b w|^{-1}\b u \otimes \b w + |\b w|^{s-2}|\b u|^{-1}\b w \otimes \b u + |\b u|^{s-3}\b u \otimes \b u\big)$, rearranging the terms, and using a triangle inequality yields
  \begin{equation}\label{eq:c:holder.cont:1}
    \begin{aligned}
      &|\b w|^{s-1}\left||\b u|^{-2}\b u\otimes \b u-|\b w|^{-2}\b w\otimes \b w\right|_{d\times d}
      \\
      &\quad
      \le \left| |\b w|^{-1}\left(|\b u|^{s-2}\b u-|\b w|^{s-2}\b w\right)\otimes \b w\right|_{d\times d}
      + \left||\b u|^{-1}\!\left(|\b u|^{\conv-2}\b u-|\b w|^{\conv-2}\b w\right)\otimes \b u\right|_{d\times d}
      \\
      &\qquad
      + \left||\b u|^{-1}\!\left(|\b w|^{\conv-2}|\b w|-|\b u|^{\conv-2}|\b u|\right)\left(|\b u|^{-1}\b u+|\b w|^{-1}\b w\right)\otimes \b u\right|_{d\times d}
      \\
      &\quad
      \le 4\left||\b u|^{s-2}\b u-|\b w|^{s-2}\b w\right| 
      \le 4C_\cst{hc} \left(|\b u|^\conv+|\b w|^\conv\right)^\frac{\conv-\convs}{\conv}|\b u - \b w|^{\convs-1},
    \end{aligned}
  \end{equation}
  where the last line is obtained using Cauchy-Schwarz and triangle inequalities along with the H\"older continuity property \eqref{eq:s.power:holder.cont}.
  Thus, inserting $\pm \frac{\b v \cdot \b u}{|\b u|^{2}}\tri{\convec(\cdot,\b w) }{\GRAD}{\b w}{\b u}$ and using a triangle inequality leads to
  \begin{equation}\label{eq:c:holder.cont:T4}
    \begin{aligned}
      |\mathcal T_4| &\le \int_\Omega \left(|\convec(\cdot,\b u)-\convec(\cdot,\b w)|+|\convec(\cdot,\b w)|\left||\b u|^{-2}\b u\otimes \b u-|\b w|^{-2}\b w\otimes \b w\right|_{d\times d}\right) |\GRAD\b w|_{d\times d}|\b v| \\
      &\le (1+4C_\cst{hc})\chi_\cst{hc}\int_\Omega \left(|\b u|^\conv+|\b w|^\conv\right)^\frac{\conv-\convs}{\conv}|\b u - \b w|^{\convs-1}|\GRAD\b w|_{d\times d}|\b v|  \\
      &\le (1+4C_\cst{hc})\chi_\cst{hc}\left(\|\b u\|_{L^{\conv m'}(\Omega)^d}^\conv+\|\b w\|_{L^{\conv m'}(\Omega)^d}^\conv\right)^\frac{\conv-\convs}{\conv}\|\b u-\b w\|_{L^{\conv m'}(\Omega)^d}^{\convs-1}| \b w|_{W^{1,m}(\Omega)^d}\|\b v \|_{L^{\conv m'}(\Omega)^d}\\
      &\lesssim\chi_\cst{hc}\left(| \b u|_{W^{1,r}(\Omega)^d}^\sob+| \b w|_{W^{1,r}(\Omega)^d}^\sob\right)^\frac{\conv+1-\convs}{\sob}| \b u-\b w|_{W^{1,m}(\Omega)^d}^{\convs-1}| \b v|_{W^{1,m}(\Omega)^d},
    \end{aligned}
  \end{equation}
  where we have used the H\"older continuity property \eqref{eq:ass:conv:holder.cont} of $\convec$ together with \eqref{eq:c:holder.cont:1} in the second line, the $(1;\frac{\conv m'}{\conv-\convs},\frac{\conv m'}{\convs-1},m,\conv m')$-H\"older inequality \eqref{eq:holder} in the third line, and the Sobolev embedding \eqref{eq:sob:emb} (since, as showed above, $\conv m' \le m^* \le r^*$) together with the $(1;\frac{r}{m},\frac{r}{r-m})$-H\"older inequality \eqref{eq:holder} (since $m \le r$) and the inequality \cite[Eq. (36)]{Botti.Castanon-Quiroz.ea:20} to conclude.
  Similar arguments give for the second and sixth terms an analogous bound:
  \begin{equation}\label{eq:c:holder.cont:T2+T6}
    |\mathcal T_2| + |\mathcal T_6|
    \lesssim\chi_\cst{hc}\left(| \b u|_{W^{1,r}(\Omega)^d}^\sob+| \b w|_{W^{1,r}(\Omega)^d}^\sob\right)^\frac{\conv+1-\convs}{\sob}| \b u-\b w|_{W^{1,m}(\Omega)^d}^{\convs-1}| \b v|_{W^{1,m}(\Omega)^d}.
  \end{equation}
  Plugging \eqref{eq:c:holder.cont:T1+T3}, \eqref{eq:c:holder.cont:T5}, \eqref{eq:c:holder.cont:T4}, and \eqref{eq:c:holder.cont:T2+T6} into \eqref{eq:c:holder.cont:decomp} gives \eqref{eq:c:holder.cont}.
\end{proof}

\begin{remark}[Role of condition \eqref{eq:cont.c:s}]\label{rem:cont.c:s}
  In the proof of Lemma \ref{lem:c:holder.cont}, condition \eqref{eq:cont.c:s} is used (along with $m\le r$)  to control $L^{sm'}$-norms of functions in $\b U$ with $W^{1,m}$-(semi)norms.
\end{remark}

%------------------------------------------------------------------------------%
%------------------------------------------------------------------------------%

\section{Discrete setting}\label{sec:discrete.setting}

In this section we establish the discrete setting.

\subsection{Mesh and notation for inequalities up to a multiplicative constant}

Let $\mathcal H \subset (0,\infty)$ be a countable set of meshsizes having $0$ as its unique accumulation point.
Given a set $X\subset\R^d$, we denote by $h_X\coloneq\sup_{(\b x,\b y)\in X^2}|\b x-\b y|$ its diameter.
For all $h \in \mathcal H$, we define a mesh as a couple $\mathcal M_h\coloneq(\T_h,\F_h)$ where $\T_h$ is a finite collection of polyhedral elements such that $h=\max_{T\in\T_h}h_T$, while $\F_h$ is a finite collection of planar faces.
Notice that, here and in what follows, we use the three-dimensional nomenclature also when $d=2$, i.e., we speak of polyhedra and faces rather than polygons and edges.
It is assumed henceforth that the mesh $\mathcal M_h$ matches the geometrical requirements detailed in \cite[Definition 1.7]{Di-Pietro.Droniou:20}.
In order to have the boundedness property \eqref{eq:I:boundedness} below for the interpolator, we additionally assume that the mesh elements are star-shaped with respect to every point of a ball of radius uniformly comparable to the element diameter.
Boundary faces lying on $\partial\Omega$ and internal faces contained in
$\Omega$ are collected in the sets $\F_h^{\rm b}$ and $\F_h^{\rm i}$, respectively.
For every mesh element $T\in\T_h$, we denote by $\F_T$ the subset of $\F_h$ containing the faces that lie on the boundary $\partial T$ of $T$. For every face $F \in \F_h$, we denote by $\T_F$ the subset of $\T_h$ containing the one (if $F\in\F_h^{\rm b}$) or two (if $F\in\F_h^{\rm i}$) elements on whose boundary $F$ lies.
For each mesh element $T\in\T_h$ and face $F\in\F_T$, $\b n_{TF}$ denotes the (constant) unit vector normal to $F$ pointing out of $T$.

Our focus is on the $h$-convergence analysis, so we assume that the mesh sequence $(\mathcal M_h)_{h\in\mathcal H}$ is regular in the sense of \cite[Definition 1.9]{Di-Pietro.Droniou:20}, with regularity parameter uniformly bounded away from zero.
The mesh regularity assumption implies, in particular, that the diameter of a mesh element and those of its faces are comparable uniformly in $h$ and that the number of faces of one element is bounded above by an integer independent of $h$.

To avoid the proliferation of generic constants, we write henceforth $a\lesssim b$ (resp., $a\gtrsim b$) for the inequality $a\le Cb$ (resp., $a\ge Cb$) with real number $C>0$ independent of $h$, of the constants $\delta,\sigma_\cst{hc},\sigma_\cst{hm},\chi_\cst{hc}$ in Assumptions \ref{ass:stress} and \ref{ass:conv}, and, for local inequalities, of the mesh element or face on which the inequality holds.
We also write $a\simeq b$ to mean $a\lesssim b$ and $b\lesssim a$.
The dependencies of the hidden constants are further specified when needed.

\subsection{Projectors and broken spaces}
  
Given $X \in \T_h \cup \F_h$ and $l \in \N$, we denote by $\Poly^l(X)$ the space spanned by the restriction to $X$ of scalar-valued, $d$-variate polynomials of total degree $\le l$.
The local $L^2$-orthogonal projector $\proj{X}{l} : L^{1}(X) \to \Poly^l(X)$ is defined such that, for all $v \in L^{1}(X)$,
\begin{equation}\label{eq:proj}
  \displaystyle\int_X (\proj{X}{l} v-v) w = 0 \qquad \forall w \in  \Poly^{l}(X).
\end{equation}
When applied to vector-valued fields in $L^1(X)^d$ (resp., tensor-valued fields in $L^1(X)^{d\times d}$), the $L^2$-orthogonal projector mapping on $\Poly^l(X)^d$ (resp., $\Poly^l(X)^{d\times d}$) acts component-wise and is denoted in boldface font.
Let $T\in\T_h$, $m\in [1,\infty]$, $i\in[0,l+1]$, and $j\in[0,i]$.
The following $(i,m,j)$-approximation properties of $\proj{T}{l}$ hold:
For any $v\in W^{i,m}(T)$,
\begin{subequations}\label{eq:proj:app}
\begin{equation}\label{eq:proj:app:T}
  |v-\proj{T}{l}v|_{W^{j,m}(T)} \lesssim h_T^{i-j}|v|_{W^{i,m}(T)}.
\end{equation}
If, additionally, $i\ge j+1$, we have the following $(i,m,j)$-trace approximation property:
\begin{equation}\label{eq:proj:app:F}
    |v-\proj{T}{l}v|_{W^{j,m}(\partial T)}\lesssim h_T^{i-j-\frac{1}{m}}|v|_{W^{i,m}(T)}.
\end{equation}
\end{subequations}
The hidden constants in \eqref{eq:proj:app} are independent of $h$ and $T$, but possibly depend on $d$, the mesh regularity parameter, $l$, $i$, and $m$.
The approximation properties \eqref{eq:proj:app} are proved for integer $i$ and $j$ in \cite[Appendix A.2]{Di-Pietro.Droniou:17} (see also \cite[Theorem 1.45]{Di-Pietro.Droniou:20}), and can be extended to non-integer values using standard interpolation techniques (see, e.g., \cite[Theorem 5.1]{Lions.Magenes:72}).

At the global level, we define the broken polynomial space $\Poly^l(\T_h)$ spanned by functions in $L^1(\Omega)$ whose restriction to each mesh element $T\in\T_h$ lies in $\Poly^l(T)$, and we define the global $L^2$-orthogonal projector $\proj{h}{l} : L^{1}(\Omega) \to \Poly^l(\T_h)$ such that, for all $v \in L^{1}(\Omega)$ and all $T \in \T_h$,
\[
(\proj{h}{l} v)\res{T} \coloneq \proj{T}{l} v\res{T}.
\]
Broken polynomial spaces are subspaces of the broken Sobolev spaces
\[
W^{n,m}(\T_h)\coloneq\left\{ v\in L^{m}(\Omega)\ : \ v\res{T}\in W^{n,m}(T)\quad\forall T\in\T_h\right\},
\]
which we endow with the usual broken seminorm $|\cdot|_{W^{n,m}(\T_h)}$; see, e.g., \cite[Section 1.2.5]{Di-Pietro.Ern:12} for further details.
As a consequence of the $(l+1,\cdot,\cdot)$-approximation properties \eqref{eq:proj:app:T}, for all $\phi \in C^\infty_c(\Omega)$,
\begin{equation}\label{eq:proj.seq.cons}
    \proj{h}{l} \phi \CV{h}{0}{} \phi \text{ strongly in }W^{[0,l],\lbrack 1,\infty)}(\T_h).
\end{equation}
We define the broken gradient operator $\brkGRAD : W^{1,1}(\T_h) \rightarrow L^1(\Omega)^d$ such that, for all $v \in W^{1,1}(\T_h)$ and all $T \in \T_h$, $(\brkGRAD v)\res{T} \coloneq \GRAD v\res{T}$. We define similarly the broken gradient acting on vector fields along with its symmetric part $\brkGRADs$, as well as the broken divergence operator $\brkDIV$ acting (row-wise) on tensor fields.
The global $L^2$-orthogonal projector $\PROJ{h}{l}$ mapping vector-valued fields in $L^1(\Omega)^d$ (resp., tensor-valued fields in $L^1(\Omega)^{d\times d}$) on $\Poly^l(\T_h)^d$ (resp., $\Poly^l(\T_h)^{d\times d}$) is obtained applying $\proj{h}{l}$ component-wise.

%------------------------------------------------------------------------------%
%------------------------------------------------------------------------------%

\section{Discrete problem and main results}\label{sec:discrete.main.results}

\subsection{Discrete spaces and norms}

Let an integer $k\ge 1$ be fixed. The HHO space of discrete velocities is
\[
\bdU{h}{k} \coloneqq \left\{
\bu v_h \coloneq ((\b v_T)_{T \in \T_h},(\b v_F)_{F\in \F_h}) \ : \ \b v_T \in \Poly^k(T)^d\ \ \forall T \in \T_h\ \mbox{ and }\ \b v_F \in \Poly^k(F)^d\ \ \forall F \in \F_h \right\}.
\]
The interpolation operator $\bI{h}{k} : W^{1,1}(\Omega)^d \to  \bdU{h}{k} $ maps a function $\b v \in W^{1,1}(\Omega)^d$ on the vector of polynomials $\bI{h}{k}\b v$ defined as follows:
\[
  \bI{h}{k} \b v \coloneqq ((\PROJ{T}{k} \b v\res{T})_{T \in \T_h},(\PROJ{F}{k} \b v\res{F})_{F \in \F_h}).
  \]
For all $T \in \T_h$, we denote by $\bdU{T}{k}$ and $\bI{T}{k}$ the restrictions of $\bI{h}{k}$ and $\bdU{h}{k}$ to $T$, respectively and, for all $\bu v_h \in \bdU{h}{k}$, we let $\bu v_T \coloneqq (\b v_T,(\b v_F)_{F\in \F_T}) \in \bdU{T}{k}$ denote the vector collecting the polynomial components of $\bu v_h$ attached to $T$ and its faces.
Furthermore, for all $\bu v_h \in \bdU{h}{k}$, we define the broken polynomial field $\b v_h\in\Poly^k(\T_h)^d$ obtained patching element unknowns, that is,
\begin{equation}\label{eq:vh}
  (\b v_h)\res{T} \coloneqq \b v_T\qquad\forall T \in \T_h.
\end{equation}

For any $m \in (1,\infty)$, we define on $\bdU{h}{k}$ the seminorms $\| {\cdot} \|_{1,m,h}$ and $\| {\cdot} \|_{\strain,m,h}$ such that, for all $\bu v_h \in \bdU{h}{k}$,
\begin{subequations}\label{eq:norm.epsilon.sob}
  \begin{equation}\label{eq:norm.epsilon.sob.h}
    \| \bu v_h \|_{\bullet,m,h} \coloneqq \left(\displaystyle\sum_{T \in \T_h}\| \bu v_T \|_{\bullet,m,T}^m\right)^\frac{1}{m}\qquad
    \forall\bullet\in\{1,\strain\},
    \end{equation}
    with, for all $T \in \T_h$,
    \begin{equation}\label{eq:norm.epsilon.sob.T}
      \begin{aligned}
        \| \bu v_T \|_{1,m,T} &\coloneqq \left(\| \GRAD \b v_T \|^m_{L^m(T)^{d\times d}} + \displaystyle\sum_{F \in \F_T} h_F^{1-m} \| \b v_F - \b v_T\|^m_{L^m(F)^d}\right)^\frac{1}{m},
        \\
      \| \bu v_T \|_{\strain,m,T} &\coloneqq \left(\| \GRADs \b v_T \|^m_{L^m(T)^{d\times d}} + \displaystyle\sum_{F \in \F_T} h_F^{1-m} \| \b v_F - \b v_T\|^m_{L^m(F)^d}\right)^\frac{1}{m}.
      \end{aligned}  
    \end{equation}
\end{subequations}
The difference between these seminorms lies in the fact that the symmetric part of the gradient replaces the gradient in $\|{\cdot}\|_{\strain,m,T}$.

The discrete velocity and pressure are sought in the following spaces, which embed, respectively, the homogeneous boundary condition for the velocity and the zero-average constraint for the pressure:
\[
\bdU{h,0}{k} \coloneqq \left\{ \bu v_h = ((\b v_T)_{T \in \T_h},(\b v_F)_{F\in \F_h}) \in \bdU{h}{k} \ : \ \b v_F = \b 0 \quad \forall F \in \Fb \right\},\quad
\dP{h}{k} \coloneqq \Poly^k(\T_h)\cap P.
\]
The following discrete Korn inequality was proved in \cite[Lemma 15]{Botti.Castanon-Quiroz.ea:20}:
  \begin{equation}\label{eq:discrete.Korn}
    \| \b v_h \|_{L^m(\Omega)^d}^m+|\b v_h|_{W^{1,m}(\T_h)^d}^m \lesssim \| \bu v_h \|_{\strain,m,h}^m\qquad
    \forall \bu v_h  \in \bdU{h,0}{k}.
  \end{equation}
A first consequence of \eqref{eq:discrete.Korn} is that $\| {\cdot} \|_{\strain,m,h}$ is a norm on $\bdU{h,0}{k}$.
A second consequence is the following equivalence uniform in $h$:
\begin{equation}\label{eq:norm.equiv}
  \| \bu v_h \|_{\strain,m,h} \simeq \| \bu v_h \|_{1,m,h}\qquad
  \forall \bu v_h  \in \bdU{h,0}{k}.
\end{equation}
The following boundedness property for $\bI{T}{k}$ is proved in \cite[Proposition 6.24]{Di-Pietro.Droniou:20} and requires the star-shaped assumption on the mesh elements:
For all $T \in \T_h$ and all $\b v \in W^{1,\sob}(T)^d$,
\begin{equation}\label{eq:I:boundedness}
  \|\bI{T}{k} \b v\|_{1,\sob,T} \le C_\cst{I}| \b v |_{W^{1,\sob}(T)^d},
\end{equation}
where $C_\cst{I} \ge 1$ depends only on $d$, the mesh regularity parameter, $\sob$, and $k$.

Finally, we recall the following discrete Sobolev embeddings (see \cite[Proposition 5.4]{Di-Pietro.Droniou:17}): For all $m \in [1,\infty)$ such that $m \le \sob^*$,
\begin{equation} \label{eq:discrete.sob.emb} 
  \| \b v_h \|_{L^m(\Omega)^d} \lesssim  \| \bu v_h \|_{1,\sob,h}
\lesssim \| \bu v_h\|_{\strain,\sob,h}
  \qquad\forall \bu v_h  \in \bdU{h,0}{k},
\end{equation}
where the last inequality is a consequence of \eqref{eq:norm.equiv}.

\subsection{Local gradient reconstruction}

For all $T \in \T_h$, we define the local gradient reconstruction $\dgrad{k}{T} : \bdU{T}{k} \to \Poly^{k}(T)^{d\times d}$ such that, for all $\bu v_T \in \bdU{T}{k}$,
\begin{equation}\label{eq:G.T}
  \int_T \dgrad{k}{T} \bu v_T : \b \tau = \int_T \GRAD \b v_T : \b \tau + \sum_{F \in \mathcal F_T} \int_F (\b v_F-\b v_T)\cdot (\b \tau \b n_{TF})\qquad \forall \b \tau \in  \Poly^{k}(T)^{d\times d}.
\end{equation}
  A global gradient reconstruction $\dgrad{k}{h} : \bdU{h}{k} \to \Poly^{k}(\mathcal T_h)^{d\times d}$ is obtained patching the local contributions, that is, for all $\bu v_h \in \bdU{h}{k}$, we set
\begin{equation}\label{eq:Gh}
  (\dgrad{k}{h} \bu v_h)\res{T} \coloneq \dgrad{k}{T} \bu v_T \qquad \forall T\in\mathcal T_h.
\end{equation}
By construction, the following commutation property holds (see \cite[Eq. (23)]{Di-Pietro.Krell:18}):
For all $T\in\T_h$ and all $\b v \in  W^{1,1}(T)^d$,
\begin{equation}\label{eq:G.T:proj}
\dgrad{k}{T} (\bI{T}{k}\b v) =\PROJ{T}{k}(\GRAD \b v).
\end{equation}
Combined with the $(k+1,\sob,0)$-approximation properties \eqref{eq:proj:app} of $\PROJ{T}{k}$, this gives, for all $\b v \in  W^{k+2,\sob}(T)^d$,
\begin{equation}\label{eq:Gh:consistency}
  \| \dgrad{k}{T}(\bI{T}{k}\b v) -\GRAD \b v\|_{L^\sob(T)^{d\times d}} + \left(\sum_{F \in \F_T} h_F \| \dgrad{k}{T}(\bI{T}{k}\b v) -\GRAD \b v\|_{L^\sob(F)^{d\times d}}^\sob\right)^\frac{1}{\sob} \lesssim h_T^{k+1} |\b v|_{W^{k+2,\sob}(T)^d}.
\end{equation}
As a consequence, for all $\b\phi \in C^\infty_c(\Omega)^d$,
\begin{equation}\label{eq:GT.seq.cons}
  \text{
    $\dgrad{k}{h}(\bI{h}{k} \b\phi) \CV{h}{0}{} \GRAD \b\phi$ strongly in $L^{\lbrack 1,\infty)}(\Omega)^{d\times d}$.
  }
\end{equation}
Combining \cite[Proposition 1.1]{Di-Pietro.Krell:18} with the local Lebesgue embeddings of \cite[Lemma 1.25]{Di-Pietro.Droniou:20} (see also \cite[Lemma 5.1]{Di-Pietro.Droniou:17}) gives
\begin{equation}\label{eq:Gh:boundedness}
  \| \dgrad{k}{T} \bu v_T \|_{L^\sob(T)^{d\times d}} \lesssim \| \bu v_T \|_{1,\sob,T}\qquad
  \forall \bu v_T \in \bdU{T}{k}.
\end{equation}

\subsection{Convective term}

The convective term is discretized through the function $\cst{c}_h : \bdU{h}{k} \times \bdU{h}{k} \to \R$ such that, for all $(\bu w_h,\bu  v_h) \in \bdU{h}{k} \times \bdU{h}{k}$, 
\begin{equation}\label{eq:ch}
\begin{aligned}
  \cst{c}_h(\bu w_h,\bu v_h) &\coloneq \frac{1}{\conv}\int_\Omega \tri{\convec(\cdot,\b w_h)}{\dgrad{k}{h}}{\bu w_h}{\b v_h}+\frac{\conv-2}{\conv}\int_\Omega \frac{\b v_h \cdot \b w_h}{|\b w_h|^2}\tri{\convec(\cdot,\b w_h)}{\dgrad{k}{h}}{\bu w_h}{\b w_h} \\
  & \quad - \frac{1}{\conv'}\int_\Omega \tri{\convec(\cdot,\b w_h)}{\dgrad{k}{h}}{\bu v_h}{\b w_h}.
  \end{aligned}
\end{equation}
  This expression is obtained replacing in \eqref{eq:weak:c} the continuous gradient by $\dgrad{k}{h}$ and the functions by their broken polynomial counterparts obtained according to \eqref{eq:vh}.

\begin{remark}[Comparison with \cite{Di-Pietro.Krell:18}]\label{rem:standard.discrete.convective} 
  For the standard convection law corresponding to $s=2$ and $\convec(\cdot,\b w) \equiv \b w$,
  the convective function \eqref{eq:ch} becomes
  \[
  \cst{c}_h(\bu w_h,\bu v_h)
  = \frac{1}{2}\int_\Omega \tri{\b w_h}{\dgrad{k}{h}}{\bu w_h}{\b v_h}
  - \frac{1}{2}\int_\Omega \tri{\b w_h}{\dgrad{k}{h}}{\bu v_h}{\b w_h}.
  \]
  This expression differs from the one originally proposed in \cite[Eq. (32)]{Di-Pietro.Krell:18} in that a gradient reconstruction of degree $k$ instead of $2k$ (noted $\dgrad{2k}{h}$ therein and defined taking $\Poly^{2k}(T)^{d\times d}$ instead of $\Poly^{k}(T)^{d\times d}$ both as a codomain for $\dgrad{2k}{T}$ and as a test space in \eqref{eq:G.T}) is used.
  The latter choice leads to a simpler expression in the standard case since, for all $\bu v_h,\bu w_h\in\bu U_h^k$ and all $T\in\T_h$, the quantity $\b v_T\otimes\convec(\cdot,\b w_T)$ is a polynomial of degree $2k$ inside $T$, and $\dgrad{2k}{T}$ can thus be expanded according to its definition; cf. \cite[Eq. (33)]{Di-Pietro.Krell:18}.
  This is no longer the case when considering general convection laws, for which the quantity $\b v_T\otimes\convec(\cdot,\b w_T)$ is possibly non-polynomial inside $T$   (hence we cannot use its degree to design a discrete gradient allowing to mimic the trick of \cite[Eq. (33)]{Di-Pietro.Krell:18}).
    Additionally, the consistency property \eqref{eq:Gh:consistency} of $\dgrad{k}{T}$ is not valid for $\dgrad{l}{T}$ with $l > k$.
    As a matter of fact, it is shown in \cite[Proposition 1]{Di-Pietro.Krell:18} that one order of convergence is lost in this case, which would result in a degradation of the error estimates if $\dgrad{2k}{T}$ were used in place of $\dgrad{k}{T}$ in the expression of $\cst{c}_T$.%
\end{remark}

\begin{lemma}[Properties of $\cst{c}_h$] \label{lem:properties:ch}
  Under Assumption \ref{ass:conv}, the following properties for $\cst{c}_h$ hold:
  \begin{enumerate}[left=0pt]
  \item \emph{Non-dissipativity.}
    For all $\bu w_h \in \bdU{h}{k}$,
    \begin{equation}\label{eq:ch:non-dissip}
      \cst{c}_h(\bu w_h,\bu w_h) = 0.
    \end{equation}
  \item \emph{H\"older continuity.} 
    For $m$ and $\conv$ as in Lemma \ref{lem:c:holder.cont}
    (i.e., $m\in [1,\sob]$ and $\conv \le \frac{m^*}{m'}$) and all $\bu u_h, \bu v_h, \bu w_h \in \bdU{h,0}{k}$, it holds
    \begin{multline}\label{eq:ch:holder.cont}
      \hspace{-3mm}\left|
      \cst{c}_h(\bu u_h,\bu v_h)-\cst{c}_h(\bu w_h,\bu v_h)
      \right|
       \le C_{\cst{dc},m}\chi_\cst{hc}\left(\|\bu u_h\|_{\strain,\sob,h}^\sob+\|\bu w_h\|_{\strain,\sob,h}^\sob\right)^\frac{\conv+1-\convs}{\sob}\|\bu u_h - \bu w_h\|_{\strain,m,h}^{\convs-1}\|\bu v_h \|_{\strain,m,h},
    \end{multline}
    where $C_{\cst{dc},m} > 0$ is independent of $h$.
  \item \emph{Consistency.} If
    \begin{equation}\label{eq:cons.ch:s}
      \conv \le \frac{\sob^*}{\sob'},
    \end{equation}
    then, for all $\b w \in \b U \cap W^{k+2,\sob}(\T_h)^d \cap W^{k+1,s\sob'}(\T_h)^d$ (so that, in particular, $\b w\in W^{1,s\sob'}(\Omega)^d$) such that $\GRAD \cdot \b w = 0$, it holds
    \begin{multline}\label{eq:ch:consistency}
      \sup_{\bu v_h \in \bdU{h,0}{k},\| \bu v_h \|_{\strain,\sob,h} = 1} \left|\int_\Omega \tri{\b w}{\GRAD}{\convec(\cdot,\b w)}{\b v_h} - \cst{c}_h(\bI{h}{k} \b w,\bu v_h)\right|
      \\
      \lesssim h^{k+1}\left[
        |\b w|_{W^{k+1,s\sob'}(\T_h)^d}^s
        {+} \left(
        |\b w|_{W^{1,s\sob'}(\Omega)^d}^s {+} |\b w|_{W^{1,\sob}(\Omega)^d}^s
        \right)^\frac{\conv-1}{s}\left(
        |\b w|_{W^{k+1,s\sob'}(\T_h)^d}{+}|\b w|_{W^{k+2,\sob}(\T_h)^d}
        \right)
        \right]
        \\
      +h^{(k+1)(\convs-1)}\left(|\b w|_{W^{1,s\sob'}(\Omega)^d}^s {+} |\b w|_{W^{1,\sob}(\Omega)^d}^s\right)^\frac{\conv+1-\convs}{s}|\b w|_{W^{k+1,s\sob'}(\T_h)^d}^{\convs-1}.
    \end{multline}
  \item \emph{Sequential consistency.}
    Let $(\bu v_h)_{h\in\mathcal H}$  denote a bounded sequence of $(\bdU{h,0}{k},\|{\cdot}\|_{\strain,\sob,h})_{h \in \mathcal H}$ such that $\b v_h \CV{h}{0} \b v \in \b U$ strongly in $L^{[1,\sob^*)}(\Omega)^d$, and $\dgrad{k}{h} \bu v_h \CV{h}{0} \GRAD \b v$ weakly in $L^{\sob}(\Omega)^{d\times d}$, and assume
    \begin{equation}\label{eq:cons.ch:s:strict}
      \conv < \frac{\sob^*}{\sob'}.
    \end{equation}    
     Then, for all $\b\phi \in C_\cst{c}^\infty(\Omega,\mathbb R^d)$, it holds, up to a subsequence,
      \begin{equation}\label{eq:ch:sequential.consistency}
        \cst{c}_h(\bu v_h, \bI{h}{k}\b\phi) \CV{h}{0} c(\b v,\b\phi).
      \end{equation}
  \end{enumerate}
\end{lemma}

\begin{proof}
  The non-dissipativity \eqref{eq:ch:non-dissip} of $\cst{c}_h$ is an immediate consequence of its definition \eqref{eq:ch}.
  The proof of the H\"older continuity \eqref{eq:ch:holder.cont} is analogous to that of the corresponding property \eqref{eq:c:holder.cont} for $c$, replacing the relevant continuous Sobolev embeddings (see Remark \ref{rem:cont.c:s} on the role of the condition $s\le\frac{m^*}{m'}$) with their discrete counterpart \eqref{eq:discrete.sob.emb}, and leveraging the norm equivalence \eqref{eq:norm.equiv}.
  Properties \eqref{eq:ch:consistency} and \eqref{eq:ch:sequential.consistency} are proved in Section \ref{sec:ch:consistency} below.
\end{proof}

\subsection{Viscous term}

For all $T \in \T_h$, we define the local symmetric gradient reconstruction $\dgrads{k}{T} : \bdU{T}{k} \to \Poly^{k}(T,\Ms{d})$ by setting, for all $\bu v_T \in \bdU{T}{k}$,
\begin{equation}\label{eq:GST}
  \dgrads{k}{T}\bu v_T \coloneq \frac12\left[\dgrad{k}{T}\bu v_T + (\dgrad{k}{T}\bu v_T)^\top\right].
\end{equation}
Similarly, the global symmetric gradient reconstruction $\dgrads{k}{h} : \bdU{h}{k} \to \Poly^{k}(\T_h,\Ms{d})$ is obtained setting, for all $\bu v_h \in \bdU{h}{k}$,
\begin{equation}\label{eq:Gsh}
  \dgrads{k}{h} \bu v_h \coloneq \frac12\left[\dgrad{k}{h}\bu v_h + (\dgrad{k}{h}\bu v_h)^\top\right].
\end{equation}

The discrete counterpart of the function $a$ defined in \eqref{eq:weak:ab} is the function $\cst{a}_h : \bdU{h}{k} \times \bdU{h}{k} \to \R$ such that, for all $\bu v_h,\bu w_h \in \bdU{h}{k}$,
\begin{equation}\label{eq:ah}
  \cst{a}_h(\bu w_h, \bu v_h) \coloneqq \displaystyle\int_\Omega \stress(\cdot,\dgrads{k}{h} \bu w_h): \dgrads{k}{h} \bu v_h + \cst{s}_h(\bu w_h,\bu v_h).
\end{equation}
Taking inspiration from \cite{Botti.Castanon-Quiroz.ea:20} and \cite{Di-Pietro.Droniou.ea:21}, we take the stabilization function $\cst{s}_h : \bdU{h}{k} \times \bdU{h}{k} \to \R$ such that
\begin{equation}\label{eq:sh}
  \cst{s}_h(\bu w_h,\bu v_h) \coloneqq \frac{\sigma_\cst{hc}+\sigma_\cst{hm}}{2}\sum_{T\in\T_h} h_T\int_{\partial T} \left(\delta^r+|\bdfbres{k}{\partial T} \bu w_T|^r\right)^\frac{\sob-2}{r}\bdfbres{k}{\partial T} \bu w_T \cdot \bdfbres{k}{\partial T} \bu v_T,
\end{equation}
where, for all $T \in \T_h$, the boundary residual operator $\bdfbres{k}{\partial T} : \bdU{T}{k} \to L^r(\partial T)^d$ is such that, for all $\bu v_T \in \bdU{T}{k}$,
\[
(\bdfbres{k}{\partial T} \bu v_T)\res{F} =
\frac{1}{h_T}\left[
  \PROJ{F}{k}(\bdrec{k+1}{T} \bu v_T-\b v_F)-\PROJ{T}{k}(\bdrec{k+1}{T} \bu v_T-\b v_T)
  \right]
\qquad\forall F\in\F_T,
\]
with $\bdrec{k+1}{T} : \bdU{T}{k} \to \Poly^{k+1}(T)^d$ velocity reconstruction consistent for polynomials of degree $\le k+1$ (see \cite[Section 4.1.3]{Botti.Castanon-Quiroz.ea:20} for one possible definition).
With this choice, it holds (see, e.g., \cite[Lemma 8]{Botti.Castanon-Quiroz.ea:20}):
\begin{equation}\label{eq:drec:polynomial.consistency}
  \bdfbres{k}{\partial T}(\bI{T}{k}\b v) = \b 0\qquad\forall\b v\in\Poly^{k+1}(T)^d.
\end{equation}
We define the corresponding boundary residual seminorm $|{\cdot}|_{\sob,h}$ such that, for all $\bu v_h \in \bdU{h}{k}$,
\begin{equation}\label{eq:bound.res.seminorm}
|\bu v_h|_{\sob,h} \coloneq \left(\sum_{T\in\T_h}h_T\Vert \bdfbres{k}{\partial T}\bu v_{T} \Vert_{L^r(\partial T)^d}^r\right)^\frac{1}{r}.
\end{equation}
For future use, we note the following local uniform seminorm equivalence:
\begin{equation}\label{eq:bound.res:stability.boundedness:local}
  \forall T\in\T_h\qquad
  \| \dgrad{k}{T}\bu v_T \|_{L^r(T)^{d\times d}}^r
  + h_T\Vert \bdfbres{k}{\partial T}\bu v_{T} \Vert_{L^r(\partial T)^d}^r
  \simeq \|\bu v_T\|_{\strain,r,T}^r
  \qquad\forall\bu v_T\in\bu U_T^k,
\end{equation}
which, summed over $T\in\T_h$, gives
\begin{equation}\label{eq:bound.res:stability.boundedness}
  \| \dgrad{k}{h}\bu v_h \|_{L^r(\Omega)^{d\times d}}^r + |\bu v_h|_{r,h}^r
  \simeq \|\bu v_h\|_{\strain,r,h}^r
  \qquad\forall\bu v_h\in\bu U_h^k.
\end{equation}
    
\begin{lemma}[Properties of $\cst{s}_h$]

  Under Assumption \ref{ass:stress}, we have the following properties for $\cst{s}_h$:
  \begin{enumerate}[left=0pt]
  \item  \emph{H\"older continuity.}  For all $\bu u_h, \bu v_h, \bu w_h \in \bdU{h}{k}$, it holds
    \begin{align}\label{eq:sh:holder.cont}
      \left|\cst{s}_h(\bu u_h,\bu v_h)-\cst{s}_h(\bu w_h,\bu v_h)
      \right|
      &\lesssim \sigma_\cst{hc}\left( \delta^r+  |\bu u_h|_{\sob,h}^\sob+|\bu w_h|_{\sob,h}^\sob\right)^\frac{\sob-\sobs}{\sob}|\bu u_h - \bu w_h|_{\sob,h}^{\sobs-1}|\bu v_h|_{\sob,h}.
    \end{align}
  \item \emph{H\"older monotonicity.}
    For all $\bu u_h, \bu w_h \in \bdU{h}{k}$, it holds
    \begin{align}\label{eq:sh:strong.mono}
      \left(
      \cst{s}_h(\bu u_h,\bu u_h - \bu w_h)
      - \cst{s}_h(\bu w_h,\bu u_h - \bu w_h)
      \right)\left(\delta^r+|\bu u_h|_{\sob,h}^\sob+|\bu w_h|_{\sob,h}^\sob\right)^\frac{2-\sobs}{\sob} &\gtrsim \sigma_\cst{hm}|\bu u_h - \bu w_h|_{\sob,h}^{\sob+2-\sobs}.
    \end{align}  
    \item \emph{Sequential consistency.} Let $(\bu v_h)_{h\in\mathcal H}$ denote a bounded sequence of $(\bdU{h,0}{k},\Vert \cdot \Vert_{\strain,\sob,h})_{h \in \mathcal H}$. Then, for all $\b\phi \in C_\cst{c}^\infty(\Omega)^d$,
    \begin{align}\label{eq:sh:sequential.consistency}
      \cst{s}_h(\bu v_h, \bI{h}{k}\b\phi)& \CV{h}{0} 0.
    \end{align}
  \end{enumerate}
\end{lemma}

\begin{proof} 
  Properties \eqref{eq:sh:holder.cont}--\eqref{eq:sh:strong.mono} can be proved reasoning as in \cite{Botti.Castanon-Quiroz.ea:20} and proceeding as in \cite{Di-Pietro.Droniou.ea:21} for the addition of $\delta$.
  It remains to prove \eqref{eq:sh:sequential.consistency}. Using the H\"older continuity \eqref{eq:sh:holder.cont} of $\cst{s}_h$ with $(\bu u_h, \bu v_h, \bu w_h) = (\bu v_h, \bI{h}{k}\b \phi, \bu 0)$, we infer
  \begin{equation}\label{eq:sh:sequential.consistency:0}
    |\cst{s}_{h}(\bu v_{h},\bI{h}{k} \b\phi)| 
    \lesssim \sigma_\cst{hc}\left( \delta^r+  |\bu v_h|_{\sob,h}^\sob+\right)^\frac{\sob-\sobs}{\sob}|\bu v_h|_{\sob,h}^{\sobs-1}|\bI{h}{k}\b\phi|_{\sob,h}.
    %% \left(\sum_{T\in\T_h} h_T\|\bdfbres{k}{\partial T}(\bI{T}{k}(\b\phi-\PROJ{T}{k+1}\b\phi))\|_{L^r(\partial T)^d}^r\right)^\frac{1}{r}\\
    %% \lesssim \sigma_\cst{hc}
    %%   \left(\delta^r+\Vert \bu v_h \Vert_{\strain,\sob,h}^r\right)^\frac{\sob-\sobs}{r}\Vert \bu v_h \Vert_{\strain,\sob,h}^{\sobs-1}|\b\phi-\PROJ{h}{k+1}\b\phi|_{W^{1,r}(\Omega)^d},
  \end{equation}
  Recalling the definition \eqref{eq:bound.res.seminorm} of the boundary residual seminorm, we get
  \[
  \begin{aligned}
    |\bI{h}{k}\b\phi|_{\sob,h}^r
    &= \sum_{T\in\T_h}h_T
    \Vert \bdfbres{k}{\partial T}(\bI{h}{k}\b\phi) \Vert_{L^r(\partial T)^d}^r
    \\
    &= \sum_{T\in\T_h} h_T\|\bdfbres{k}{\partial T}[\bI{T}{k}(\b\phi-\PROJ{T}{k+1}\b\phi)]\|_{L^r(\partial T)^d}^r
    \\
    &\lesssim \sum_{T\in\T_h} \|\bI{T}{k}(\b\phi-\PROJ{T}{k+1}\b\phi)\|_{\strain,\sob,T}^r
    \lesssim |\b\phi-\PROJ{h}{k+1}\b\phi|_{W^{1,r}(\Omega)^d}^r,
  \end{aligned}
  \]
  where we have used the polynomial consistency \eqref{eq:drec:polynomial.consistency} of the boundary residual to insert $\PROJ{T}{k+1}\b\phi$ in the second line,
  the local seminorm equivalence \eqref{eq:bound.res:stability.boundedness:local} to pass to the third line,
  and the boundedness \eqref{eq:I:boundedness} of $\bI{T}{k}$ to conclude.
  Plugging this bound into \eqref{eq:sh:sequential.consistency:0} and using \eqref{eq:proj.seq.cons} along with the boundedness of $(\Vert \bu v_{h} \Vert_{\strain,\sob,h})_{h \in \mathcal H}$ (which implies that of $(|\bu v_h|_{\sob,h})_{h\in\mathcal H}$ by virtue of \eqref{eq:bound.res:stability.boundedness}) yields \eqref{eq:sh:sequential.consistency}.
\end{proof}
\begin{lemma}[Properties of $\cst{a}_h$]

  Under Assumption \ref{ass:stress}, we have the following properties for $\cst{a}_h$:
  \begin{enumerate}[left=0pt]
  \item  \emph{H\"older continuity.}  For all $\bu u_h, \bu v_h, \bu w_h \in \bdU{h}{k}$, it holds
    \begin{align}\label{eq:ah:holder.cont}
      \left|\cst{a}_h(\bu u_h,\bu v_h)-\cst{a}_h(\bu w_h,\bu v_h)
      \right|
      &\lesssim \sigma_\cst{hc}\left( \delta^r+  \| \bu u_h \|_{\strain,\sob,h}^\sob+\| \bu w_h \|_{\strain,\sob,h}^\sob\right)^\frac{\sob-\sobs}{\sob}\| \bu u_h - \bu w_h \|_{\strain,\sob,h}^{\sobs-1}\| \bu v_h \|_{\strain,\sob,h}.
    \end{align}
  \item \emph{H\"older monotonicity.} 
   For $m \in \{\sobs,\sob\}$ and all $\bu u_h, \bu w_h \in \bdU{h}{k}$, it holds, with $C_\cst{da} > 0$ independent of $h$,
    \begin{multline}\label{eq:ah:strong.mono}
     \cst{a}_h(\bu u_h,\bu u_h - \bu w_h)-\cst{a}_h(\bu w_h,\bu u_h - \bu w_h) \\
       \ge C_\cst{da}\sigma_\cst{hm}\left(\delta^r+\| \bu u_h \|_{\strain,\sob,h}^\sob+\| \bu w_h \|_{\strain,\sob,h}^\sob\right)^\frac{\sobs-2}{\sob}\delta^{r-m}\| \bu u_h - \bu w_h \|_{\strain,m,h}^{m+2-\sobs}.
    \end{multline}
  \item \emph{Consistency.} Let $\b w \in \b U \cap W^{k+2,\sob}(\T_h)^d$ be such that $\stress(\cdot,\GRADs \b w) \in W^{1,\sob'}(\Omega)^{d\times d} \cap W^{k+1,\sob'}(\T_h)^{d\times d}$. Then,
    \begin{multline}\label{eq:ah:consistency}
      \sup_{\bu v_h \in \bdU{h,0}{k},\| \bu v_h \|_{\strain,\sob,h} = 1} \left|\int_\Omega (\DIV \stress(\cdot,\GRADs \b w)) \cdot \b v_h + \cst{a}_h(\bI{h}{k} \b w,\bu v_h)\right| 
      \lesssim h^{k+1}|\stress(\cdot,\GRADs \b w)|_{W^{k+1,\sob'}(\T_h)^{d\times d}} \\
      + h^{(k+1)(\sobs-1)}\min\left(\zeta_h(\b w);1\right)^{2-\sobs}
      \sigma_\cst{hc}\left(\delta^\sob + |\b w|_{W^{1,\sob}(\Omega)^d}^\sob\right)^\frac{\sob-\sobs}{\sob}|\b w|_{W^{k+2,\sob}(\T_h)^d}^{\sobs-1},
    \end{multline}
    where $\zeta_h(\b w) \coloneq \delta^{-1}\max_{T\in\T_h}\big(|T|^{-\frac{1}{p}}|\b w|_{W^{k+2,\sob}(T)^d}\big)h^{k+1}$ if $\delta \ne 0$, 
$\zeta_h(\b w) \coloneq \infty \text{ otherwise}$.
  \end{enumerate}
\end{lemma}
 
\begin{proof} 
  Properties \eqref{eq:ah:holder.cont} is proved in \cite{Botti.Castanon-Quiroz.ea:20} (with \eqref{eq:sh:holder.cont} replacing \cite[Eq. (41b)]{Botti.Castanon-Quiroz.ea:20}).
Similarly, \eqref{eq:ah:strong.mono} is shown replacing \cite[Eq. (41c)]{Botti.Castanon-Quiroz.ea:20} by \eqref{eq:sh:strong.mono} when $r \le 2$, and the proof of the case $r > 2$ is analogous to that of the corresponding property \eqref{eq:a:strong.mono} for $a$, replacing the continuous Korn inequality \eqref{eq:Korn} by its discrete counterpart \eqref{eq:discrete.Korn}.
  Finally, \eqref{eq:ah:consistency} is obtained modifying the reasoning of \cite{Botti.Castanon-Quiroz.ea:20} according to \cite[Theorem 10]{Di-Pietro.Droniou.ea:21} in order to introduce the term involving $\zeta_h(\b w)$.
\end{proof}

\subsection{Pressure-velocity coupling}

We define the global divergence reconstruction $\ddiv{k}{h} : \bdU{h}{k} \to \Poly^{k}(\T_h)$ by setting for all $\bu v_h \in \bdU{h}{k}$,
\begin{equation}\label{eq:Dh}
  \ddiv{k}{h}\bu v_{h} \coloneq \cst{tr}(\dgrad{k}{h}\bu v_{h}).
\end{equation}
The pressure-velocity coupling is realized by the bilinear form $\cst{b}_h : \bdU{h}{k} \times \Poly^k(\T_h) \to \R$ such that, for all $(\bu v_h,q_h) \in \bdU{h}{k} \times \Poly^k(\T_h)$, 
\begin{equation}\label{eq:bh}
  \cst{b}_h(\bu v_h,q_h) \coloneqq -\int_\Omega\ddiv{k}{h} \bu v_h~ q_h.
\end{equation}

\begin{lemma}[Properties of $\cst{b}_h$]
  We have the following properties for $\cst{b}_h$:
  \begin{enumerate}[left=0pt]
  \item \emph{Inf-sup stability.}\label{bh:inf-sup}   It holds, for all $q_h \in \dP{h}{k}$,
    \begin{equation}\label{eq:bh:inf-sup}
      \| q_h \|_{L^{\sob'}(\Omega)} \lesssim \sup\limits_{\bu v_h \in \bdU{h,0}{k}, \| \bu v_h \|_{\strain,\sob,h} = 1} \cst{b}_h(\bu v_h,q_h).
    \end{equation}
  \item \emph{Fortin operator.} For all $\b v \in W^{1,\sob}(\Omega)^d$ and all $q_h \in \Poly^k(\T_h)$,
    \begin{equation}\label{eq:bh:fortin}
      \cst{b}_h(\bI{h}{k}\b v,q_h) = b(\b v,q_h).
    \end{equation}
  \item\emph{Consistency.} For all $q \in W^{1,\sob'}(\Omega) \cap W^{k+1,\sob'}(\T_h)$,
    \begin{equation}\label{eq:bh:consistency}
      \sup\limits_{\bu v_h \in \bdU{h,0}{k},\| \bu v_h \|_{\strain,\sob,h} = 1} \left|\int_\Omega \GRAD q \cdot \b v_h - \cst{b}_h(\bu v_h,\proj{h}{k} q)\right| \lesssim h^{k+1} |q|_{W^{k+1,\sob'}(\T_h)}.
    \end{equation}
  \item \emph{Sequential consistency/1.}
    Let $(q_h)_{h \in \mathcal H} \in (\dP{h}{k})_{h\in\mathcal H}$ be such that $q_h \CV{h}{0} q \in P$ weakly in $L^{\sob'}(\Omega)$. Then, for all $\b\phi \in C_\cst{c}^\infty(\Omega)^d$, it holds
    \begin{equation}\label{eq:bh:first.sequential.consistency}
      \cst{b}_h(\bI{h}{k} \b\phi, q_h) \CV{h}{0} b(\b\phi,q).
    \end{equation}
  \item \emph{Sequential consistency/2.} Let $(\bu v_h)_{h\in\mathcal H} \in (\bdU{h,0}{k})_{h\in\mathcal H}$ be such that $\ddiv{k}{h} \bu v_h\CV{h}{0}\DIV\b v$ weakly in $L^{\sob}(\Omega)$. Then, for all $\psi \in C_\cst{c}^\infty(\Omega)$, it holds
    \begin{equation}\label{eq:bh:second.sequential.consistency}
      \cst{b}_h(\bu v_h,\proj{h}{k} \psi) \CV{h}{0} b(\b v,\psi).
    \end{equation}
  \end{enumerate}
\end{lemma}
 
\begin{proof}
  Properties \eqref{eq:bh:inf-sup}--\eqref{eq:bh:consistency} are proved in \cite{Botti.Castanon-Quiroz.ea:20}.
  Let us prove \eqref{eq:bh:first.sequential.consistency}.
  Given $\b\phi \in C_\cst{c}^\infty(\Omega)^d$, \eqref{eq:GT.seq.cons} combined with the definition \eqref{eq:Dh} of $\ddiv{k}{h}$ yields $\ddiv{k}{h} (\bI{h}{k} \b\phi)\CV{h}{0} \div\b\phi$ strongly in $L^\sob(\Omega)$. Hence,
  $
  \cst{b}_h(\bI{h}{k} \b\phi, q_h) = -\int_\Omega \ddiv{k}{h} (\bI{h}{k} \b\phi)\ q_h \CV{h}{0} -\int_\Omega (\div \b\phi)\ q = b(\b\phi,q).
  $
  Let now $\psi \in C_\cst{c}^\infty(\Omega)$. Combining the fact that $\ddiv{k}{h} \bu v_h \CV{h}{0} \div \b v$ weakly in $L^\sob(\Omega)$ by assumption with \eqref{eq:proj.seq.cons}, we obtain \eqref{eq:bh:second.sequential.consistency} by writing
  $
  \cst{b}_h(\bu v_h, \proj{h}{k} \psi)
  = -\int_\Omega \ddiv{k}{h} \bu v_h\ \proj{h}{k} \psi \CV{h}{0} -\int_\Omega (\div\b v)\ \psi = b(v,\psi).
  $
\end{proof}

\subsection{Discrete problem and main results}

The discrete problem reads: Find $(\bu u_h,p_h) \in \bdU{h,0}{k} \times \dP{h}{k}$ such that
\begin{subequations}\label{eq:ns.discrete}
  \begin{alignat}{2}
    \label{eq:ns.discrete:momentum} \cst{a}_h(\bu u_h,\bu v_h) + \cst{c}_h(\bu u_h,\bu v_h) + \cst{b}_h(\bu v_h,p_h) &= \displaystyle\int_\Omega \b f \cdot \b v_h &\qquad& \forall \bu v_h \in \bdU{h,0}{k}, \\
    \label{eq:ns.discrete:mass} -\cst{b}_h(\bu u_h,q_h) &= 0 &\qquad& \forall q_h \in \dP{h}{k}. 
  \end{alignat}
\end{subequations}

The following theorem states the existence of a discrete solution to problem \eqref{eq:ns.discrete} and provide conditions for uniqueness.

\begin{theorem}[Existence and uniqueness for problem \eqref{eq:ns.discrete}]\label{thm:discrete.well-posedness}
Under Assumptions \ref{ass:stress} and \ref{ass:conv}, there exists a solution $(\bu u_h,p_h) \in \bdU{h,0}{k} \times \dP{h}{k}$ to the discrete problem \eqref{eq:ns.discrete}, and any solution satisfies
\begin{subequations}\label{thm:discrete.well-posedness:bounds}
    \begin{align}
      \| \bu u_h \|_{\strain,\sob,h} &\le C_\cst{dv}\left[
        \left(\sigma_\cst{hm}^{-1}\| \b f \|_{L^{\sob'}(\Omega)^d}\right)^{r'}+\min\left(\delta^r;\left(\delta^{2-\sobs}\sigma_\cst{hm}^{-1}\| \b f \|_{L^{\sob'}(\Omega)^d}\right)^\frac{r}{\sob+1-\sobs}\right)
        \right]^\frac{1}{r}, \label{eq:discrete.well-posedness:bounds:uh}\\
      \| p_h \|_{L^{\sob'}(\Omega)} &\lesssim \sigma_\cst{hc}\left[
        \sigma_\cst{hm}^{-1}\| \b f \|_{L^{\sob'}(\Omega)^d}+\delta^{|\sob-2|(\sobs-1)}\left(\sigma_\cst{hm}^{-1}\| \b f \|_{L^{\sob'}(\Omega)^d}\right)^\frac{\sobs-1}{\sob+1-\sobs}
        \right], \label{eq:discrete.well-posedness:bounds:ph}
    \end{align}
  \end{subequations}
  where $C_\cst{dv} > 0$ is independent of $h$.
  Moreover, assuming $2 \le s \le \frac{\sobs^*}{\sobs'}$ (cf. \eqref{eq:intervals.s:cont.uniqueness}) and that the following data smallness condition holds:
   \begin{equation}\label{eq:discrete.well-posedness:f}
  \delta^r+\left(\sigma_\cst{hm}^{-1}\| \b f \|_{L^{\sob'}(\Omega)^d}\right)^{r'} < \left(1+2C_\cst{dv}^r\right)^{-1}\left(C_{\cst{dc},\sobs}^{-1}\chi_\cst{hc}^{-1}C_\cst{da}\sigma_\cst{hm}\delta^{r-\sobs}\right)^{\frac{r}{s+1-\sobs}},  
  \end{equation}
   the solution of \eqref{eq:ns.discrete} is unique.
\end{theorem}

\begin{proof}
  Replacing $\cst{a}_h$ by $\cst{a}_h+\cst{c}_h$ in the proof of \cite[Theorem 11]{Botti.Castanon-Quiroz.ea:20} and using the non-dissipativity \eqref{eq:ch:non-dissip} of $\cst{c}_h$, yields the existence of a solution to problem \eqref{eq:ns.discrete} and the a priori estimates \eqref{thm:discrete.well-posedness:bounds}, noticing that the H\"older monotonicity \eqref{eq:ah:strong.mono} of $\cst{a}_h$ with $m=\sob$ is the key property leveraged in the proof.
  Uniqueness of the solution under the above assumptions on $s$ and $r$ and the data smallness condition \eqref{eq:discrete.well-posedness:f} can be proved as its continuous counterpart in Theorem \ref{thm:well-posedness} leveraging the inf-sup stability \eqref{eq:bh:inf-sup} of $\cst{b}_h$, the H\"older monotonicity \eqref{eq:ah:strong.mono} of $\cst{a}_h$ with $m=\sobs$, and the H\"older continuity \eqref{eq:ch:holder.cont} of $\cst{c}_h$.
\end{proof}

We next state convergence results and error estimates.

\begin{theorem}[Convergence to minimal regularity solutions]\label{thm:convergence}
  Let ${((\bu u_h,p_h))_{h \in \mathcal H}}$ be a sequence of \\
   ${(\bdU{h,0}{k} \times \dP{h}{k})_{h \in \mathcal H}}$ such that, for all $h \in \mathcal H$, $(\bu u_h,p_h)$ solves \eqref{eq:ns.discrete}. Assume \eqref{eq:cons.ch:s:strict}, namely $\conv < \frac{\sob^*}{\sob'}$, i.e.,
    \begin{equation}\label{eq:intervals.s:convergence}
    s\in\begin{cases}
      \big(1,\frac{d(r-1)}{d-r}\big) & \text{if $d=2$ and $r\in(1,2)$ or $d=3$ and $r\in(1,3)$,}\\
      (1,\infty) & \text{if $d=2$ and $r\in[2,\infty)$ or $d=3$ and $r\in[3,\infty)$.}
    \end{cases}
  \end{equation}
  Then, under Assumptions \ref{ass:stress} and  \ref{ass:conv}, there exists $(\b u,p) \in \b U \times P$ solving  \eqref{eq:ns.weak} such that up to a subsequence,
\begin{itemize}[parsep=0pt,noitemsep]
\item \label{thm:convergence:velocity} $\b u_h \CV{h}{0} \b u$ strongly in $L^{[1,\sob^*)}(\Omega)^d$;
\item \label{thm:convergence:grad} $\dgrads{k}{h} \bu u_h \CV{h}{0} \GRADs \b u$ strongly in $L^\sob(\Omega)^{d\times d}$;
\item \label{thm:convergence:sh} $|\bu u_h|_{r,h}  \CV{h}{0} 0$;
\item \label{thm:convergence:pressure} $p_h \CV{h}{0} p$ strongly in $L^{\sob'}(\Omega)$.
\end{itemize}
Moreover, if the solution to \eqref{eq:ns.weak} is unique (cf. Theorem \ref{thm:well-posedness}), the convergences
extend to the whole sequence.
\end{theorem}

\begin{proof}
  See Section \ref{proof:thm:convergence}.
\end{proof}

\begin{theorem}[Error estimate]\label{thm:error.estimate}
  Assume $\sob \le 2 \le \conv \le \frac{\sob^*}{\sob'}$ so that, in particular, the fluid is shear-thinning and
  \begin{equation}\label{eq:intervals.s:error.estimate}
    s\in\begin{cases}
    \big[2,\frac{d(r-1)}{d-r}\big] & \text{if $d=2$ and $r\in\big[\frac32,2\big)$ or $d=3$ and $r\in\big[\frac95,2\big]$,}
      \\
        [2,\infty) & \text{if $d=2$ and $r=2$}.
    \end{cases}
  \end{equation}
  Let $(\b u,p) \in \b U \times P$ and $(\bu u_h,p_h) \in \bdU{h,0}{k} \times \dP{h}{k}$ solve \eqref{eq:ns.weak} and \eqref{eq:ns.discrete}, respectively.  
  Assume the uniqueness of such solutions (which is verified, under \eqref{eq:intervals.s:error.estimate}, if the data smallness conditions \eqref{eq:well-posedness:f} and \eqref{eq:discrete.well-posedness:f} hold), and the additional regularity $\b u \in W^{k+2,\sob}(\T_h)^d \cap W^{k+1,s\sob'}(\T_h)^d$ (so that, in particular, $\b u\in W^{1,s\sob'}(\Omega)^d$), $\stress(\cdot,\GRADs \b u) \in W^{1,\sob'}(\Omega)^{d\times d} \cap W^{k+1,\sob'}(\T_h)^{d\times d}$, and $p \in W^{1,\sob'}(\Omega) \cap W^{k+1,\sob'}(\T_h)$.
  Let, furthermore, the following data smallness condition be verified:
  \begin{equation}\label{eq:small.f}
    \begin{aligned}
      \mathcal N_1 \coloneqq \delta^\sob+
      \left(\sigma_\cst{hm}^{-1}\| \b f \|_{L^{\sob'}(\Omega)^d}\right)^{r'} \le \left(1+C_\cst{v}^r+C_\cst{I}^rC_\cst{dv}^r\right)^{-1}\left(\frac{C_\cst{da}\sigma_\cst{hm}}{2C_{\cst{dc},\sob}\chi_\cst{hc}}\right)^\frac{r}{\conv+1-\sob}.
    \end{aligned}
  \end{equation}
  Then, under Assumptions \ref{ass:stress} and \ref{ass:conv}, it holds
  \begin{subequations}\label{eq:error.estimate}
    \begin{align}\label{eq:error.estimate:velocity}
      \| \bu u_h - \bI{h}{k} \b u \|_{\strain,\sob,h}
      &\lesssim 
      \sigma_\cst{hm}^{-1}\mathcal N_1^\frac{2-r}{r}\left(h^{(k+1)(\sob-1)}\min\left(\zeta_h(\b u);1\right)^{2-\sob} \mathcal N_2+h^{k+1} \mathcal N_3\right),
      \\
      \label{eq:error.estimate:pressure}
      \| p_h - \proj{h}{k} p \|_{L^{\sob'}(\Omega)}
      &\lesssim
      \sigma_\cst{hc}    
      \sigma_\cst{hm}^{1-r}\mathcal N_1^\frac{2-r}{r'}\left(h^{(k+1)(\sob-1)^2}\min\left(\zeta_h(\b u);1\right)^{(2-\sob)(r-1)} \mathcal N_2^{\sob-1}+h^{(k+1)(r-1)} \mathcal N_3^{\sob-1}\right) \\
      &\quad + \left(1+\chi_\cst{hc}\sigma_\cst{hm}^{-1}\mathcal N_1^\frac{s+1-r}{r}\right)\left(h^{(k+1)(\sob-1)}\min\left(\zeta_h(\b u);1\right)^{2-\sob} \mathcal N_2+h^{k+1} \mathcal N_3\right)\nonumber,
    \end{align}
  \end{subequations}
  with $\zeta_h$ introduced in \eqref{eq:ah:consistency} and where we have set, for the sake of brevity,
  \begin{gather*}
    \mathcal N_2
    \coloneq
    \sigma_\cst{hc}|\b u|_{W^{k+2,\sob}(\T_h)^d}^{\sob-1},\\
    \mathcal N_3
    \coloneq\begin{aligned}[t]
    &|\stress(\cdot,\GRADs \b u)|_{W^{k+1,\sob'}(\T_h)^{d\times d}}+ |p|_{W^{(k+1)(\sob-1),\sob'}(\T_h)} \\
    &\quad
    +|\b u|_{W^{k+1,s\sob'}(\T_h)^d}^s+\left(
    |\b u|_{W^{1,s\sob'}(\Omega)^d}^s
    + |\b u|_{W^{1,\sob}(\Omega)^d}^s
    \right)^\frac{\conv-1}{s}\left(
    |\b u|_{W^{k+1,s\sob'}(\T_h)^d} + |\b u|_{W^{k+2,\sob}(\T_h)^d}
    \right).
    \end{aligned}
  \end{gather*}
\end{theorem}

\begin{proof}
  See Section \ref{proof:thm:error.estimate}.
\end{proof}

\begin{remark}[Orders of convergence]\label{rem:ocv}
  From \eqref{eq:error.estimate}, neglecting higher-order terms, we infer asymptotic convergence rates of $\mathcal O_\cst{vel}^k \in [(k+1)(\sob-1),k+1]$ for the velocity, and $\mathcal O_\cst{pre}^k \in [(k+1)(\sob-1)^2,(k+1)(\sob-1)]$ for the pressure, according to the dimensionless number $\zeta_h(\b u)$.
  Notice that, owing to the presence of higher-order terms in the right-hand sides of \eqref{eq:error.estimate}, higher convergence rates may be observed in practice before attaining the asymptotic ones.
  
  In the case $s=2$, the error estimate given in \cite[Theorem 3.1]{Hirn:13} for the approximation of the $p$-Stokes equations with conforming finite elements (to be compared with the case $k=0$ in the present work) gives an order of convergence of the velocity coinciding with our upper bound irrespectively of the degeneracy of the problem.
  This difference with respect to conforming finite elements had already been observed in the context of the $p$-Laplacian, cf. \cite[Remark 3.3]{Di-Pietro.Droniou:17*1}, with improvements on the original HHO estimate recently made in \cite{Di-Pietro.Droniou.ea:21}. On the other hand, the order of convergence for the pressure given by \cite[Theorem 3.1]{Hirn:13} seems higher than the one derived in the present work.
  This point will make the object of future investigations.
\end{remark}

%------------------------------------------------------------------------------%
%------------------------------------------------------------------------------%

\section{Numerical examples}\label{sec:num.res}

The method \eqref{eq:ns.discrete} was implemented within the SpaFEDTe library (cf. \url{https://spafedte.github.io}). 
We used a Picard method for the solution of the nonlinear algebraic problem corresponding to the HHO discretization with a tolerance of $10^{-10}$. The linear systems at each iteration were solved using the sparse direct solver PardisoLU.
In this section we present a numerical validation including a verification of the convergence rates in dimensions $d=2$ and $d=3$, as well as the more physical two-dimensional lid-driven cavity problem.

\subsection{Numerical verification of the convergence rates}\label{sec:num.res.trigo}

We consider manufactured solutions of problem \eqref{eq:ns.continuous} with diffusion law corresponding to the $(1,1,\sob,\sob)$-Carreau--Yasuda model \eqref{eq:Carreau--Yasuda} and convection law given by the $(1,s)$-Laplace formula \eqref{eq:Laplace}.
The corresponding Sobolev exponents are the couples $(\sob,\conv)\in \left\{\frac32,\frac95,2,\frac52,3\right\}^2$ that match the condition $s \le \frac{r^*}{r'}$.
The volumetric load $\b f$ and the Dirichlet boundary conditions are inferred from the exact solution.
Polynomial degrees ranging from 1 to 3 are considered.
For each value of $d\in \{2,3\}$, we let $\Omega=(0,1)^d$, and consider the exact velocity $\b u$ and pressure $p$ such that, for all $\b x = (x_i)_{1\le i \le d} \in \Omega$,
\begin{equation}
  \b u(\b x) = \left(\prod_{j=1,j\ne i}^d
  \sin\left(\tfrac{\pi}{2}x_j\right)\right)_{1\le i \le d}\quad \text{and} \quad 
  p(\b x) = \prod_{i=1}^d\sin\left(\tfrac{\pi}{2} x_i\right)-\tfrac{2^d}{\pi^d}.
\end{equation}
We consider the HHO scheme on distorted triangular (if $d=2$) and cubic (if $d=3$) mesh families.
In Figure \ref{tab:num.res} we display detailed convergence results for the case $\frac{r^*}{r'} = 2$, i.e. $r = \frac{3}{2}$ if $d=2$ and $r = \frac{9}{5}$  if $d=3$.
Table \ref{table:OCV} collects the asymptotic convergence rates predicted by Theorem \ref{thm:error.estimate} with the interval for $s$ in which the assumptions of Theorem \ref{thm:error.estimate} hold; when the interval is empty, we display the asymptotic convergence rates given by \cite[Theorem 12]{Botti.Castanon-Quiroz.ea:20} for the (generalized) Stokes problem, i.e. the same convergence rates if $r \le 2$ and $O_\cst{vel}^k = O_\cst{pre}^k = \frac{k+1}{r-1}$ otherwise.
When $r < 2$, the convergence rates are expected over an interval depending on the degeneracy of the problem (cf. \cite[Theorem 11]{Di-Pietro.Droniou.ea:21}); since $\delta = 1$, the expected convergence rates should correspond to the maximum of these intervals.
In Tables \ref{table:2D} and \ref{table:3D} we provide an overview of the experimental convergence rates obtained for $d=2$ and $d=3$, respectively.
Overall, the results are in agreement with the theoretical predictions.
The expected asymptotic orders of convergence of the pressure are exceeded for $r < 2$, where the experimental convergence rates is closer to $k + 1$ (i.e., the same rate as the velocity). 
One explanation could be the partial nature of the error estimate \eqref{eq:error.estimate:pressure} involving only the $L^2$-orthogonal projection of the continuous solution. Should this behaviour be confirmed by further numerical evidence, it could suggest that the error estimate for the pressure can be improved.
When the assumption $r \le 2 \le s$ of Theorem \ref{thm:error.estimate} is not met, the convergence rates seem to coincide with the predictions of \cite[Theorem 12]{Botti.Castanon-Quiroz.ea:20} for the corresponding Stokes problem with no convective terms when $s \ge 2$.
However, in the case $s < 2$, we notice that the presence of the convective term seems to influence the convergence rates of the velocity and the pressure which are often lower than expected. This could be explained by the consistency \eqref{eq:ch:consistency} of $\cst{c}_h$, which involves a term in $h^{(k+1)(s-1)}$ when $s < 2$.
While the assumption $r \le 2 \le s$ seems necessary to obtain estimates of the convergence rates in the present setting, we do not exclude that it could be lifted using different techniques (notice that convergence is guaranteed by Theorem \ref{thm:convergence} under significantly milder assumptions).

\begin{table}
  \caption{Asymptotic convergence rates predicted by Theorem \ref{thm:error.estimate} (cf. Remark \ref{rem:ocv}) according to $r$ and $k$.
    We also indicate, for the sake of completeness, the interval for $s$ in which the assumptions of Theorem \ref{thm:error.estimate} hold; when the interval is empty, the asymptotic convergence rates correspond to the predictions of \cite[Theorem 12]{Botti.Castanon-Quiroz.ea:20} valid for the corresponding generalized Stokes problem.}
\label{table:OCV}
\begin{center}
\begin{tabular}{cc|cc|cc|cc|cc|cc}
\toprule
\multicolumn{2}{c|}{$r$} & \multicolumn{2}{c|}{$\frac32$} & \multicolumn{2}{c|}{$\frac95$} & \multicolumn{2}{c|}{2} & \multicolumn{2}{c|}{$\frac52$} & \multicolumn{2}{c}{3} \\[2pt]
\hline
\multirow{2}{*}{$s$} & $d=2$ & \multicolumn{2}{c|}{$2$} & \multicolumn{2}{c|}{$[2,8]$} & \multicolumn{2}{c|}{$[2,\infty)$}  & \multicolumn{2}{c|}{$\emptyset$}  & \multicolumn{2}{c}{$\emptyset$}  \\
  & $d=3$ & \multicolumn{2}{c|}{$\emptyset$} & \multicolumn{2}{c|}{$2$} & \multicolumn{2}{c|}{$[2,3]$}  & \multicolumn{2}{c|}{$\emptyset$}  & \multicolumn{2}{c}{$\emptyset$}  \\
\hline
\multicolumn{2}{c|}{$k$}  & $\mathcal O_\cst{vel}^k$ & $\mathcal O_\cst{pre}^k$ & $\mathcal O_\cst{vel}^k$ & $\mathcal O_\cst{pre}^k$ & $\mathcal O_\cst{vel}^k$ & $\mathcal O_\cst{pre}^k$ & $\mathcal O_\cst{vel}^k$ & $\mathcal O_\cst{pre}^k$ & $\mathcal O_\cst{vel}^k$ & $\mathcal O_\cst{pre}^k$   \\[2pt]
\hline
\multicolumn{2}{c|}{$1$}  & $[1,2]$  & $\big[\frac{1}{2},1\big]$ & $\big[\frac{8}{5},2\big]$  & $\big[\frac{32}{25},\frac{8}{5}\big]$ & $2$  & $2$ & $\frac{4}{3}$ & $\frac{4}{3}$ & $1$ & $1$ \\[2pt]
\hline
\multicolumn{2}{c|}{$2$}  & $\big[\frac{3}{2},3\big]$  & $\big[\frac{3}{4},\frac{3}{2}\big]$ & $\big[\frac{12}{5},3\big]$  & $\big[\frac{48}{25},\frac{12}{5}\big]$ & $3$  & $3$ & $2$  & $2$ & $\frac{3}{2}$  & $\frac{3}{2}$ \\[2pt]
\hline
\multicolumn{2}{c|}{$3$}  & $[2,4]$  & $[1,2]$  & $\big[\frac{16}{5},4\big]$  & $\big[\frac{64}{25},\frac{16}{5}\big]$ & $4$  & $4$ & $\frac{8}{3}$ & $\frac{8}{3}$ & $2$ & $2$ \\[2pt]
\bottomrule
\end{tabular}
\end{center}
\end{table}

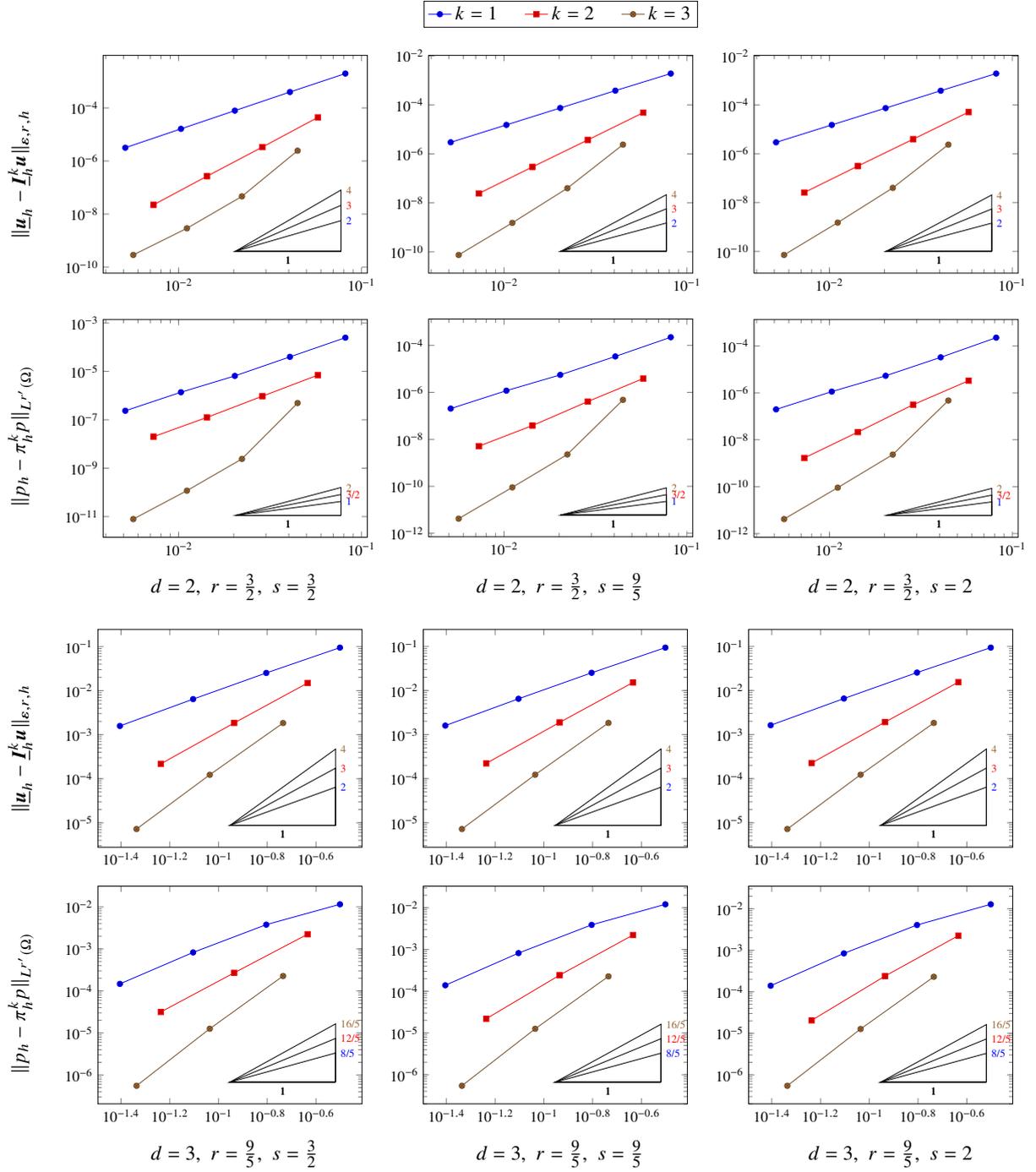
\begin{figure}
  \begin{center}
    %VITESSE
    \begin{minipage}[b]{0.04\columnwidth}
      \begin{tikzpicture}
      	\draw node[scale=0.8,rotate=90]{\hspace*{0.9cm} $\| \bu u_h - \bI{h}{k} \b u \|_{\strain,\sob,h}$};
      \end{tikzpicture}
    \end{minipage}
    \hfill
    \begin{minipage}[b]{0.3\columnwidth}
      \begin{tikzpicture}[scale=0.6]
        \begin{loglogaxis}
          \addplot table[x=meshsize,y=erru] {results/trigo2D/distTriMeshes/k1_del1_a150_r150_s150.dat};
          \addplot table[x=meshsize,y=erru] {results/trigo2D/distTriMeshes/k2_del1_a150_r150_s150.dat};
          \addplot table[x=meshsize,y=erru] {results/trigo2D/distTriMeshes/k3_del1_a150_r150_s150.dat};
          \logLogSlopeTriangle{0.90}{0.4}{0.1}{2}{blue};
          \logLogSlopeTriangle{0.90}{0.4}{0.1}{3}{red};
          \logLogSlopeTriangle{0.90}{0.4}{0.1}{4}{darkbrown};
        \end{loglogaxis}
      \end{tikzpicture}
    \end{minipage}
    \hfill
    \begin{minipage}[b]{0.30\columnwidth}
      \begin{tikzpicture}[scale=0.6]
        \begin{loglogaxis}[
            legend style = { 
              at={(0.5,1.25)},
              anchor = north,
              legend columns=-1,
              font={\Large}
            }
          ]
          \addplot table[x=meshsize,y=erru] {results/trigo2D/distTriMeshes/k1_del1_a150_r150_s180.dat};
          \addplot table[x=meshsize,y=erru] {results/trigo2D/distTriMeshes/k2_del1_a150_r150_s180.dat};
          \addplot table[x=meshsize,y=erru] {results/trigo2D/distTriMeshes/k3_del1_a150_r150_s180.dat};
          \logLogSlopeTriangle{0.90}{0.4}{0.1}{2}{blue};
          \logLogSlopeTriangle{0.90}{0.4}{0.1}{3}{red};
          \logLogSlopeTriangle{0.90}{0.4}{0.1}{4}{darkbrown};
          \legend{$k=1$ \quad\quad, $k=2$ \quad\quad, $k=3$};
        \end{loglogaxis}
      \end{tikzpicture}
    \end{minipage}
    \hfill
    \begin{minipage}[b]{0.31\columnwidth}
      \begin{tikzpicture}[scale=0.6]
        \begin{loglogaxis}
          \addplot table[x=meshsize,y=erru] {results/trigo2D/distTriMeshes/k1_del1_a150_r150_s200.dat};
          \addplot table[x=meshsize,y=erru] {results/trigo2D/distTriMeshes/k2_del1_a150_r150_s200.dat};
          \addplot table[x=meshsize,y=erru] {results/trigo2D/distTriMeshes/k3_del1_a150_r150_s200.dat};
          \logLogSlopeTriangle{0.90}{0.4}{0.1}{2}{blue};
          \logLogSlopeTriangle{0.90}{0.4}{0.1}{3}{red};
          \logLogSlopeTriangle{0.90}{0.4}{0.1}{4}{darkbrown};
        \end{loglogaxis}
      \end{tikzpicture}
    \end{minipage}
  \\
  \medskip
  	%PRESSION
    \begin{minipage}[b]{0.04\columnwidth}
      \begin{tikzpicture}
      	\draw node[scale=0.8,rotate=90]{\hspace*{1.7cm}$\| p_h - \proj{h}{k} p \|_{L^{\sob'}(\Omega)}$};
      \end{tikzpicture}
    \end{minipage}
    \hfill
    \begin{minipage}[b]{0.3\columnwidth}
      \begin{tikzpicture}[scale=0.6]
        \begin{loglogaxis}
          \addplot table[x=meshsize,y=errp] {results/trigo2D/distTriMeshes/k1_del1_a150_r150_s150.dat};
          \addplot table[x=meshsize,y=errp] {results/trigo2D/distTriMeshes/k2_del1_a150_r150_s150.dat};
          \addplot table[x=meshsize,y=errp] {results/trigo2D/distTriMeshes/k3_del1_a150_r150_s150.dat};
          \logLogSlopeTriangle{0.90}{0.4}{0.1}{1}{blue};
          \logLogSlopeTriangle{0.90}{0.4}{0.1}{3/2}{red};
          \logLogSlopeTriangle{0.90}{0.4}{0.1}{2}{darkbrown};
        \end{loglogaxis}
      \end{tikzpicture}
      \subcaption*{\quad\quad $d=2$,\ \ $r = \frac{3}{2}$,\ \ $s=\frac{3}{2}$}
    \end{minipage}
    \hfill
    \begin{minipage}[b]{0.3\columnwidth}
      \begin{tikzpicture}[scale=0.6]
        \begin{loglogaxis}
          \addplot table[x=meshsize,y=errp] {results/trigo2D/distTriMeshes/k1_del1_a150_r150_s180.dat};
          \addplot table[x=meshsize,y=errp] {results/trigo2D/distTriMeshes/k2_del1_a150_r150_s180.dat};
          \addplot table[x=meshsize,y=errp] {results/trigo2D/distTriMeshes/k3_del1_a150_r150_s180.dat};
          \logLogSlopeTriangle{0.90}{0.4}{0.1}{1}{blue};
          \logLogSlopeTriangle{0.90}{0.4}{0.1}{3/2}{red};
          \logLogSlopeTriangle{0.90}{0.4}{0.1}{2}{darkbrown};
        \end{loglogaxis}
      \end{tikzpicture}
      \subcaption*{\quad\quad $d=2$,\ \ $r = \frac{3}{2}$,\ \ $s=\frac{9}{5}$}
    \end{minipage}
    \hfill
    \begin{minipage}[b]{0.31\columnwidth}
      \begin{tikzpicture}[scale=0.6]
        \begin{loglogaxis}
          \addplot table[x=meshsize,y=errp] {results/trigo2D/distTriMeshes/k1_del1_a150_r150_s200.dat};
          \addplot table[x=meshsize,y=errp] {results/trigo2D/distTriMeshes/k2_del1_a150_r150_s200.dat};
          \addplot table[x=meshsize,y=errp] {results/trigo2D/distTriMeshes/k3_del1_a150_r150_s200.dat};
          \logLogSlopeTriangle{0.90}{0.4}{0.1}{1}{blue};
          \logLogSlopeTriangle{0.90}{0.4}{0.1}{3/2}{red};
          \logLogSlopeTriangle{0.90}{0.4}{0.1}{2}{darkbrown};
        \end{loglogaxis}
      \end{tikzpicture}
      \subcaption*{\quad\quad $d=2$,\ \ $r = \frac{3}{2}$,\ \ $s=2$}
    \end{minipage}
  \end{center}
  \begin{center}
    %VITESSE
    \begin{minipage}[b]{0.04\columnwidth}
      \begin{tikzpicture}
      	\draw node[scale=0.8,rotate=90]{\hspace*{0.9cm} $\| \bu u_h - \bI{h}{k} \b u \|_{\strain,\sob,h}$};
      \end{tikzpicture}
    \end{minipage}
    \hfill
    \begin{minipage}[b]{0.3\columnwidth}
      \begin{tikzpicture}[scale=0.6]
        \begin{loglogaxis}
          \addplot table[x=meshsize,y=erru] {results/trigo3D/cubeMeshes/k1_del1_a180_r180_s150.dat};
          \addplot table[x=meshsize,y=erru] {results/trigo3D/cubeMeshes/k2_del1_a180_r180_s150.dat};
          \addplot table[x=meshsize,y=erru] {results/trigo3D/cubeMeshes/k3_del1_a180_r180_s150.dat};
          \logLogSlopeTriangle{0.90}{0.4}{0.1}{2}{blue};
          \logLogSlopeTriangle{0.90}{0.4}{0.1}{3}{red};
          \logLogSlopeTriangle{0.90}{0.4}{0.1}{4}{darkbrown};
        \end{loglogaxis}
      \end{tikzpicture}
    \end{minipage}
    \hfill
    \begin{minipage}[b]{0.30\columnwidth}
      \begin{tikzpicture}[scale=0.6]
        \begin{loglogaxis}[
            legend style = { 
              at={(0.5,1.25)},
              anchor = north,
              legend columns=-1,
              font={\Large}
            }
          ]
          \addplot table[x=meshsize,y=erru] {results/trigo3D/cubeMeshes/k1_del1_a180_r180_s180.dat};
          \addplot table[x=meshsize,y=erru] {results/trigo3D/cubeMeshes/k2_del1_a180_r180_s180.dat};
          \addplot table[x=meshsize,y=erru] {results/trigo3D/cubeMeshes/k3_del1_a180_r180_s180.dat};
          \logLogSlopeTriangle{0.90}{0.4}{0.1}{2}{blue};
          \logLogSlopeTriangle{0.90}{0.4}{0.1}{3}{red};
          \logLogSlopeTriangle{0.90}{0.4}{0.1}{4}{darkbrown};
        \end{loglogaxis}
      \end{tikzpicture}
    \end{minipage}
    \hfill
    \begin{minipage}[b]{0.31\columnwidth}
      \begin{tikzpicture}[scale=0.6]
        \begin{loglogaxis}
          \addplot table[x=meshsize,y=erru] {results/trigo3D/cubeMeshes/k1_del1_a180_r180_s200.dat};
          \addplot table[x=meshsize,y=erru] {results/trigo3D/cubeMeshes/k2_del1_a180_r180_s200.dat};
          \addplot table[x=meshsize,y=erru] {results/trigo3D/cubeMeshes/k3_del1_a180_r180_s200.dat};
          \logLogSlopeTriangle{0.90}{0.4}{0.1}{2}{blue};
          \logLogSlopeTriangle{0.90}{0.4}{0.1}{3}{red};
          \logLogSlopeTriangle{0.90}{0.4}{0.1}{4}{darkbrown};
        \end{loglogaxis}
      \end{tikzpicture}
    \end{minipage}
  \\
  \medskip
  	%PRESSION
    \begin{minipage}[b]{0.04\columnwidth}
      \begin{tikzpicture}
      	\draw node[scale=0.8,rotate=90]{\hspace*{1.7cm}$\| p_h - \proj{h}{k} p \|_{L^{\sob'}(\Omega)}$};
      \end{tikzpicture}
    \end{minipage}
    \hfill
    \begin{minipage}[b]{0.3\columnwidth}
      \begin{tikzpicture}[scale=0.6]
        \begin{loglogaxis}
          \addplot table[x=meshsize,y=errp] {results/trigo3D/cubeMeshes/k1_del1_a180_r180_s150.dat};
          \addplot table[x=meshsize,y=errp] {results/trigo3D/cubeMeshes/k2_del1_a180_r180_s150.dat};
          \addplot table[x=meshsize,y=errp] {results/trigo3D/cubeMeshes/k3_del1_a180_r180_s150.dat};
          \logLogSlopeTriangle{0.90}{0.4}{0.1}{8/5}{blue};
          \logLogSlopeTriangle{0.90}{0.4}{0.1}{12/5}{red};
          \logLogSlopeTriangle{0.90}{0.4}{0.1}{16/5}{darkbrown};
        \end{loglogaxis}
      \end{tikzpicture}
      \subcaption*{\quad\quad $d=3$,\ \ $r = \frac95$,\ \ $s=\frac32$}
    \end{minipage}
    \hfill
    \begin{minipage}[b]{0.3\columnwidth}
      \begin{tikzpicture}[scale=0.6]
        \begin{loglogaxis}
          \addplot table[x=meshsize,y=errp] {results/trigo3D/cubeMeshes/k1_del1_a180_r180_s180.dat};
          \addplot table[x=meshsize,y=errp] {results/trigo3D/cubeMeshes/k2_del1_a180_r180_s180.dat};
          \addplot table[x=meshsize,y=errp] {results/trigo3D/cubeMeshes/k3_del1_a180_r180_s180.dat};
          \logLogSlopeTriangle{0.90}{0.4}{0.1}{8/5}{blue};
          \logLogSlopeTriangle{0.90}{0.4}{0.1}{12/5}{red};
          \logLogSlopeTriangle{0.90}{0.4}{0.1}{16/5}{darkbrown};
        \end{loglogaxis}
      \end{tikzpicture}
      \subcaption*{\quad\quad $d=3$,\ \ $r = \frac95$,\ \ $s=\frac{9}{5}$}
    \end{minipage}
    \hfill
    \begin{minipage}[b]{0.31\columnwidth}
      \begin{tikzpicture}[scale=0.6]
        \begin{loglogaxis}
          \addplot table[x=meshsize,y=errp] {results/trigo3D/cubeMeshes/k1_del1_a180_r180_s200.dat};
          \addplot table[x=meshsize,y=errp] {results/trigo3D/cubeMeshes/k2_del1_a180_r180_s200.dat};
          \addplot table[x=meshsize,y=errp] {results/trigo3D/cubeMeshes/k3_del1_a180_r180_s200.dat};
          \logLogSlopeTriangle{0.90}{0.4}{0.1}{8/5}{blue};
          \logLogSlopeTriangle{0.90}{0.4}{0.1}{12/5}{red};
          \logLogSlopeTriangle{0.90}{0.4}{0.1}{16/5}{darkbrown};
        \end{loglogaxis}
      \end{tikzpicture}
      \subcaption*{\quad\quad $d=3$,\ \ $r = \frac95$,\ \ $s=2$}
    \end{minipage}
  \end{center}
   \caption{Numerical results for the test cases of Section \ref{sec:num.res.trigo} where $\frac{r^*}{r'} = 2$. The slopes indicate the convergence rates expected from Theorem \ref{thm:error.estimate} when $s=2$, and \cite[Theorem 12]{Botti.Castanon-Quiroz.ea:20} otherwise.
   \label{tab:num.res}}
\end{figure}

%2D
\begin{table}\setlength{\tabcolsep}{5pt}
\caption{Convergence rates of the numerical tests of Section \ref{sec:num.res.trigo} with $d=2$ and $r \in \big\{\frac{9}{5},2,\frac52\big\}$.}
\label{table:2D}
\begin{small}
\begin{center}
\begin{tabular}{cc|cc|cc|cc|cc|cc}
\toprule
\multicolumn{12}{c}{$r = \frac{9}{5}$}\\[2pt]
\hline
\multicolumn{2}{c|}{$s$} & \multicolumn{2}{c|}{$\frac{3}{2}$} & \multicolumn{2}{c|}{$\frac{9}{5}$} & \multicolumn{2}{c|}{2} & \multicolumn{2}{c|}{$\frac52$}  & \multicolumn{2}{c}{3} \\[2pt]
\hline
$k$ & $h$ & $\b u$ & $p$  & $\b u$ & $p$  & $\b u$ & $p$  & $\b u$ & $p$  & $\b u$ & $p$  \\
\hline
\multirow{4}{*}{1} & 4.06e-02 & 2.06 & 1.95 & 2.07 & 1.99 & 2.09 & 2.06 & 2.14 & 2.19 & 2.18 & 2.17   \\
& 2.03e-02 & 2.11 & 2.32 & 2.12 & 2.34 & 2.13 & 2.35 & 2.14 & 2.33 & 2.14 & 2.29 \\
& 1.03e-02 & 2.06 & 1.99 & 2.05 & 1.97 & 2.05 & 1.96 & 2.05 & 1.99 & 2.04 & 2.04  \\
& 5.11e-03 & 2.12 & 2.23 & 2.13 & 2.21 & 2.13 & 2.20 & 2.14 & 2.16 & 2.15 & 2.16 \\
\hline
\multirow{3}{*}{2} & 4.06e-02 & 3.31 &  2.32 & 3.30 & 2.67 & 3.29 & 2.94 & 3.26 & 2.50 & 3.23 &  2.34   \\
& 2.03e-02 & 3.20 &  2.36 & 3.21 & 2.65 & 3.22 & 3.13 & 3.22 & 2.47 & 3.23 &  2.38  \\
& 1.03e-02 & 3.12 &  2.19 & 3.14 &  2.37 & 3.14 & 3.22 & 3.14 &  2.26 & 3.15 &  2.22 \\
\hline
\multirow{3}{*}{3} & 4.06e-02 & 4.50 & 6.13 & 4.77 & 6.42 & 4.74 & 6.42 & 4.64 & 6.29 & 4.49 & 6.07   \\
& 2.03e-02 &  3.31 & 3.36 & 4.20 & 3.71 & 4.25 & 3.72 & 4.29 & 3.76 & 4.32 & 3.83  \\
& 1.03e-02 &  2.83 &  3.09 & 4.05 & 3.87 & 4.21 & 3.93 & 4.21 & 3.99 & 4.22 & 4.07  \\
\toprule
\multicolumn{12}{c}{$r = 2$}\\[2pt]
\hline
\multicolumn{2}{c|}{$s$} & \multicolumn{2}{c|}{$\frac{3}{2}$} & \multicolumn{2}{c|}{$\frac{9}{5}$} & \multicolumn{2}{c|}{2} & \multicolumn{2}{c|}{$\frac52$}  & \multicolumn{2}{c}{3} \\[2pt]
\hline
$k$ & $h$ & $\b u$ & $p$  & $\b u$ & $p$  & $\b u$ & $p$  & $\b u$ & $p$  & $\b u$ & $p$  \\
\hline
\multirow{4}{*}{1} & 4.06e-02 &  1.95 &  1.72 &  1.96 &  1.74 &  1.97 &  1.79 & 2.04 &  1.90 & 2.07 &  1.94   \\
& 2.03e-02 & 2.01 & 2.09 & 2.01 & 2.10 & 2.01 & 2.11 & 2.01 & 2.13 & 2.00 & 2.10 \\
& 1.03e-02 &  1.95 &  1.88 &    1.94 &  1.88 &  1.93 &  1.87 &  1.93 &  1.86 &  1.92 &  1.88 \\
& 5.11e-03 & 2.01 & 2.03 & 2.01 & 2.02 & 2.01 & 2.01 & 2.01 &  1.98 & 2.01 &  1.97 \\
\hline
\multirow{3}{*}{2} & 4.06e-02 & 3.05 &  2.08 & 3.04 &  2.49 & 3.03 &   2.80 & 3.01 &  2.37 &  2.99 &  2.20 \\
& 2.03e-02 &  2.99 &  2.11 & 3.00 &  2.37 & 3.00 &  2.82 & 3.00 &  2.24 & 3.01 &  2.16 \\
& 1.03e-02 &  2.91 &  1.97 &  2.92 &  2.13 &  2.93 &  2.99 &  2.93 &  2.04 &  2.93 &  2.00 \\
\hline
\multirow{3}{*}{3} & 4.06e-02 &  3.83 &  3.55 & 4.07 &  3.80 & 4.08 &  3.87 & 4.09 & 4.00 & 4.09 & 4.06   \\
& 2.03e-02 &  3.11 &  3.15 &  3.86 & 3.90 & 3.89 & 3.90 & 3.90 & 3.80 & 3.91 & 3.73 \\
& 1.03e-02 &  2.66 & 2.63 & 3.80 & 3.71 & 3.92 & 3.92 & 3.90 & 3.96 & 3.88 & 3.98 \\
\toprule
\multicolumn{12}{c}{$r = \frac52$}\\[2pt]
\hline
\multicolumn{2}{c|}{$s$} & \multicolumn{2}{c|}{$\frac{3}{2}$} & \multicolumn{2}{c|}{$\frac{9}{5}$} & \multicolumn{2}{c|}{2} & \multicolumn{2}{c|}{$\frac52$}  & \multicolumn{2}{c}{3} \\[2pt]
\hline
$k$ & $h$ & $\b u$ & $p$  & $\b u$ & $p$  & $\b u$ & $p$  & $\b u$ & $p$  & $\b u$ & $p$  \\
\hline
\multirow{4}{*}{1} & 4.06e-02 & 1.68 & 1.43 & 1.79 & 1.47 & 1.79 & 1.49 & 1.79 & 1.53 & 1.80 & 1.55   \\
& 2.03e-02 & 1.64 & 1.58 & 1.78 & 1.61 & 1.78 & 1.62 & 1.76 & 1.63 & 1.75 & 1.64 \\
& 1.03e-02 & 1.50 & 1.59 & 1.67 & 1.60 & 1.66 & 1.59 & 1.65 & 1.57 & 1.64 & 1.55 \\
& 5.11e-03 & 1.46 & 1.60 & 1.65 & 1.61 & 1.65 & 1.61 & 1.66 & 1.61 & 1.67 & 1.61 \\
\hline
\multirow{3}{*}{2} & 4.06e-02 & 2.68 & 2.22 & 2.68 & 2.60 & 2.67 & 2.68 & 2.65 & 2.47 & 2.64 & 2.26   \\
& 2.03e-02 & 2.64 & 1.83 & 2.64 & 2.28 & 2.64 & 2.54 & 2.64 & 2.07 & 2.64 & 1.89 \\
& 1.03e-02 & 2.53 & 1.61 & 2.55 & 1.86 & 2.55 & 2.39 & 2.55 & 1.71 & 2.55 & 1.63 \\
\hline
\multirow{3}{*}{3} & 4.06e-02 & 3.62 & 3.59 & 3.65 & 3.62 & 3.65 & 3.62 & 3.65 & 3.61 & 3.64 & 3.59   \\
& 2.03e-02 & 2.97 & 2.78 & 3.14 & 2.92 & 3.15 & 2.92 & 3.15 & 2.93 & 3.15 & 2.94 \\
& 1.03e-02 & 2.62 & 2.48 & 3.09 & 2.90 & 3.10 & 2.91 & 3.10 & 2.92 & 3.11 & 2.93 \\
\bottomrule
\end{tabular}
\end{center}
\end{small}
\end{table}

%3D
\begin{table}\setlength{\tabcolsep}{5pt}
\caption{Convergence rates of the numerical tests of Section \ref{sec:num.res.trigo} with $d=3$ and $r \in \big\{2,\frac52,3\big\}$.}
\label{table:3D}
\begin{small}
\begin{center}
\begin{tabular}{cc|cc|cc|cc|cc|cc}
\toprule
\multicolumn{12}{c}{$r = 2$}\\[2pt]
\hline
\multicolumn{2}{c|}{$s$} & \multicolumn{2}{c|}{$\frac{3}{2}$} & \multicolumn{2}{c|}{$\frac{9}{5}$} & \multicolumn{2}{c|}{2} & \multicolumn{2}{c|}{$\frac52$}  & \multicolumn{2}{c}{3} \\[2pt]
\hline
$k$ & $h$ & $\b u$ & $p$  & $\b u$ & $p$  & $\b u$ & $p$  & $\b u$ & $p$  & $\b u$ & $p$  \\
\hline
\multirow{3}{*}{1} & 1.57e-01 &  1.87 &  1.52 &  1.85 &  1.56 &  1.84 &  1.58 &  1.81 &  1.64 &  1.76 &  1.71   \\
& 7.87e-02 &  1.92 & 2.08 &  1.91 & 2.15 &  1.90 & 2.18 &  1.88 & 2.22 &  1.86 & 2.26 \\
& 3.94e-02 &  1.95 & 2.33 &  1.95 & 2.46 &  1.95 & 2.49 &  1.94 & 2.48 &  1.93 & 2.45 \\
\hline
\multirow{2}{*}{2} & 1.57e-01 &  2.88 &  2.86 &  2.87 & 3.01 &  2.87 & 3.06 &  2.86 & 3.09 &  2.87 & 3.10   \\
& 7.87e-02 &  2.91 &  2.77 &  2.92 & 3.23 &  2.92 & 3.34 &  2.92 & 3.04 &  2.92 &  2.74 \\
\hline
\multirow{2}{*}{3} & 1.57e-01 &  3.88 &  3.88 &  3.87 &  3.90 &  3.86 &  3.91 &  3.85 &  3.94 &  3.83 & 4.00   \\
& 7.87e-02 &  3.92 & 4.24 &  3.92 & 4.27 &  3.92 & 4.28 &  3.91 & 4.29 &  3.91 & 4.29 \\
\toprule
\multicolumn{12}{c}{$r = \frac52$}\\[2pt]
\hline
\multicolumn{2}{c|}{$s$} & \multicolumn{2}{c|}{$\frac{3}{2}$} & \multicolumn{2}{c|}{$\frac{9}{5}$} & \multicolumn{2}{c|}{2} & \multicolumn{2}{c|}{$\frac52$}  & \multicolumn{2}{c}{3} \\[2pt]
\hline
$k$ & $h$ & $\b u$ & $p$  & $\b u$ & $p$  & $\b u$ & $p$  & $\b u$ & $p$  & $\b u$ & $p$  \\
\hline
\multirow{3}{*}{1} & 1.57e-01 & 1.77 & 1.38 & 1.75 & 1.42 & 1.74 & 1.43 & 1.72 & 1.47 & 1.68 & 1.53   \\
& 7.87e-02 & 1.88 & 1.78 & 1.87 & 1.95 & 1.86 & 1.99 & 1.83 & 2.03 & 1.80 & 2.06 \\
& 3.94e-02 & 1.92 & 1.90 & 1.91 & 2.21 & 1.90 & 2.29 & 1.88 & 2.17 & 1.85 & 2.04 \\
\hline
\multirow{2}{*}{2} & 1.57e-01 & 2.46 & 2.61 & 2.46 & 2.71 & 2.46 & 2.73 & 2.46 & 2.70 & 2.46 & 2.65   \\
& 7.87e-02 & 2.57 & 2.31 & 2.57 & 2.89 & 2.57 & 3.05 & 2.57 & 2.64 & 2.57 & 2.32 \\
\hline
\multirow{2}{*}{3} & 1.57e-01 & 3.08 & 3.21 & 3.08 & 3.22 & 3.07 & 3.22 & 3.06 & 3.22 & 3.05 & 3.22   \\
& 7.87e-02 & 3.32 & 3.51 & 3.32 & 3.52 & 3.32 & 3.52 & 3.31 & 3.52 & 3.31 & 3.52 \\
\toprule
\multicolumn{12}{c}{$r = 3$}\\[2pt]
\hline
\multicolumn{2}{c|}{$s$} & \multicolumn{2}{c|}{$\frac{3}{2}$} & \multicolumn{2}{c|}{$\frac{9}{5}$} & \multicolumn{2}{c|}{2} & \multicolumn{2}{c|}{$\frac52$}  & \multicolumn{2}{c}{3} \\[2pt]
\hline
$k$ & $h$ & $\b u$ & $p$  & $\b u$ & $p$  & $\b u$ & $p$  & $\b u$ & $p$  & $\b u$ & $p$  \\
\hline
\multirow{3}{*}{1} & 1.57e-01 & 1.63 & 1.29 & 1.62 & 1.31 & 1.62 & 1.31 & 1.60 & 1.31 & 1.59 & 1.35   \\
& 7.87e-02 & 1.73 & 1.61 & 1.72 & 1.82 & 1.72 & 1.88 & 1.71 & 1.90 & 1.69 & 1.90 \\
& 3.94e-02 & 1.17 & 1.56 & 1.18 & 1.89 & 1.19 & 2.03 & 1.21 & 1.80 & 1.25 & 1.61 \\
\hline
\multirow{2}{*}{3} & 1.57e-01 & 2.16 & 2.42 & 2.16 & 2.45 & 2.16 & 2.45 & 2.16 & 2.43 & 2.16 & 2.39   \\
& 7.87e-02 & 2.30 & 2.21 & 2.30 & 2.57 & 2.30 & 2.64 & 2.30 & 2.43 & 2.30 & 2.20 \\
\hline
\multirow{2}{*}{2} & 1.57e-01 & 2.76 & 2.84 & 2.76 & 2.84 & 2.75 & 2.83 & 2.75 & 2.83 & 2.75 & 2.83   \\
& 7.87e-02 & 2.93 & 2.89 & 2.93 & 2.89 & 2.93 & 2.89 & 2.93 & 2.89 & 2.92 & 2.89 \\
\bottomrule
\end{tabular}
\end{center}
\end{small}
\end{table}

\subsection{Lid-driven cavity flow}\label{sec:cavity}

We next consider the lid-driven cavity flow, a well-known problem in fluid mechanics.
While this problem has been solved with a large variety of numerical methods for Newtonian fluids with standard convection law, some of the combinations of general viscosity and convection laws considered here appear to be entirely new.
The domain is the unit square $\Omega=(0,1)^2$, and we enforce a unit tangential velocity $\boldsymbol{u}=(1,0)$ on the top edge (of equation $x_2=1$) and wall boundary conditions on the other edges.
This boundary condition is incompatible with the formulation \eqref{eq:ns.weak}, even generalized to non-homogeneous boundary conditions, since $\b u \notin\b U$.
However, this is a very classical test that demonstrates the performance of the method in situations closer to real-life problems.
We consider the diffusion law corresponding to the $(\mu,1,\sob,\sob)$-Carreau--Yasuda model \eqref{eq:Carreau--Yasuda} with a moderate Reynolds number $\cst{Re} \coloneq \frac{2}{\mu} = 1000$, and the convection law given by the $(1,s)$-Laplace formula \eqref{eq:Laplace}.
In order to compare the flow behavior with respect to both $r$ and $s$, with $(r,s)$ in $\big\{\frac{3}{2},2,3\big\}\times \{2\}$ and $\big\{\frac{5}{2}\big\}\times\big\{\frac{3}{2},2,\frac{9}{2}\big\}$, we solve the discrete problem on a Cartesian mesh of size $32\times 32$ for $k=3$, corresponding to $15872$ degrees of freedom.
In Figure \ref{fig:cavity.graph} we display the velocity magnitude, while in Figure \ref{fig:cavity} we plot the horizontal component $u_1$ of the velocity along the vertical centreline $x_1=\frac12$ (resp., vertical component $u_2$ along the horizontal centreline $x_2=\frac12$).
When $r=s=2$, reference solutions from the literature \cite{Ghia.Ghia.ea:82,Erturk.Corke.ea:05} are also plotted for the sake of comparison.
We observe significant differences in the behavior of the flow according to the viscous exponent $r$ and the convective exponent $s$, coherent with the expected physical behavior.
In particular, the viscous effects increase with $r$, as reflected by the size of the central vortex and the inclination of the centrelines. We observed the same phenomenon on the Stokes problem, cf. \cite[Sec. 5.2]{Botti.Castanon-Quiroz.ea:20}.
Moreover, the turbulent effects increase with $s$ as shown by the circular nature of the central vortex and the sharpness of the centrelines.

\begin{figure}
\begin{center}
 	  \begin{minipage}[b]{0.4\columnwidth}
   		\includegraphics[scale=0.44]{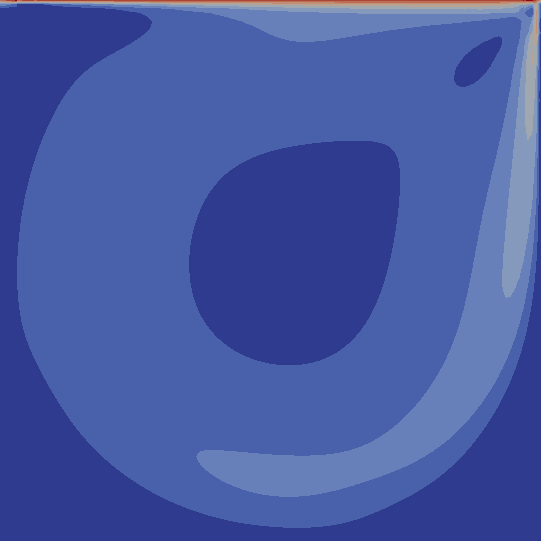}
   		\begin{center}
   		\vspace{-3mm}$r=\frac{3}{2}$,\ \ $s=2$
   		\end{center}
   		\includegraphics[scale=0.44]{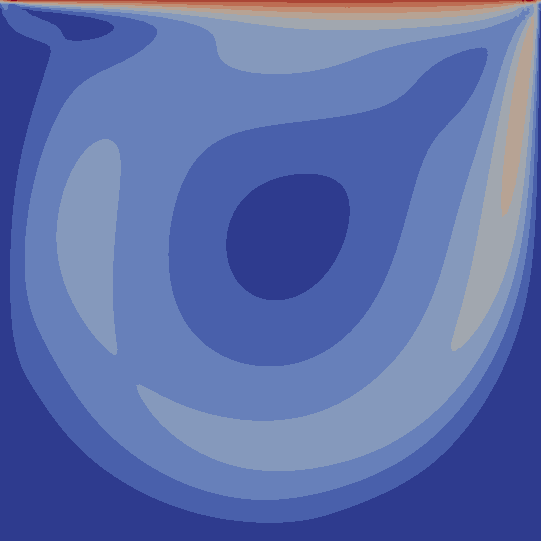}
   		\begin{center}
   		\vspace{-3mm}$r=2$,\ \ $s=2$ \phantom{$\frac{3}{2}$}
   		\end{center}
   		 \includegraphics[scale=0.44]{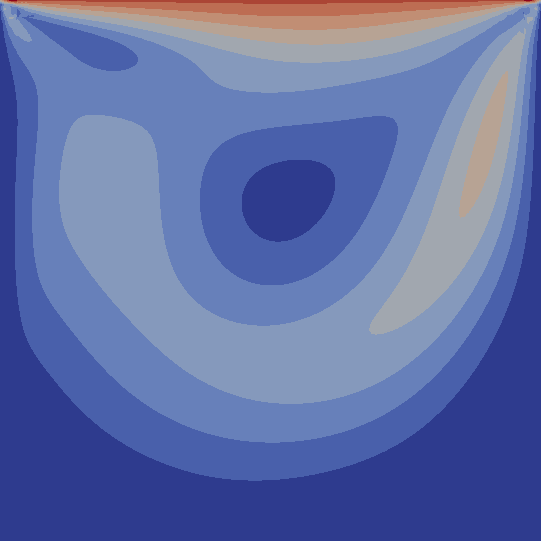}
   		\begin{center}
   		\vspace{-3mm}$r=3$,\ \ $s=2$ \phantom{$\frac{3}{2}$}
   		\end{center}
   	\end{minipage} 
   	\hspace{5mm}
   	 	  \begin{minipage}[b]{0.4\columnwidth}
   		\includegraphics[scale=0.44]{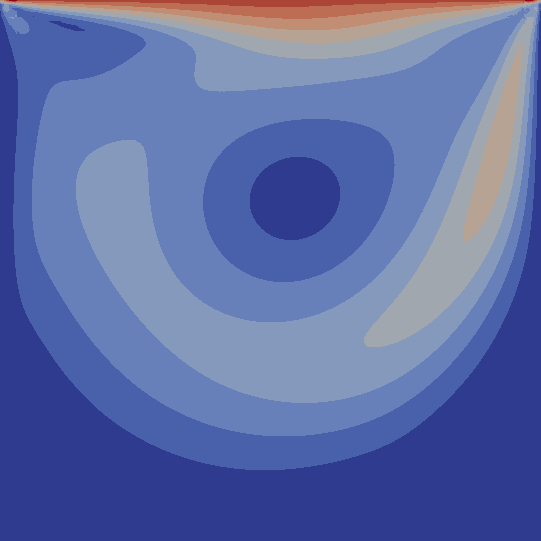}
   		\begin{center}
   		\vspace{-3mm}$r=\frac{5}{2}$,\ \ $s=\frac{3}{2}$
   		\end{center}
   		\includegraphics[scale=0.44]{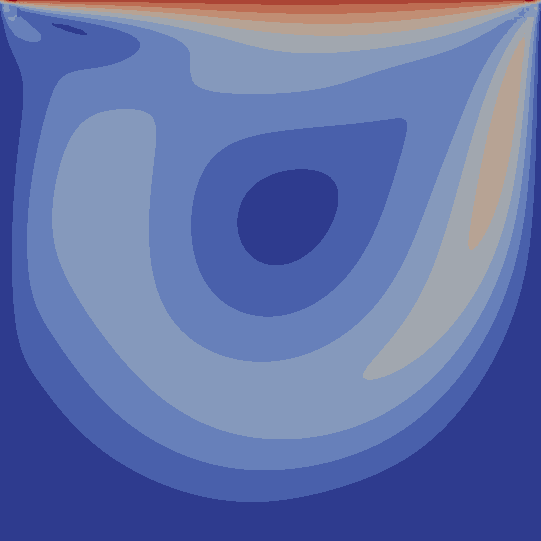}
   		\begin{center}
   		\vspace{-3mm}$r=\frac{5}{2}$,\ \ $s=2$
   		\end{center}
   		 \includegraphics[scale=0.44]{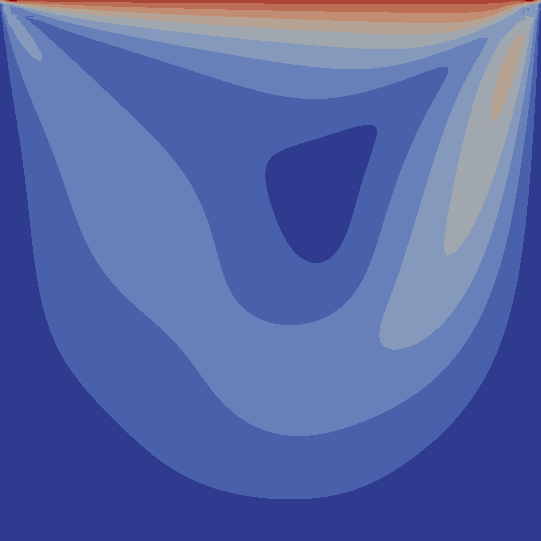}
   		\begin{center}
   		\vspace{-3mm}$r=\frac{5}{2}$,\ \ $s=\frac{9}{2}$
   		\end{center}
   	\end{minipage} 
	\end{center}
  \caption{Numerical results for the test case of Section \ref{sec:cavity}. Velocity magnitude contours ($10$ equispaced values in the range $[0,1]$).  
  \label{fig:cavity.graph}
  }
	\end{figure}

	\begin{figure}
\begin{center}
   		\begin{tikzpicture}[font=\footnotesize, 
          spy using outlines={magnification=4, size=3cm, connect spies, fill=none} 
        ]
        \begin{axis}[height=10cm, width=10cm,
            xmin=-1, xmax=1, ymin=0, ymax=1,
            xlabel={$u_1$}, ylabel={$x_2$},
            legend style = { at={(0,1)}, anchor=north west, draw=none, fill=none }, 
            axis x line=top, ytick pos=bottom,
            legend cell align={left}]
          
          \addplot +[mark=none, thick, blue] table [x=a, y=Points:1] {results/lid_driven_cavity/squaremesh/re1000_r150_s200_k3_32_u1.txt};
          \addplot +[mark=none, thick, black] table [x=a, y=Points:1] {results/lid_driven_cavity/squaremesh/re1000_r200_s200_k3_32_u1.txt};
          \addplot +[mark=none, thick, red] table [x=a, y=Points:1] {results/lid_driven_cavity/squaremesh/re1000_r300_s200_k3_32_u1.txt};
          \addplot +[mark=+, only marks, mark size=1.4, black] table [x=Re1000, y=y] {results/lid_driven_cavity/squaremesh/ghia-ghia-shin-u1.txt};  
          \addplot +[mark=o, only marks, mark size=1.4, black] table [x=Re1000, y=y] {results/lid_driven_cavity/squaremesh/erturk-corke-gokcol-u1.txt}; 
          \legend{
            {$r=\frac{3}{2}$\hspace{2.35cm}$s=2$},
            {$r=2$},
            {$r=3$},
            {Ghia et al.},
            {Erturk et al.}
          }
        \end{axis}
        \begin{axis}[height=10cm, width=10cm,
            xmin=0, xmax=1, ymin=-1, ymax=1,
            xlabel={$x_1$}, ylabel={$u_2$}, 
            axis y line=right, xtick pos=left]
          \addplot +[mark=none, thick, blue] table [x=Points:0, y=b] {results/lid_driven_cavity/squaremesh/re1000_r150_s200_k3_32_u2.txt};  
          \addplot +[mark=none, thick, black] table [x=Points:0, y=b] {results/lid_driven_cavity/squaremesh/re1000_r200_s200_k3_32_u2.txt};   
          \addplot +[mark=none, thick, red] table [x=Points:0, y=b] {results/lid_driven_cavity/squaremesh/re1000_r300_s200_k3_32_u2.txt};  
          \addplot +[mark=+, only marks, mark size=1.4, black] table [x=x, y=Re1000] {results/lid_driven_cavity/squaremesh/ghia-ghia-shin-u2.txt};  
          \addplot +[mark=o, only marks, mark size=1.4, black] table [x=x, y=Re1000] {results/lid_driven_cavity/squaremesh/erturk-corke-gokcol-u2.txt}; 
        \end{axis}
      \end{tikzpicture}
	  ~\medbreak
   		\begin{tikzpicture}[font=\footnotesize, 
          spy using outlines={magnification=4, size=3cm, connect spies, fill=none} 
        ]
        \begin{axis}[height=10cm, width=10cm,
            xmin=-1, xmax=1, ymin=0, ymax=1,
            xlabel={$u_1$}, ylabel={$x_2$},
            legend style = { at={(0,1)}, anchor=north west, draw=none, fill=none }, 
            axis x line=top, ytick pos=bottom,
            legend cell align={left}]
          
          \addplot +[mark=none, thick, blue] table [x=a, y=Points:1] {results/lid_driven_cavity/squaremesh/re1000_r250_s150_k3_32_u1.txt};
          \addplot +[mark=none, thick, black] table [x=a, y=Points:1] {results/lid_driven_cavity/squaremesh/re1000_r250_s200_k3_32_u1.txt};
          \addplot +[mark=none, thick, red] table [x=a, y=Points:1] {results/lid_driven_cavity/squaremesh/re1000_r250_s450_k3_32_u1.txt};

          \legend{%
            {$s=\frac{3}{2}$\hspace{2.35cm}$r=\frac{5}{2}$},
            {$s=2$},
            {$s=\frac{9}{2}$}
          }
        \end{axis}
        \begin{axis}[height=10cm, width=10cm,
            xmin=0, xmax=1, ymin=-1, ymax=1,
            xlabel={$x_1$}, ylabel={$u_2$}, 
            axis y line=right, xtick pos=left]
          \addplot +[mark=none, thick, blue] table [x=Points:0, y=b] {results/lid_driven_cavity/squaremesh/re1000_r250_s150_k3_32_u2.txt};  
          \addplot +[mark=none, thick, black] table [x=Points:0, y=b] {results/lid_driven_cavity/squaremesh/re1000_r250_s200_k3_32_u2.txt};  
          \addplot +[mark=none, thick, red] table [x=Points:0, y=b] {results/lid_driven_cavity/squaremesh/re1000_r250_s450_k3_32_u2.txt};  
        \end{axis}
      \end{tikzpicture}
	\end{center}
	\vspace{-1mm}
  \caption{Numerical results for the test case of Section \ref{sec:cavity}. 
    Horizontal component $u_1$ of the velocity along the vertical centreline $x_1=\frac12$ and vertical component $u_2$ of the velocity along the horizontal centreline $x_2=\frac12$. \emph{Top:} Results with $(r,s) \in \big\{\frac{3}{2},2,3\big\}\times \{2\}$. \emph{Bottom:} Results with $(r,s)\in \big\{\frac{5}{2}\big\}\times\big\{\frac{3}{2},2,\frac{9}{2}\big\}$.\label{fig:cavity}
  }
	\end{figure}
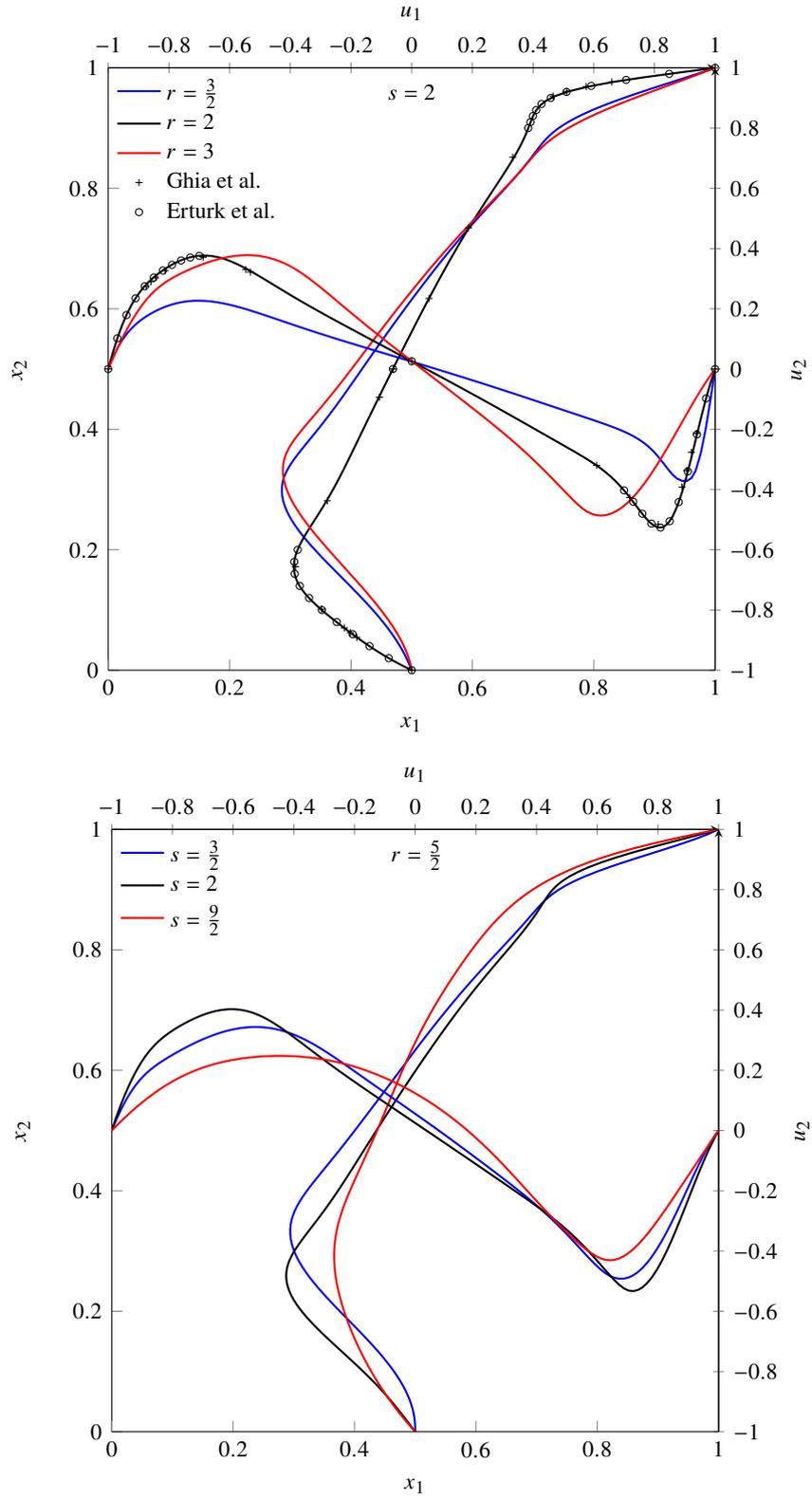

%------------------------------------------------------------------------------%
%------------------------------------------------------------------------------%

\section{Proofs of the main results}\label{sec:main.results.properties}

In this section we first give the proof of the properties \eqref{eq:ch:consistency} and \eqref{eq:ch:sequential.consistency} of the discrete convective function $\cst{c}_h$, then prove, in this order, Theorems \ref{thm:convergence} and \ref{thm:error.estimate}.

\subsection{Consistency of $\cst{c}_h$}\label{sec:ch:consistency}

\begin{proof}[Proof of \eqref{eq:ch:consistency} (Consistency).]
  Let, for the sake of conciseness, $\bu{\hat w}_h \coloneq \bI{h}{k} \b w$. Using the single-valuedness of $(\b w \cdot \b n_{TF})\convec(\cdot,\b w)$ across any interface $F\in\Fi$ together with the fact that $\b v_F = \b 0$ on any boundary face $F\in\Fb$, we get
  \[
  \sum_{T \in \mathcal T_h} \sum_{F \in \mathcal F_T} \int_F (\b w \cdot\b n_{TF})\convec(\cdot,\b w)\cdot \b v_F = 0.
  \]
    Proceeding as in \eqref{eq:c:weak.formulation} but with an element-by-element integration by parts, and using the previous relation to insert $\b v_F$ into the boundary term, we infer
  \begin{multline}\label{eq:ch:cons:diff:left}
      \int_\Omega \tri{\b w}{\GRAD}{\convec(\cdot,\b w)}{\b v_h} = \frac{1}{\conv}\int_\Omega \tri{\convec(\cdot,\b w)}{\GRAD}{\b w}{\b v_h}+\frac{\conv-2}{\conv}\int_\Omega \frac{\b v_h \cdot \b w}{|\b w|^2}\tri{\convec(\cdot,\b w)}{\GRAD}{\b w}{\b w} \\
      \qquad -\frac{1}{\conv'}\left(\int_\Omega \tri{\convec(\cdot,\b w)}{\brkGRAD}{\b v_h}{\b w} +\sum_{T \in \T_h} \sum_{F \in \F_T} \int_F (\b w \cdot \b n_{TF})\convec(\cdot,\b w) \cdot (\b v_F -\b v_T)\right).
  \end{multline}
  Using the definitions \eqref{eq:proj} of $\PROJ{T}{k}$, \eqref{eq:Gh} of $\dgrad{k}{h}$, and \eqref{eq:G.T} of $\dgrad{k}{T}$ (the latter with $\b\tau = \PROJ{T}{k}\left(\convec(\cdot,\b{\hat w}_T) \otimes \b{\hat w}_{T}\right)$), we get
  \begin{multline}\label{eq:ch:cons:diff:etap}
      \int_\Omega \tri{\convec(\cdot,\b{\hat w}_h)}{\dgrad{k}{h}}{\bu v_h}{\b{\hat w}_h}
      = \sum_{T\in\T_h}\int_{T} \dgrad{k}{T}\bu v_{T} : \PROJ{T}{k}\left(\convec(\cdot,\b{\hat w}_T) \otimes \b{\hat w}_{T}\right) \\
      = \sum_{T\in\T_h}\int_{T} \GRAD\b v_T : \cancel{\PROJ{T}{k}}\left(\convec(\cdot,\b{\hat w}_T) \otimes \b{\hat w}_{T}\right)
      +\sum_{T \in \T_h} \sum_{F \in \F_T} \int_F  \left(\PROJ{T}{k}\left(\convec(\cdot,\b{\hat w}_T) \otimes \b{\hat w}_{T}\right)\b n_{TF}\right)\cdot (\b v_F -\b v_T),
  \end{multline}
  where the removal of $\PROJ{T}{k}$ is justified by its definition after observing that $\GRAD\b v_T\in\Poly^{k-1}(T)^{d\times d}\subset\Poly^k(T)^{d\times d}$.
  Plugging \eqref{eq:ch:cons:diff:etap} into the definition \eqref{eq:ch} of $\cst{c}_h$, we obtain
  \begin{multline}\label{eq:ch:cons:diff:right}
    \cst{c}_h(\bu{\hat w}_h,\bu v_h)
    = \frac{1}{\conv}\int_\Omega \tri{\convec(\cdot,\b{\hat w}_h)}{\dgrad{k}{h}}{\bu{\hat w}_{h}}{\b v_h}
    +\frac{\conv-2}{\conv}\int_\Omega \frac{\b v_h \cdot \b{\hat w}_h}{|\b w_h|^{2}}\tri{\convec(\cdot,\b{\hat w}_h)}{\dgrad{k}{h}}{\bu{\hat w}_{h}}{\b{\hat w}_h} \\
    - \frac{1}{\conv'}\left(\int_\Omega \tri{\convec(\cdot,\b{\hat w}_h)}{\brkGRAD}{\b v_h}{\b{\hat w}_h} +\sum_{T \in \T_h} \sum_{F \in \F_T} \int_F \left(\PROJ{T}{k}\left(\convec(\cdot,\b{\hat w}_T) \otimes \b{\hat w}_{T}\right)\b n_{TF}\right)\cdot (\b v_F -\b v_T)\right).
  \end{multline}
  Subtracting \eqref{eq:ch:cons:diff:right} from \eqref{eq:ch:cons:diff:left}, then adding and subtracting to the right-hand side of the resulting expression the quantity
  \begin{multline*}
    \frac{1}{\conv}\int_\Omega \tri{\convec(\cdot,\b w)}{\dgrad{k}{h}}{\bu{\hat w}_h}{\b v_h}
    +\frac{\conv-2}{\conv}\int_\Omega \frac{\b v_h \cdot \b w}{|\b w|^{2}}\tri{\convec(\cdot,\b w)}{\dgrad{k}{h}}{\bu{\hat w}_{h}}{\b w}
    +\frac{1}{\conv'}\bigg[\int_\Omega \tri{\convec(\cdot,\b w)}{\brkGRAD}{\b v_h}{\b{\hat w}_h} \\
      +\sum_{T \in \T_h} \sum_{F \in \F_T}\left(\int_F  \left(\left(\convec(\cdot,\b{\hat w}_T) \otimes \b{\hat w}_{T}\right)\b n_{TF}\right)\cdot (\b v_F -\b v_T)+\int_F (\b{\hat w}_T \cdot \b n_{TF})\convec(\cdot,\b w) \cdot (\b v_F -\b v_T)\right)
      \bigg],
  \end{multline*}
  we obtain
  \begin{equation}\label{eq:ch:cons:diff}
    \int_\Omega \tri{\b w}{\GRAD}{\convec(\cdot,\b w)}{\b v_h} - \cst{c}_h(\bI{h}{k} \b w,\bu v_h)  \\
    = \frac{1}{\conv}\left(\mathcal T_1 + \mathcal T_2\right)
    +\frac{\conv-2}{\conv}\left(\mathcal T_3 + \mathcal T_4\right)
    -\frac{1}{\conv'}\left(\mathcal T_5 + \mathcal T_6 + \mathcal T_7 + \mathcal T_8 + \mathcal T_9\right),
  \end{equation}
  with
  \begin{align*}
    \mathcal T_1 &\coloneq
    \int_\Omega \tri{\convec(\cdot,\b w)}{(\GRAD-\dgrad{k}{h}\bI{h}{k})}{\b w}{\b v_h},
    \\
    \mathcal T_2 &\coloneq
    \int_\Omega \btri{(\convec(\cdot,\b w)-\convec(\cdot,\b{\hat w}_h))}{\dgrad{k}{h}}{\bu{\hat w}_{h}}{\b v_h},
    \\
    \mathcal T_3 &\coloneq
    \int_\Omega \frac{\b v_h \cdot \b w}{|\b w|^{2}}\btri{\convec(\cdot,\b w)}{(\GRAD-\dgrad{k}{h}\bI{h}{k})}{\b w}{\b w},
    \\
    \mathcal T_4 &\coloneq
    \int_\Omega\left( \frac{\b v_h \cdot \b w}{|\b w|^{2}}\tri{\convec(\cdot,\b w)}{\dgrad{k}{h}}{\bu{\hat w}_{h}}{\b w}-\frac{\b v_h \cdot \b{\hat w}_h}{|\b{\hat w}_h|^{2}}\tri{\convec(\cdot,\b{\hat w}_h)}{\dgrad{k}{h}}{\bu{\hat w}_{h}}{\b{\hat w}_h}\right),
    \\
    \mathcal T_5 &\coloneq
    \int_\Omega \tri{\convec(\cdot,\b w)}{\brkGRAD}{\b v_h}{(\b w-\b{\hat w}_h)},
    \\
    \mathcal T_6 &\coloneq
    \int_\Omega \btri{(\convec(\cdot,\b w)-\convec(\cdot,\b{\hat w}_h))}{\brkGRAD}{\b v_h}{\b{\hat w}_h},
    \\
    \mathcal T_7 &\coloneq
    \sum_{T \in \T_h} \sum_{F \in \F_T} \int_F ((\b w-\b{\hat w}_T) \cdot \b n_{TF})\convec(\cdot,\b w) \cdot (\b v_F -\b v_T),
    \\
    \mathcal T_8 &\coloneq
    \sum_{T \in \T_h} \sum_{F \in \F_T} \int_F (\b{\hat w}_T \cdot \b n_{TF})~(\convec(\cdot,\b w)-\convec(\cdot,\b{\hat w}_T)) \cdot (\b v_F -\b v_T),
    \\
    \mathcal T_9 &\coloneq
    \sum_{T \in \T_h} \sum_{F \in \F_T} \int_F  \left(\left(\convec(\cdot,\b{\hat w}_T) \otimes \b{\hat w}_{T}-\PROJ{T}{k}(\convec(\cdot,\b{\hat w}_T) \otimes \b{\hat w}_{T})\right)\b n_{TF}\right)\cdot (\b v_F -\b v_T).
  \end{align*}
  We proceed to estimate these terms.
  For $\mathcal T_1$ and $\mathcal T_3$, using the $(1;\conv'\sob',\sob,\conv r')$-H\"{o}lder inequality \eqref{eq:holder} together with the H\"older continuity \eqref{eq:ass:conv:holder.cont} of $\convec$, we get
  \[
  \begin{aligned}
    |\mathcal T_1|+|\mathcal T_3| &\lesssim \chi_\cst{hc}\|\b w\|_{L^{\conv r'}(\Omega)^d}^{s-1}\|\GRAD\b w  - \dgrad{k}{h}(\bI{h}{k}\b w)\|_{L^\sob(\Omega)^{d \times d}} \| \b v_h \|_{L^{\conv r'}(\Omega)^d}\\
    &\lesssim h^{k+1}\chi_\cst{hc}|\b w|_{W^{1,\sob}(\Omega)^d}^{s-1}|\b w|_{W^{k+2,\sob}(\T_h)^d}\| \bu v_h \|_{\strain,\sob,h},
  \end{aligned}
  \]
  where we concluded with the consistency \eqref{eq:Gh:consistency} of the gradient reconstruction together with the continuous \eqref{eq:sob:emb} and discrete \eqref{eq:discrete.sob.emb} Sobolev embeddings (valid since $sr' \le r^*$ by \eqref{eq:cons.ch:s}) and the norm equivalence \eqref{eq:norm.equiv}.

  Moving to $\mathcal T_2$ and $\mathcal T_4$, using a Cauchy--Schwarz inequality together with the H\"older continuity \eqref{eq:ass:conv:holder.cont} of $\convec$ on $\mathcal T_2$ and the same reasoning as in \eqref{eq:c:holder.cont:T4} on $\mathcal T_4$ yields
  \[
  \begin{aligned}
    |\mathcal T_2|+|\mathcal T_4| &\le \chi_\cst{hc}\int_\Omega \left(|\b w|^\conv+|\b{\hat w}_h|^\conv\right)^\frac{\conv-\convs}{\conv}|\b w - \b{\hat w}_h|^{\convs-1}|\dgrad{k}{h}\bu{\hat w}_h|_{d\times d}|\b v_h|\\
    &\le \chi_\cst{hc}\left(\|\b w\|_{L^{\conv r'}(\Omega)^d}^\conv+\|\b{\hat w}_h\|_{L^{\conv r'}(\Omega)^d}^\conv\right)^\frac{\conv-\convs}{\conv}\|\b w-\b{\hat w}_h\|_{L^{\conv r'}(\Omega)^d}^{\convs-1}\|\dgrad{k}{h}\bu{\hat w}_h\|_{L^\sob(\Omega)^{d\times d}}\|\b v_h \|_{L^{\conv r'}(\Omega)^d}\\
    &\lesssim  h^{(k+1)(\convs-1)}\chi_\cst{hc}|\b w|_{W^{1,\sob}(\Omega)^d}^{\conv+1-\convs}|\b w|_{W^{k+1,s\sob'}(\T_h)^d}^{\convs-1}\|\bu v_h \|_{\strain,\sob,h},
  \end{aligned}
  \]
  where we have used the $(1;\frac{\conv r'}{\conv-\convs},\frac{\conv r'}{\convs-1},\sob,\conv r')$-H\"{o}lder inequality \eqref{eq:holder} in the second line, while the conclusion follows from the continuous \eqref{eq:sob:emb} and discrete \eqref{eq:discrete.sob.emb} Sobolev embeddings (again valid since $sr' \le r^*$ by \eqref{eq:cons.ch:s}) along with the boundedness \eqref{eq:Gh:boundedness} of $\dgrad{k}{h}$ and \eqref{eq:I:boundedness} of $\bI{h}{k}$, and the $(k+1,sr',0)$-approximation properties \eqref{eq:proj:app:T} of $\PROJ{h}{k}$.
 
  With a similar reasoning as for $\mathcal T_2$, we get the following bounds for $\mathcal T_5$ and $\mathcal T_6$:
  \begin{align*}
    |\mathcal T_5| &\lesssim h^{k+1}\chi_\cst{hc}|\b w|_{W^{1,\sob}(\Omega)^d}^{\conv-1}|\b w|_{W^{k+1,s\sob'}(\T_h)^d}\|\bu v_h \|_{\strain,\sob,h},
    \\
    |\mathcal T_6| &\lesssim h^{(k+1)(\convs-1)}\chi_\cst{hc}|\b w|_{W^{1,\sob}(\Omega)^d}^{\conv+1-\convs}|\b w|_{W^{k+1,s\sob'}(\T_h)^d}^{\convs-1}\|\bu v_h \|_{\strain,\sob,h}.
  \end{align*}
  Moving to $\mathcal T_8$, the H\"older continuity \eqref{eq:ass:conv:holder.cont} of $\convec$, the $(1;\conv r',\frac{\conv r'}{\conv-\convs},\frac{\conv r'}{\convs-1},\sob)$-H\"{o}lder inequality \eqref{eq:holder}, and the bound $h_F \lesssim h_T$ yield
  \[
  \begin{aligned}
    |\mathcal T_8| &\le  \chi_\cst{hc}\left(\sum_{T \in \mathcal T_h}h_T\| \b{\hat w}_T \|_{L^{\conv r'} (\partial T)^d}^{\conv r'} \right)^\frac{1}{\conv r'}  \left(\sum_{T \in \mathcal T_h} h_T\left(\| \b w \|_{L^{\conv r'}(\partial T)^d}^{\conv r'}+\| \b{\hat w}_T \|_{L^{\conv r'}(\partial T)^d}^{\conv r'}\right) \right)^\frac{\conv-\convs}{\conv r'}\\
    &\qquad \times \left(\sum_{T \in \mathcal T_h}h_T\| \b w-\b{\hat w}_T \|_{L^{\conv r'} (\partial T)^d}^{\conv r'} \right)^\frac{\convs-1}{\conv r'}\left(\sum_{T \in \mathcal T_h}\sum_{F \in \mathcal F_T} h_F^{1-\sob}\| \b v_F-\b v_T \|_{L^{\sob}(F)^d}^\sob \right)^\frac{1}{\sob}  \\
    &\lesssim  h^{(k+1)(\convs-1)}\chi_\cst{hc}\left(|\b w|_{W^{1,s\sob'}(\Omega)^d}^\conv+\|\b w\|_{L^{\conv r'}(\Omega)^d}^\conv+\|\b{\hat w}_h\|_{L^{\conv r'}(\Omega)^d}^\conv\right)^\frac{\conv+1-\convs}{\conv}|\b w|_{W^{k+1,\conv r'}(\T_h)^d}^{\convs-1}\| \bu v_h \|_{\strain,\sob,h} \\
    &\lesssim h^{(k+1)(\convs-1)}\chi_\cst{hc}\left(|\b w|_{W^{1,s\sob'}(\Omega)^d}^s+|\b w|_{W^{1,\sob}(\Omega)^d}^s\right)^\frac{\conv+1-\convs}{s}|\b w|_{W^{k+1,s\sob'}(\T_h)^d}^{\convs-1}\|\bu v_h \|_{\strain,\sob,h},
  \end{aligned}
  \]
  where, to pass to the second line, we have used the discrete trace inequality  \cite[Eq. (1.55)]{Di-Pietro.Droniou:20}, the continuous trace inequality \cite[Eq. (1.51)]{Di-Pietro.Droniou:20} together with the bound $h_T \lesssim 1$, the $(k+1,sr',0)$-trace approximation properties \eqref{eq:proj:app:F} of $\PROJ{T}{k}$, and the definition \eqref{eq:norm.epsilon.sob} of $\|{\cdot}\|_{\strain,\sob,h}$, while the conclusion is obtained using the Sobolev embedding \eqref{eq:sob:emb} (again valid since $sr' \le r^*$ by \eqref{eq:cons.ch:s}).
  
  With a similar reasoning, we get the following bound for $\mathcal T_7$:
  \[
  \begin{aligned}
    |\mathcal T_7| &\lesssim h^{k+1}\chi_\cst{hc}\left(|\b w|_{W^{1,s\sob'}(\Omega)^d}^s+|\b w|_{W^{1,\sob}(\Omega)^d}^s\right)^\frac{\conv-1}{s}|\b w|_{W^{k+1,s\sob'}(\T_h)^d}\|\bu v_h \|_{\strain,\sob,h}.
  \end{aligned}
  \]
  
  Moving to $\mathcal T_9$, the H\"{o}lder inequality together with the bound $h_F \lesssim h_T$ and the definition \eqref{eq:norm.epsilon.sob} of $\|{\cdot}\|_{\strain,\sob,h}$ yield
  \[
  \begin{aligned}
    |\mathcal T_9| 
    &\lesssim  \left(\sum_{T \in \mathcal T_h}h_T\left\| \convec(\cdot,\b{\hat w}_T) \otimes \b{\hat w}_{T}-\PROJ{T}{k}\left(\convec(\cdot,\b{\hat w}_T) \otimes \b{\hat w}_{T}\right) \right\|_{L^{r'} (\partial T)^{d\times d}}^{r'} \right)^\frac{1}{r'}  \| \bu v_h \|_{\strain,\sob,h}  \\
    &\lesssim h^{k+1} \chi_\cst{hc}\left|\b{\hat w}_{T} \right|_{W^{k+1,\conv r'} (\T_h)^d}^s\| \bu v_h \|_{\strain,\sob,h} \lesssim h^{k+1} \chi_\cst{hc}\left|\b w \right|_{W^{k+1,s\sob'}(\T_h)^d}^s\| \bu v_h \|_{\strain,\sob,h},
  \end{aligned}
  \]
  where we passed to the second line using the $(k+1,r',0)$-trace approximation properties \eqref{eq:proj:app:F} of $\PROJ{T}{k}$ together with a triangle inequality and the H\"older continuity \eqref{eq:ass:conv:holder.cont} of $\convec$ (with $(\b v,\b w) = (\b 0,\b{\hat w}_{T})$), while the conclusion follows using a triangle inequality and the $(k+1,sr',k+1)$-approximation properties \eqref{eq:proj:app:T} of $\PROJ{T}{k}$ to write $\left|\b{\hat w}_{T} \right|_{W^{k+1,\conv r'} (\T_h)^d} \le \left|\b{\hat w}_{T} -\b w\right|_{W^{k+1,\conv r'} (\T_h)^d}+\left|\b w \right|_{W^{k+1,s\sob'}(\T_h)^d} \lesssim \left|\b w \right|_{W^{k+1,s\sob'}(\T_h)^d}$.

  Plugging the above bounds for $\mathcal T_1,\ldots,\mathcal T_9$ into \eqref{eq:ch:cons:diff} and passing to the supremum, \eqref{eq:ch:consistency} follows.
\end{proof}

\begin{proof}[Proof of \eqref{eq:ch:sequential.consistency} (Sequential consistency).]
  Let $\b\phi \in C_\cst{c}^\infty(\Omega)^d$ and set $\bu{\hat\phi}_h \coloneq \bI{h}{k} \b\phi$.
  Writing the definition \eqref{eq:ch} of $\cst{c}_h$ with $(\bu w_h,\bu v_h)$ replaced by $(\bu v_h,\bu{\hat\phi}_h)$, we have
  \begin{equation}\label{eq:ch:sequential.consistency:0}
    \begin{aligned}
      \cst{c}_h(\bu v_h,\bu{\hat\phi}_h) &\coloneq \frac{1}{\conv}\underbrace{\int_\Omega \tri{\convec(\cdot,\b v_h)}{\dgrad{k}{h}}{\bu v_h}{\b{\hat\phi}_h}}_{\mathcal T_1}+\frac{\conv-2}{\conv}\underbrace{\int_\Omega \frac{\b{\hat\phi}_h \cdot \b v_h}{|\b v_h|^{2}}\tri{\convec(\cdot,\b v_h)}{\dgrad{k}{h}}{\bu v_h}{\b v_h}}_{\mathcal T_2} \\
      & \quad - \frac{1}{\conv'}\underbrace{\int_\Omega \tri{\convec(\cdot,\b v_h)}{\dgrad{k}{h}}{\bu{\hat\phi}_h}{\b v_h}}_{\mathcal T_3}.
    \end{aligned}
  \end{equation}
Inserting $\pm |\b v_h|^{-2}(\b v_h \otimes \b v_h)\convec(\cdot,\b v)$, then using a triangle inequality, we get
\begin{equation}\label{eq:ch:sequential.consistency:1}
\begin{aligned}
      &\left\||\b v_h|^{-2}(\b v_h \otimes \b v_h)\convec(\cdot,\b v_h)-|\b v|^{-2}(\b v \otimes \b v)\convec(\cdot,\b v)\right\|_{L^{s'r'}(\Omega)^d} \\
     &\qquad \le \|\convec(\cdot,\b v_h)-\convec(\cdot,\b v)\|_{L^{s'r'}(\Omega)^d}+\left\||\convec(\cdot,\b v)|\left(|\b v_h|^{-2}\b v_h \otimes \b v_h-|\b v|^{-2}\b v \otimes \b v\right)\right\|_{L^{s'r'}(\Omega)^{d\times d}} \\
     &\qquad \le \chi_\cst{hc} \left(\|\b v_h\|_{L^{sr'}(\Omega)^d}^\conv+\|\b v\|_{L^{sr'}(\Omega)^d}^\conv\right)^\frac{\conv-\convs}{\conv}\|\b v_h-\b v\|_{L^{sr'}(\Omega)^d}^{\convs-1},
\end{aligned}
\end{equation} 
where we have concluded proceeding as in \eqref{eq:c:holder.cont:1} and using the H\"older continuity \eqref{eq:ass:conv:holder.cont} of $\convec$ and the $(s'r';\frac{sr'}{\conv-\convs},\frac{sr'}{\convs-1})$-H\"older inequality \eqref{eq:holder}.
Thus, recalling that $\b v_h \CV{h}{0} \b v$ strongly in $L^{sr'}(\Omega)^d$ (since $sr' < r^*$ by \eqref{eq:cons.ch:s:strict}), which implies that $\|\b v_h\|_{L^{sr'}(\Omega)^d}^\conv+\|\b v\|_{L^{sr'}(\Omega)^d}^\conv$ is bounded uniformly in $h$, we infer from \eqref{eq:ch:sequential.consistency:1} that,
\begin{align}
  |\b v_h|^{-2}(\b v_h \otimes \b v_h)\convec(\cdot,\b v_h) &\CV{h}{0} |\b v|^{-2}(\b v \otimes \b v)\convec(\cdot,\b v) \text{ strongly in } L^{s'r'}(\Omega)^d,\label{eq:ch:sequential.consistency:chi1}\\
    \convec(\cdot,\b v_h) &\CV{h}{0} \convec(\cdot,\b v) \text{ strongly in } L^{s'r'}(\Omega)^d.\label{eq:ch:sequential.consistency:chi2}
\end{align}
Hence, observing that $\frac{1}{sr'}+\frac{1}{r}+\frac{1}{s'r'} = 1$, the strong convergence \eqref{eq:proj.seq.cons} of $\b{\hat\phi}_h$ in $L^{sr'}(\Omega)^d$ together with the fact that $\dgrad{k}{h}\bu v_h \CV{h}{0} \GRAD\b v$ weakly in $L^{\sob}(\Omega)^{d\times d}$ by assumption, along with \eqref{eq:ch:sequential.consistency:chi1} for $\mathcal T_1$ and \eqref{eq:ch:sequential.consistency:chi2} for $\mathcal T_2$ yield
    \begin{equation}\label{eq:ch:sequential.consistency:T12}
      \mathcal T_2 \CV{h}{0} \int_\Omega \frac{\b\phi \cdot \b v}{|\b v|^2}\tri{\convec(\cdot,\b v)}{\GRAD}{\b v}{\b v}\ \text{ and }\ \mathcal T_1 \CV{h}{0} \int_\Omega \tri{\convec(\cdot,\b v)}{\GRAD}{\b v}{\b\phi}.
    \end{equation}
    Moving to $\mathcal T_3$, the fact that $\b v_h \CV{h}{0} \b v$ strongly in $L^{sr'}(\Omega)^d$ (since $sr' < r^*$ by \eqref{eq:cons.ch:s:strict}) together with the strong convergence \eqref{eq:GT.seq.cons} of $\dgrad{k}{h}\bu{\hat\phi}_h$ in $L^{r}(\Omega)^{d\times d}$ (notice that the weak convergence would suffice to infer the result in \eqref{eq:convergence:vph.ch.c} below) and \eqref{eq:ch:sequential.consistency:chi2} give
    \begin{equation}\label{eq:ch:sequential.consistency:T3}
      \mathcal T_3 \CV{h}{0} \int_\Omega \tri{\convec(\cdot,\b v)}{\GRAD}{\b\phi}{\b v}.
    \end{equation}
    Hence passing to the limit in \eqref{eq:ch:sequential.consistency:0}, using \eqref{eq:ch:sequential.consistency:T12}--\eqref{eq:ch:sequential.consistency:T3}, and recalling the definition \eqref{eq:weak:c} of $c$, we infer \eqref{eq:ch:sequential.consistency}.
\end{proof}

\subsection{Convergence}\label{sec:convergence}

\begin{proof}[Proof of Theorem \ref{thm:convergence}]\label{proof:thm:convergence}
  \textbf{Step 1.} \textit{Existence of a limit}. Since, for all $h \in \mathcal H$, $(\bu u_h,p_h) \in \bdU{h,0}{k} \times \dP{h}{k}$ solves  \eqref{eq:ns.discrete}, the a priori bounds \eqref{thm:discrete.well-posedness:bounds} implies that the sequences $(\|\bu u_h\|_{\strain,\sob,h})_{h \in \mathcal H}$ (hence also $(\|\bu u_h\|_{1,\sob,h})_{h\in\mathcal H}$ by \eqref{eq:norm.equiv}) and $(\|p_h\|_{L^{\sob'}(\Omega)})_{h \in \mathcal H}$ are bounded uniformly in $h$.
  Thus, invoking the discrete compactness result of \cite[Theorem 9.29]{Di-Pietro.Droniou:20}, we infer the existence of $(\b u,p) \in \b U \times P$ such that, up to a subsequence,
  \begin{itemize}[parsep=0pt,noitemsep]
  \item $\b u_{h} \CV{h}{0} \b u$ strongly in $L^{[1,\sob^*)}(\Omega)^d$;
  \item $\dgrad{k}{h} \bu u_{h} \CV{h}{0} \GRAD \b u$ weakly in $L^\sob(\Omega)^{d\times d}$;
  \item $p_{h} \CV{h}{0} p$ weakly in $L^{\sob'}(\Omega)$.
  \end{itemize}
  \textbf{Step 2.} \textit{Identification of the limit.} 
The discrete mass equation \eqref{eq:ns.discrete:mass} along with the sequential consistency \eqref{eq:bh:second.sequential.consistency} of $\cst{b}_h$ yield, for all $\psi\in C^\infty_c(\Omega)$,
  \[
  0 = \cst{b}_{h}(\bu u_{h},\proj{h}{k}\psi) \CV{h}{0} b(\b u,\psi) = 0.
  \]
  By density of $C^\infty_c(\Omega)$ in $L^{r'}(\Omega)$, this shows that $\b u$ satisfies the mass equation \eqref{eq:ns.weak:mass}.   

Let us show now that $(\b u,p)$ satisfy the momentum equation \eqref{eq:ns.weak:momentum}.  
  Since the sequence $(\|\bu u_h\|_{\strain,\sob,h})_{h \in \mathcal H}$ is bounded uniformly in $h$, \eqref{eq:bound.res:stability.boundedness} along with the fact that $\|\dgrads{k}{h}\bu u_h\|_{L^\sob(\Omega)^{d\times d}}\le \|\dgrad{k}{h}\bu u_h\|_{L^\sob(\Omega)^{d\times d}}$ imply that the sequence $(\|\dgrads{k}{h}\bu u_h\|_{L^\sob(\Omega)^{d\times d}})_{h \in \mathcal H}$ is also bounded uniformly in $h$.
  Combined with the H\"older continuity \eqref{eq:ass:stress:holder.cont} of $\stress$, this result implies that the sequence $(\|\stress(\cdot,\dgrads{k}{h}\bu u_h)\|_{L^{\sob'}(\Omega)^{d\times d}})_{h \in \mathcal H}$ is also bounded uniformly in $h$.
  Therefore, $(\stress(\cdot,\dgrads{k}{h}\bu u_h))_{h \in \mathcal H}$ weakly converges to some $\stress_{\b u} \in L^{\sob'}(\Omega,\Ms{d})$ up to a subsequence.
  Furthermore, using the discrete momentum equation \eqref{eq:ns.discrete:momentum} together with the definition \eqref{eq:ah} of $\cst{a}_h$, for all $\b\phi \in C_\cst{c}^\infty(\Omega)^d$, letting $\bu{\hat\phi}_h \coloneqq \bI{h}{k} \b\phi$, we get
  \begin{equation}\label{eq:convergence:G}
    \int_\Omega \stress(\cdot,\dgrads{k}{h}\bu u_h) : \dgrads{k}{h}\bu{\hat\phi}_h = \int_\Omega \b f \cdot \PROJ{h}{k}\b\phi - \cst{s}_h(\bu u_h,\bu{\hat\phi}_h) -\cst{c}_h(\bu u_h,\bu{\hat\phi}_h)-\cst{b}_h(\bu{\hat\phi}_h,p_h).
  \end{equation}
  So, since $\dgrads{k}{h}\bu{\hat\phi}_h \CV{h}{0}{} \GRADs\b\phi$ strongly in $L^\sob(\Omega)^{d\times d}$ thanks to  \eqref{eq:GT.seq.cons} and $\stress(\cdot,\dgrad{k}{h}\bu u_h) \CV{h}{0}{} \stress_{\b u}$ weakly in $L^{\sob'}(\Omega)^{d\times d}$, passing to the sub-limit equality \eqref{eq:convergence:G} yields
  \begin{equation}\label{eq:convergence:sigma.u}
    \int_\Omega \stress_{\b u} : \GRADs \b\phi = \int_\Omega \b f \cdot \b\phi-c(\b u,\b\phi)-b(\b\phi,p),
  \end{equation}
  where we used \eqref{eq:proj.seq.cons} together with the sequential consistencies \eqref{eq:sh:sequential.consistency} of $\cst{s}_h$, \eqref{eq:ch:sequential.consistency} of $\cst{c}_h$, and \eqref{eq:bh:first.sequential.consistency} of $\cst{b}_h$ in the right-hand side.
  By density, \eqref{eq:convergence:sigma.u} is valid for $\b\phi\in \b U$.
  On the other hand, using the definition \eqref{eq:ah} of $\cst{a}_h$, the fact that $\cst{s}_h(\b u_h,\b u_h) \ge 0$, and \eqref{eq:ns.discrete} together with the non-dissipativity \eqref{eq:ch:non-dissip} of $\cst{c}_h$, we infer
  \begin{equation}\label{eq:convergence:bound:1}
    \begin{aligned}
      \int_\Omega \stress(\cdot,\dgrads{k}{h}\bu u_h) : \dgrads{k}{h}\bu u_h = \cst{a}_h(\b u_h,\b u_h) - \cst{s}_h(\b u_h,\b u_h) \leq \cst{a}_h(\b u_h,\b u_h) = \int_\Omega \b f \cdot \b u_h.
    \end{aligned}
  \end{equation}
  Hence, using the H\"older monotonicity \eqref{eq:ass:stress:strong.mono} of $\stress$ and inequality \eqref{eq:convergence:bound:1} we obtain, for all $\boldsymbol{\Lambda} \in L^r(\Omega,\Ms{d})$,
  \begin{equation}\label{eq:convergence:bound:2}
    \begin{aligned}
      0 &\le \int_\Omega \left(\stress(\cdot,\dgrads{k}{h}\bu u_h)-\stress(\cdot,\boldsymbol{\Lambda})\right) : \left(\dgrads{k}{h}\bu u_h-\boldsymbol{\Lambda}\right) \\
      &\le \int_\Omega \b f \cdot \b u_h - \int_\Omega \stress(\cdot,\dgrads{k}{h}\bu u_h) : \boldsymbol{\Lambda} -\int_\Omega \stress(\cdot,\boldsymbol{\Lambda}) : \left(\dgrads{k}{h}\bu u_h-\boldsymbol{\Lambda}\right) \\
      &\CV{h}{0} \int_\Omega \b f \cdot \b u - \int_\Omega \stress_{\b u} : \boldsymbol{\Lambda} -\int_\Omega \stress(\cdot,\boldsymbol{\Lambda}) : \left(\GRADs\b u-\boldsymbol{\Lambda}\right),
    \end{aligned}
  \end{equation}
  where the limit of the third term is justified by the fact that $\stress(\cdot,\boldsymbol{\Lambda}) \in L^{r'}(\Omega,\Ms{d})$ thanks to the H\"older continuity \eqref{eq:ass:stress:holder.cont} of $\stress$ and the weak convergence of $\dgrads{k}{h}\bu u_h\CV{h}{0}\GRADs\b u$ in $L^r(\Omega)^{d\times d}$ (consequence of the analogous property for $\dgrad{k}{h}\bu u_h$).
  It follows from the classical Minty's trick \cite{Leray.Lions:65,Minty:63} (see \cite[Theorem 4.6]{Di-Pietro.Droniou:17} concerning its application to HHO methods) that $(\b u,p)$ satisfies \eqref{eq:ns.weak:momentum}.
   Indeed, taking $\boldsymbol{\Lambda} = \GRADs\b u \pm t\GRADs\b v$ with $t > 0$ and $\b v \in \b U$ into \eqref{eq:convergence:bound:2}, and using \eqref{eq:convergence:bound:1} with $\b\phi = \b u \pm t\b v$, we get
   \begin{equation}\label{eq:Minty.trick}
   t\int_\Omega \stress(\GRADs\b u \pm t\GRADs\b v,\boldsymbol{\Lambda}) : \GRADs\b v  = t\int_\Omega \b f \cdot \b v +\cancel{c(\b u,\b u)}-tc(\b u,\b v)+\cancel{b(\b u,p)}-tb(\b v,p),
   \end{equation}
   where we have used the fact that $\b u$ satisfies the mass equation \eqref{eq:ns.weak:mass} (proved above) and the non-dissipativity property \eqref{eq:c:zero} of $c$, respectively, in the cancellations.
   Dividing \eqref{eq:Minty.trick} by $t$, letting $t \to 0$, and using a dominated convergence argument made possible by the H\"older continuity \eqref{eq:ass:stress:holder.cont} of $\stress$ gives \eqref{eq:ns.weak:momentum}.

  Hence, $(\b u,p) \in \b U \times P$ is a solution of the weak formulation \eqref{eq:ns.weak}. 
  \\
  \\
  \textbf{Step 3.} \textit{Strong convergence of the velocity gradient and convergence of the boundary residual seminorm.}
  Passing to the upper limit inequality \eqref{eq:convergence:bound:1}, and using \eqref{eq:convergence:sigma.u} with $\b\phi$ replaced by $\b u$ together with the mass equation \eqref{eq:ns.weak:mass} and the non-dissipativity property \eqref{eq:c:zero} of $c$ to cancel the two rightmost terms in the resulting equation, we obtain
  \begin{equation} \label{eq:convergence:bound.Fatou}
    \begin{aligned}
      \limsup\limits_{h \to 0}\displaystyle\int_\Omega \stress(\cdot,\dgrads{k}{h}\bu u_h) : \dgrads{k}{h}\bu u_h \leq \int_\Omega \b f \cdot \b u = \int_\Omega \stress_{ \b u}: \GRADs\b u.
    \end{aligned}
  \end{equation}
  Thus, the H\"older monotonicity \eqref{eq:ass:stress:strong.mono} of $\stress$ together with the $(\sob,\frac{\sob}{\sob+2-\sobs},\frac{\sob+2-\sobs}{2-\sobs})$-H\"older inequality \eqref{eq:holder} yields,
  \begin{equation} \label{eq:convergence:limit.0}
    \begin{aligned}
      &\limsup\limits_{h \to 0}\left(\sigma_\cst{hm}\| \dgrads{k}{h} \bu u_h - \GRADs\b u\|_{L^{\sob}(\Omega)^{d\times d}}^{\sob+2-\sobs}  \left(\delta^\sob+\| \bu u_h \|_{\strain,\sob,h}^\sob+\|\GRADs\b u\|_{L^{\sob}(\Omega)^{d\times d}}^\sob\right)^\frac{\sobs-2}{\sob}\right) \\
      &\quad \lesssim \limsup\limits_{h \to 0}\int_\Omega (\stress(\cdot,\dgrads{k}{h}\bu u_h)-\stress(\cdot, \GRADs \b u)) :  (\dgrads{k}{h}\bu u_h - \GRADs \b u ) \leq 0,
    \end{aligned}
  \end{equation}
  where we concluded by strategically separating the terms in order to use inequality \eqref{eq:convergence:bound.Fatou} along with the weak convergences of $\stress(\cdot,\dgrads{k}{h}\bu u_h)\CV{h}{0}\stress_{ \b u}$ in $L^{r'}(\Omega)^{d\times d}$ and $\dgrads{k}{h}\bu u_h\CV{h}{0}\GRADs\b u$ in $L^r(\Omega)^{d\times d}$.
  Hence, since $(\| \bu u_h \|_{\strain,\sob,h})_{h \in \mathcal H}$ is bounded uniformly in $h$, $\dgrads{k}{h} \bu u_h \CV{h}{0}{} \GRADs\b u$ strongly in $L^\sob(\Omega)^{d\times d}$ up to a subsequence.
  Now, using the H\"older continuity \eqref{eq:ass:stress:holder.cont} of $\stress$ together with the $(\sob';\frac{\sob}{\sob-\sobs},\frac{\sob}{\sobs-1})$-H\"older inequality, we get
  \[
  \begin{aligned}
    &\|\stress(\cdot,\dgrads{k}{h} \bu u_h)-\stress(\cdot,\GRADs\b u)\|_{L^{\sob'}(\Omega)^{d\times d}} \\
    &\quad \lesssim \sigma_\cst{hc} \left(\delta^\sob+\|\dgrads{k}{h} \bu u_h\|_{L^{\sob}(\Omega)^{d\times d}}^\sob+\|\GRADs\b u\|_{L^{\sob}(\Omega)^{d\times d}}^\sob\right)^\frac{\sob-\sobs}{\sob}\|\dgrads{k}{h} \bu u_h-\GRADs\b u\|_{L^{\sob}(\Omega)^{d\times d}}^{\sobs-1} \CV{h}{0} 0. \\
  \end{aligned}
  \]
  Thus, $\stress(\cdot,\dgrads{k}{h} \bu u_h) \CV{h}{0} \stress(\cdot,\GRADs\b u)$ strongly in $L^{\sob'}(\Omega)^{d\times d}$ up to a subsequence.
  In particular, we have
  \begin{equation}\label{eq:convergence:sh.uh.0}
    \sigma_\cst{hm}\left(\delta^r+|\bu u_h|_{\sob,h}^\sob\right)^\frac{\sobs-2}{\sob}|\bu u_h|_{\sob,h}^{\sob+2-\sobs}
    \lesssim \cst{s}_h(\bu u_h,\bu u_h) = \int_\Omega \b f \cdot \b u_h  - \int_\Omega \stress(\cdot,\dgrads{k}{h}\bu u_h):\dgrads{k}{h}\bu u_h \CV{h}{0} 0,
  \end{equation}
  thanks to the H\"older monotonicity \eqref{eq:sh:strong.mono} of $\cst{s}_h$ and the definition \eqref{eq:ah} of $\cst{a}_h$. Hence, $|\bu u_h|_{\sob,h} \CV{h}{0} 0$.
  \\[2pt]
  \textbf{Step 4.} \textit{Strong convergence of the pressure.}
  Set $p_{h,\Omega} \coloneqq \int_\Omega |p_h|^{\sob'-2}p_h$.
  Reasoning as in \cite[Eq. (50)]{Botti.Castanon-Quiroz.ea:20}, we infer the existence of $\b v_{p_h} \in W^{1,\sob}_0(\Omega)^d$ such that
  \begin{equation}\label{eq:convergence:vph}
    \div \b v_{p_h} = |p_h|^{\sob'-2}p_h-p_{h,\Omega}\quad \text{ and } \quad \Vert \b v_{p_h} \Vert_{W^{1,\sob}(\Omega)^d} \lesssim \Vert p_h \Vert_{L^{\sob'}(\Omega)}^{\sob'-1}. 
  \end{equation}
  Furthermore, letting $\bu{\hat v}_{p_h} \coloneq \bI{h}{k} \b v_{p_h}$, and using the boundedness \eqref{eq:I:boundedness} of $\bI{h}{k}$ together with \eqref{eq:norm.equiv}, we get
  \[
  \Vert \bu{\hat v}_{p_h}\Vert_{\strain,\sob,h} \lesssim | \b v_{p_h} |_{W^{1,\sob}(\Omega)^d} \lesssim \Vert p_h \Vert_{L^{\sob'}(\Omega)}^{\sob'-1}.
  \]
  Since $(\|p_h\|_{L^{\sob'}(\Omega)})_{h \in \mathcal H}$ is bounded uniformly in $h$, by \cite[Theorem 9.29]{Di-Pietro.Droniou:20} there exists $\b v_p \in \b U$ such that $\b{\hat v}_{p_{h}} \CV{h}{0} \b v_p$ strongly in $L^{[1,\sob^*)}(\Omega)^d$ and $\dgrad{k}{h}\bu{\hat v}_{p_{h}} \CV{h}{0} \GRAD \b v_p$ weakly in $L^{\sob}(\Omega)^{d\times d}$.
    Therefore, recalling that $\stress(\cdot,\dgrads{k}{h} \bu u_h) \CV{h}{0} \stress(\cdot,\GRADs\b u)$ strongly in $L^{\sob'}(\Omega)^{d\times d}$, we get 
    \begin{equation}\label{eq:convergence:vph.sigma}
      \int_\Omega \stress(\cdot,\dgrads{k}{h} \bu u_{h}):\dgrads{k}{h}\bu{\hat v}_{p_{h}} \CV{h}{0} \int_\Omega  \stress(\cdot,\GRADs\b u):\GRADs\b v_p.
    \end{equation}     
    Using the H\"older continuity \eqref{eq:sh:holder.cont} of $\cst{s}_h$ with $\bu w_h = \bu 0$, we next infer
    \begin{equation}\label{eq:convergence:vph.sh}
      \begin{aligned}
        |\cst{s}_h(\bu u_h,\bu{\hat v}_{p_h})|
        &\lesssim \sigma_\cst{hc}\left( \delta^r+  |\bu u_h|_{\sob,h}^\sob\right)^\frac{\sob-\sobs}{\sob}|\bu u_h|_{\sob,h}^{\sobs-1}|\bu{\hat v}_{p_h}|_{\sob,h} \CV{h}{0} 0,
      \end{aligned}
    \end{equation}
    since $|\bu u_h|_{\sob,h} \CV{h}{0} 0$ by \eqref{eq:convergence:sh.uh.0} and noticing that the sequence $(|\bu{\hat v}_{p_h}|_{r,h})_{h\in\mathcal H}$ is bounded uniformly in $h$ by \eqref{eq:bound.res:stability.boundedness}.
    Thus, recalling the definition \eqref{eq:ah} of $\cst{a}_h$, \eqref{eq:convergence:vph.sigma} and \eqref{eq:convergence:vph.sh} yields
    \begin{equation}\label{eq:convergence:vph.ah.a}
      \cst{a}_h(\bu u_h,\bu{\hat v}_{p_h}) \CV{h}{0} a(\b u,\b v_p).
    \end{equation}
 Furthermore, replacing $\b\phi$ by $\b v_{p_{h}}$ in the reasoning that gives \eqref{eq:ch:sequential.consistency:T12}--\eqref{eq:ch:sequential.consistency:T3} after recalling that $\dgrad{k}{h}\bu{\hat v}_{p_{h}} \CV{h}{0} \GRAD \b v_p$ weakly in $L^{\sob}(\Omega)^{d\times d}$ and that $\b{\hat v}_{p_{h}} \CV{h}{0} \b v_p$ strongly in $L^{sr'}(\Omega)^d$ (since $sr' < r^*$ by \eqref{eq:intervals.s:convergence}), we infer
    \begin{equation}\label{eq:convergence:vph.ch.c}
      \cst{c}_h(\bu u_h,\bu{\hat v}_{p_h}) \CV{h}{0} c(\b u,\b v_p).
    \end{equation}
    Moreover, since $\b f \in L^{\sob'}(\Omega)^d$ and $\b{\hat v}_{p_{h}} \CV{h}{0}{} \b v_p$ strongly in $L^\sob(\Omega)^d$, we get
    \begin{equation}\label{eq:convergence:vph.f}
      \int_\Omega \b f \cdot \b{\hat v}_{p_h} \CV{h}{0}  \int_\Omega \b f \cdot \b v_p.
    \end{equation}
    Hence, using the fact that $p_h$ has zero mean value over $\Omega$ together with the equality in \eqref{eq:convergence:vph}, the definition \eqref{eq:weak:ab} of $b$, the Fortin property \eqref{eq:bh:fortin} of $\cst{b}_h$, and equality \eqref{eq:ns.discrete:momentum} yields
    \begin{equation}\label{eq:convergence:ph.p}
      \begin{aligned}
        &\int_\Omega (|p_h|^{\sob'-2}p_h)p_h = \int_\Omega (\div \b v_{p_h})p_h  = -b(\b v_{p_h},p_h)  = -\cst{b}_h(\bu{\hat v}_{p_h},p_h)  \\
        &\qquad = \cst{a}_h(\bu u_h,\bu{\hat v}_{p_h}) + \cst{c}_h(\bu u_h,\bu{\hat v}_{p_h}) - \int_\Omega \b f \cdot \b{\hat v}_{p_h} \CV{h}{0} a(\b u,\b v_p) + c(\b u,\b v_p) - \int_\Omega \b f \cdot \b v_p\\
        &\qquad = -b(\b v_p,p) = \int_\Omega (\div\b v_p)p,
      \end{aligned}
    \end{equation}
    where we have
    used the discrete momentum equation \eqref{eq:ns.discrete:momentum} and applied the limits \eqref{eq:convergence:vph.ah.a}, \eqref{eq:convergence:vph.ch.c}, and \eqref{eq:convergence:vph.f} in the second line,
    used the continuous momentum equation \eqref{eq:ns.weak:momentum} to pass to the third line,
    and invoked the mass equation \eqref{eq:ns.discrete:mass} to conclude.
    Meanwhile, since $(\||p_h|^{\sob'-2}p_h\|_{L^\sob(\Omega)})_{h \in \mathcal H} = (\|p_h\|_{L^{\sob'}(\Omega)})_{h \in \mathcal H}$ is bounded uniformly in $h$, $|p_{h}|^{\sob'-2}p_{h}$ weakly converges to $p_\sob \in L^\sob(\Omega)$ up to a subsequence. In particular, $p_{{h},\Omega} \CV{h}{0}{} p_{\sob,\Omega} \coloneq \int_\Omega p_\sob$, and we deduce that $\div \b v_{p_{h}}   \CV{h}{0}{}  p_\sob-p_{\sob,\Omega} $ weakly in $L^\sob(\Omega)$ and, by uniqueness of the limit in the distributional sense, that $\div \b v_p = p_\sob-p_{\sob,\Omega}$.
    Therefore, using \eqref{eq:convergence:ph.p} together with the fact that $p$ has zero mean value over $\Omega$, we infer
    \begin{equation}\label{eq:convergence:ph.p:1}
      \begin{aligned}
        \int_\Omega (|p_h|^{\sob'-2}p_h)p_h \CV{h}{0} \int_\Omega p_rp.
      \end{aligned}
    \end{equation}
    Moreover, using the H\"older monotonicity property \eqref{eq:s.power:strong.mono} (with $(n,x,y,\conv) = (1,p_h,p,\sob')$) together with the $(\sob',\frac{\sob'}{\sob'+2-\widetilde{\sob'}},\frac{\sob'+2-\widetilde{\sob'}}{2-\widetilde{\sob'}})$-H\"older inequality, and passing to the limit with \eqref{eq:convergence:ph.p:1}, we obtain
    \begin{equation}\label{eq:convergence:ph.p:2}
      \begin{aligned}
        &\| p_h - p\|_{L^{\sob}(\Omega)}^{r'+2-\widetilde{\sob'}} \left(\| p_h \|_{L^{\sob'}(\Omega)}^{\sob'}+\|p\|_{L^{\sob'}(\Omega)}^{\sob'}\right)^\frac{\widetilde{\sob'}-2}{\sob'}  \lesssim \int_\Omega (|p_h|^{\sob'-2}p_h-|p|^{\sob'-2}p)(p_h-p) \CV{h}{0} 0.
      \end{aligned}
    \end{equation}
    Hence, $p_h \CV{h}{0}{} p$ strongly in $L^{\sob}(\Omega)$ up to a subsequence.
\end{proof}

\subsection{Error estimate}\label{sec:error.estimate}

\begin{proof}[Proof of Theorem \ref{thm:error.estimate}]\label{proof:thm:error.estimate}
  Let $(\bu e_h, \epsilon_h) \coloneqq (\bu u_h - \bu{\hat u}_h,p_h - \hat p_h) \in \bdU{h,0}{k} \times \dP{h}{k}$ where $\bu{\hat u}_h \coloneqq\bI{h}{k}\b u$ and $\hat p_h \coloneqq \proj{h}{k} p$.
  \smallskip\\
  \textbf{Step 1.} \emph{Consistency error.} Let $\mathcal E_h : \bdU{h,0}{k} \to \R$ be the consistency error linear form such that, for all $\bu v_h \in \bdU{h,0}{k}$,
  \begin{equation}\label{eq:Eh}
    \mathcal E_h(\bu v_h) \coloneqq \int_\Omega \b f \cdot \b v_h - \cst{a}_h(\bu{\hat u}_h,\bu v_h)- \cst{c}_h(\bu{\hat u}_h,\bu v_h)-\cst{b}_h(\bu v_h,{\hat p}_h).
  \end{equation}
  Using in \eqref{eq:Eh} the fact that $\b f = -\DIV \stress(\cdot,\GRADs \b u)+\NAV{\b u}{\convec(\cdot,\b u)}+\GRAD p$ almost everywhere in $\Omega$ together with the consistency properties \eqref{eq:ah:consistency} of $\cst{a}_h$,  \eqref{eq:bh:consistency} of $\cst{b}_h$, and \eqref{eq:ch:consistency} of $\cst{c}_h$ (since $\conv \le \frac{\sob^*}{\sob'}$), we obtain
  \begin{equation}\label{eq:error.estimate:step1:eh0}
    \$\coloneq
    \sup\limits_{\bu v_h \in \bdU{h,0}{k},\| \bu v_h \|_{\strain,\sob,h} = 1}\mathcal E_h(\bu v_h)
    \lesssim
    h^{(k+1)(\sob-1)}\min\left(\zeta_h(\b u);1\right)^{2-\sob} \mathcal N_2+h^{k+1} \mathcal N_3.
  \end{equation}
  \\
  \textbf{Step 2.} \emph{Error estimate for the velocity.}  
  Replacing $\cst{a}_h$ by $\cst{a}_h+\cst{c}_h$ in the reasoning of \cite[Eq. (70)]{Botti.Castanon-Quiroz.ea:20} yields,
  \begin{equation}\label{eq:error.estimate:step2:eh0}
    \hspace{-1mm}\begin{aligned}
       \mathcal E_h(\bu e_h)
      &=
      \cst{a}_h(\bu u_h,\bu e_h)-\cst{a}_h(\bu{\hat u}_h,\bu e_h)+\cst{c}_h(\bu u_h,\bu e_h)-\cst{c}_h(\bu{\hat u}_h,\bu e_h) \\
      &\ge \left(
        C_\cst{da}\sigma_\cst{hm}-C_{\cst{dc},\sob}\chi_\cst{hc}\left(\delta^\sob\!+
      \| \bu u_h \|_{\strain,\sob,h}^\sob+\|\bu{\hat u}_h\|_{\strain,\sob,h}^\sob
      \right)^\frac{\conv+1-\sob}{\sob}\right)\\
      &\qquad \times 
      \left(
      \delta^\sob\!+\| \bu u_h \|_{\strain,\sob,h}^\sob+\|\bu{\hat u}_h\|_{\strain,\sob,h}^\sob
      \right)^\frac{\sob-2}{\sob}\| \bu e_h \|_{\strain,\sob,h}^{2}\\
      &\gtrsim \left(C_\cst{da}\sigma_\cst{hm}-C_{\cst{dc},\sob}\chi_\cst{hc}\mathcal N_1^\frac{\conv+1-\sob}{r}\right)\mathcal N_1^\frac{\sob-2}{r}\| \bu e_h \|_{\strain,\sob,h}^{2}\\
            &\gtrsim \sigma_\cst{hm}\mathcal N_1^\frac{r-2}{r}\| \bu e_h \|_{\strain,\sob,h}^{2},
    \end{aligned}
  \end{equation}
where, noting that $\sobs = \sob$ and $\convs = 2$ since $\sob \le 2 \le \conv$, we have used the H\"older monotonicity \eqref{eq:ah:strong.mono} of $\cst{a}_h$ together with the H\"older continuity \eqref{eq:ch:holder.cont} of $\cst{c}_h$ (since $\conv \le \frac{\sob^*}{\sob'}$) in the second line, the a priori bound \eqref{eq:discrete.well-posedness:bounds:uh} on the discrete solution along with the boundedness \eqref{eq:I:boundedness} of the global interpolator and the a priori bound \eqref{eq:well-posedness:bound:u} on the continuous solution in the penultimate line, and the small data assumption \eqref{eq:small.f} to conclude.
   Hence, the fact that $\mathcal E_h(\bu e_h) \le \$\| \bu e_h \|_{\strain,\sob,h}$ together with \eqref{eq:error.estimate:step1:eh0} yields \eqref{eq:error.estimate:velocity}.	
   \medskip\\
   \textbf{Step 3.} \emph{Error estimate for the pressure.}	
Using the H\"older continuity \eqref{eq:ch:holder.cont} of $\cst{c}_h$ together with the same raisoning giving the penultimate line of \eqref{eq:error.estimate:step2:eh0} we infer, for all $\bu v_h \in \bdU{h,0}{k}$,
  \begin{equation}\label{eq:error.estimate:step3:ah}
   \begin{aligned}
     \left|\cst{c}_h(\bu{\hat u}_h,\bu v_h)-\cst{c}_h(\bu u_h,\bu v_h)\right|
  &\lesssim \chi_\cst{hc}\mathcal N_1^\frac{s-1}{r}\| \bu e_h \|_{\strain,\sob,h}\| \bu v_h \|_{\strain,\sob,h}.
  \end{aligned}
  \end{equation}
   Thus, replacing $\cst{a}_h$ with $\cst{a}_h+\cst{c}_h$ in the reasoning of \cite[Eq. (72)]{Botti.Castanon-Quiroz.ea:20}, we obtain
  \begin{equation}\label{eq:error.estimate:step3:epsilonh}
    \begin{aligned}
      \| \epsilon_h \|_{L^{\sob'}(\Omega)} 
      &\lesssim \sup\limits_{\bu v_h \in \bdU{h,0}{k},\| \bu v_h \|_{\strain,\sob,h} = 1} \left(\mathcal E_h(\bu v_h)+\cst{a}_h(\bu{\hat u}_h,\bu v_h)-\cst{a}_h(\bu u_h,\bu v_h)+\cst{c}_h(\bu{\hat u}_h,\bu v_h)-\cst{c}_h(\bu u_h,\bu v_h)\right)\\
      &\lesssim \$+\sigma_\cst{hc}\| \bu e_h \|_{\strain,\sob,h}^{\sob-1}+\chi_\cst{hc}\mathcal N_1^\frac{s-1}{r}\| \bu e_h \|_{\strain,\sob,h}, \\
    \end{aligned}
  \end{equation}
  where we have used the H\"older continuity \eqref{eq:ah:holder.cont} of $\cst{a}_h$ together with \eqref{eq:error.estimate:step3:ah} to conclude. Finally, the bounds \eqref{eq:error.estimate:step1:eh0} and \eqref{eq:error.estimate:velocity} (proved in Step 2) give \eqref{eq:error.estimate:pressure}.
\end{proof}

%------------------------------------------------------------------------------------%

\section*{Acknowledgements}

Daniele Di Pietro acknowledges the partial support of \emph{Agence Nationale de la Recherche} through grant NEMESIS (ANR-20-MRS2-0004).

%------------------------------------------------------------------------------------%

\raggedright
\printbibliography

\end{document}